\documentclass{amsart}

\usepackage[latin1]{inputenc}
\usepackage{amssymb,amsmath,amsthm}
\usepackage{amsfonts}
\usepackage{ifthen}
\usepackage[neverdecrease,pointedenum]{paralist}

\newcommand{\vect}[1]{\ensuremath{\mathbf{#1}}}
\newcommand{\card}[1]{\ensuremath{\lvert{#1}\rvert}}
\newcommand{\cl}[1]{\ensuremath{\mathcal{#1}}}

\newcommand\fBold[1]{\textbf #1\relax}

\newcommand{\nset}[1]{\ensuremath{[{#1}]}}

\newcommand{\couples}[1][n]{\ensuremath{\binom{#1}{2}}} % the set of 2-element subsets of an n-element set
\newcommand{\symm}[1]{\ensuremath{\Sigma_{#1}}} % symmetric group
\newcommand{\alt}[1]{\ensuremath{A_{#1}}} % alternating group

\newcommand{\pos}[2]{\overset{\text{\makebox[0mm][c]{$\overset{#1}{\downarrow}$}}}{\phantom{\makebox[0mm]{f}}{#2}}}

\DeclareMathOperator{\deck}{deck}                % deck
\DeclareMathOperator{\Inv}{Inv}                  % invariance group
\DeclareMathOperator{\supp}{supp}                % supp
\DeclareMathOperator{\oddsupp}{oddsupp}          % oddsupp
\DeclareMathOperator{\ofo}{ofo}                  % order of first occurrence
\DeclareMathOperator{\msupp}{ms}                 % multiset support
\DeclareMathOperator{\set}{set}                  % underlying set of a multiset
\DeclareMathOperator{\range}{Im}                 % range
\DeclareMathOperator{\pr}{pr}                    % projection
\DeclareMathOperator{\id}{id}                    % identity permutation

\theoremstyle{plain}
\newtheorem{theorem}{Theorem}[section]
\newtheorem{proposition}[theorem]{Proposition}
\newtheorem{lemma}[theorem]{Lemma}
\newtheorem{corollary}[theorem]{Corollary}

\theoremstyle{definition}
\newtheorem{definition}[theorem]{Definition}
\newtheorem{fact}[theorem]{Fact}
\newtheorem{example}[theorem]{Example}
\newtheorem{question}[theorem]{Question}

\theoremstyle{remark}
\newtheorem{remark}[theorem]{Remark}
\newtheorem{claim}{Claim}
\numberwithin{claim}{theorem}
\newenvironment{pfclaim}[1][Proof]{\begin{trivlist}
\item[\hskip \labelsep {\textit{{#1}.}}]} {{\footnotesize $\blacksquare$}\end{trivlist}}

\begin{document}
\title[On the reconstructibility of totally symmetric functions]{On the reconstructibility of totally symmetric functions and of other functions with a unique identification minor}
\author{Erkko Lehtonen}
\address{University of Luxembourg \\
Computer Science and Communications Research Unit \\
6, rue Richard Coudenhove-Kalergi \\
L--1359 Luxembourg \\
Luxembourg}
\email{erkko.lehtonen@uni.lu}
\date{October 11, 2012}
\begin{abstract}
We investigate the problem whether a function of several arguments can be reconstructed from its identification minors. We focus on functions with a unique identification minor, and we establish some positive and negative results on the reconstruction problem. In particular, we show that totally symmetric functions (of sufficiently large arity) are reconstructible and the class of functions weakly determined by the order of first occurrence (of sufficiently large arity) is weakly reconstructible.
\end{abstract}

\maketitle

%%%%%%%%%%%%%%%%%%%%%%%%%%%%%%%%%%%%%%%%%%%%%%%%%%

\section{Introduction}
\label{sec:intro}

Reconstruction problems have received great attention over the past decades. Perhaps the most famous reconstruction problem is the following: Can every graph with at least three vertices be reconstructed, up to isomorphism, from its collection of one-vertex-deleted subgraphs?
It was conjectured by Kelly~\cite{Kelly1942} (see also Ulam's problem book~\cite{Ulam}) that the answer is positive.
The conjecture has been verified by computer for graphs with at most 11 vertices (McKay~\cite{McKay}), and it has been proved for several infinite classes of graphs, such as trees (Kelly~\cite{Kelly1957}), regular graphs, disconnected graphs, and so on. We refer the reader to survey articles, textbooks, and reference books~\cite{Babai,Bondy,BonHem,BonMur,Harary1974,Manvel,NashWilliams} for further details and references.

The reconstruction problem stated above can be varied in several ways. For example, we might consider the collection of subgraphs formed by deleting edges instead of vertices (see Harary~\cite{Harary1964} and Ellingham~\cite{Ellingham}), or we could consider directed graphs (infinite nonreconstructible families have been constructed by Stockmeyer~\cite{Stockmeyer}) or hypergraphs (infinite nonreconstructible families have been constructed by Kocay~\cite{Kocay}).
Reconstruction problems have been formulated also for other kinds of mathematical objects, such as relations (see Fra\"iss\'e~\cite{Fraisse}), posets (see the survey by Rampon~\cite{Rampon}), matrices (see Manvel and Stockmeyer~\cite{ManSto}), matroids (see Brylawski~\cite{Brylawski1974,Brylawski1975}), and integer partitions (see Monks~\cite{Monks}).

In this paper we formulate a reconstruction problem for functions of several arguments. We shall take as the derived objects of a function $f \colon A^n \to B$ its identification minors, i.e., functions obtained from $f$ by identifying a pair of its arguments. The notion of isomorphism is based on the equivalence relation that relates two $n$-ary functions if and only if each one can be obtained from the other by permutation of arguments. The reconstruction problem can thus be stated as follows: Can a function $f \colon A^n \to B$ be reconstructed, up to equivalence, from its identification minors? This study is a first step towards answering this question. We focus on functions with a unique identification minor, and we obtain several results, both positive and negative, about the reconstructibility of such functions.

Identification minors and equivalence of functions are related in an essential way to a quasiordering of functions, the so-called minor relation defined as follows: a function $f \colon A^n \to B$ is a minor of another function $g \colon A^m \to B$, if there exists a map $\sigma \colon \{1, \dots, m\} \to \{1, \dots, n\}$ such that $f(a_1, \dots, a_n) = g(a_{\sigma(1)}, \dots, a_{\sigma(m)})$ for all $(a_1, \dots, a_n) \in A^n$. The reconstruction problem is thus a deep and intriguing question about the structure of the minor ordering of functions.
Minors have been widely studied in the literature, under different names.
Minors are called
``identification minors'' by Ekin, Foldes, Hammer, and Hellerstein~\cite{EFHH},
``$\mathcal{I}$-minors'' (where $\mathcal{I}$ stands for the set containing just the identity function on $A$) by Pippenger~\cite{Pippenger},
``subfunctions'' by Zverovich~\cite{Zverovich},
``functions obtained by simple variable substitution'' by Couceiro and Foldes~\cite{CouFol2005},
``$\mathcal{J}$-subfunctions'' (where $\mathcal{J}$ stands for the clone of projections on $A$) by Lehtonen~\cite{Lehtonen2006},
``$\mathcal{P}_A$-minors'' (where $\mathcal{P}_A$ stands for the clone of projections on $A$) by Lehtonen and Szendrei~\cite{LehSze2009},
and
``simple minors'' by Couceiro and Lehtonen~\cite{CouLeh2009}.

This paper is organised as follows.
In Section~\ref{sec:preliminaries}, we provide basic definitions on functions and identification minors that will be needed in the sequel, and we recall some facts about permutation groups.
In Section~\ref{sec:reconstruction}, we formulate a reconstruction problem for functions and identification minors, and we recall the usual terminology of reconstruction problems in the current setting. We present some examples of reconstructible and recognizable classes of functions, as well as examples of nonreconstructible functions and reconstructible parameters. One of our first results is that functions determined by $\supp$ or $\oddsupp$ (of sufficiently large arity) are reconstructible.
In Section~\ref{sec:unique}, we introduce functions with a unique identification minor; these include functions that are $2$-set-transitive or determined by the order of first occurrence.
In Section~\ref{sec:symmetric}, we establish that totally symmetric functions (of sufficiently large arity) are reconstructible.
The remainder of the paper deals with functions (weakly) determined by the order of first occurrence.
In Section~\ref{sec:prsupp} we show that functions determined by $(\pr, \supp)$ (of sufficiently large arity) are reconstructible. These functions form a subclass of the class of functions weakly determined by the order of first occurrence. In the particular case of functions defined on a two-element set, these classes coincide. Consequently, functions defined on a two-element set that are weakly determined by the order of first occurrence are reconstructible.
In Section~\ref{sec:eqeq}, we investigate conditions under which functions determined by the order of first occurrence that are equivalent are actually equal.
This will find applications in Section~\ref{sec:results}, where we show that the class of functions weakly determined by the order of first occurrence (of sufficiently large arity) is weakly reconstructible.
Occasionally throughout the paper, we present open problems to indicate directions for future work.

%%%%%%%%%%%%%%%%%%%%%%%%%%%%%%%%%%%%%%%%%%%%%%%%%%

\section{Preliminaries}
\label{sec:preliminaries}

\subsection{General}
\label{sec:preliminaries:general}

Throughout this paper, we let $k$, $m$ and $n$ be positive integers, and we let $A$ and $B$ be arbitrary sets with at least two elements. For reasons that will become clear in Remark~\ref{rem:finiteA}, we may assume that these sets are finite, and $k$ usually stands for the cardinality of $A$.
For a positive integer $n$, we denote the set $\{1, \dots, n\}$ by $\nset{n}$. We denote the set of all $2$-element subsets of $\nset{n}$ by $\couples$. We denote tuples by bold-face letters and components of a tuple by the corresponding italic letters, e.g., $\vect{a} = (a_1, \dots, a_n)$. We reserve the symbol $\vect{k}$ to denote the $k$-tuple $(1, 2, \dots, k) \in \nset{k}^k$.

Let $\vect{a} \in A^n$, and let $\sigma \colon \nset{m} \to \nset{n}$. We will write $\vect{a} \sigma$ to denote the $m$-tuple $(a_{\sigma(1)}, \dots, a_{\sigma(m)})$. Note that if $A = \nset{k}$, then every $n$-tuple over $A$ is of the form $\vect{k} \sigma$ for some $\sigma \colon \nset{n} \to \nset{k}$. Since the $n$-tuple $\vect{a}$ can be formally seen as the map $\vect{a} \colon \nset{n} \to A$, $i \mapsto a_i$, the $m$-tuple $\vect{a} \sigma$ is just the composite map $\vect{a} \circ \sigma \colon \nset{m} \to A$. We are going to use this notation in a recurrent way, so, for instance, if $\tau \colon \nset{\ell} \to \nset{m}$, then we write $\vect{a} \sigma \tau$ for the $\ell$-tuple $(a_{\sigma(\tau(1))}, \dots, a_{\sigma(\tau(\ell))})$, or, equivalently, for the composite map $\vect{a} \circ \sigma \circ \tau \colon \nset{\ell} \to A$. It is worth stressing here that we always compose functions and permutations right-to-left, so $\sigma \tau$ or $\sigma \circ \tau$ means ``do $\tau$ first, then do $\sigma$''.

Let
\begin{equation}
\label{eq:Anneq}
A^n_{\neq} := \{(a_1, \dots, a_n) \in A^n : \text{$a_1$, $a_2$, \dots, $a_n$ are pairwise distinct}\}.
\end{equation}
Note that if $n \geq \card{A}$, then $A^n_{\neq} = \emptyset$.

\subsection{Multisets}

Let $\mathbb{N} := \{0, 1, 2, \dots\}$.
A \emph{finite multiset} $S$ on a set $X$ is a map $\mathbf{1}_S \colon X \to \mathbb{N}$, called a \emph{multiplicity function,} such that the set $\{x \in X : \mathbf{1}_S(x) \neq 0\}$ is finite. Then the sum $\sum_{x \in X} \mathbf{1}_S(x)$ is a well-defined natural number, and it is called the \emph{cardinality} of $S$. For each $x \in X$, the number $\mathbf{1}_S(x)$ is called the \emph{multiplicity} of $x$ in $S$. We denote the set of all finite multisets on $X$ by $\mathcal{M}(X)$, and we denote the set of all finite multisets of cardinality $n$ on $X$ by $\mathcal{M}_n(X)$.

We may represent a finite multiset $S$ as a list enclosed in set brackets, where each element $x \in X$ occurs $\mathbf{1}_S(x)$ times, e.g., $\{0, 0, 0, 1, 1, 2\}$. Also, if $(a_i)_{i \in I}$ is a finite indexed family of elements of $X$, then we will write $\{a_i : i \in I\}$ to denote the multiset in which the multiplicity of each $x \in X$ equals $\card{\{i \in I : a_i = x\}}$. While this notation is similar to that used for sets, it will always be clear from the context whether we refer to a set or to a multiset.

Let $S$ and $T$ be finite multisets over $X$. The \emph{multiset sum} $S \uplus T$ and the \emph{difference} $S \setminus T$ of $S$ and $T$ are defined by the multiplicity functions
\begin{gather*}
\mathbf{1}_{S \uplus T}(x) = \mathbf{1}_S(x) + \mathbf{1}_T(x), \\
\mathbf{1}_{S \setminus T}(x) = \max(\mathbf{1}_S(x) - \mathbf{1}_T(x), 0).
\end{gather*}

If $S$ is a multiset on $X$, then we write $\set(S)$ to denote the set $\{x \in X : \mathbf{1}_S(x) \neq 0\}$, called the \emph{underlying set} of $S$.

\subsection{Functions of several arguments and identification minors}

A \emph{function} (\emph{of several arguments}) from $A$ to $B$ is a map $f \colon A^n \to B$ for some positive integer $n$, called the \emph{arity} of $f$. Functions of several arguments from $A$ to $A$ are called \emph{operations} on $A$. Operations on $\{0, 1\}$ are called \emph{Boolean functions.} We denote the set of all $n$-ary functions from $A$ to $B$ by $\cl{F}_{AB}^{(n)}$, and we denote the set of all functions from $A$ to $B$ of any finite arity by $\cl{F}_{AB}$; in other words, $\cl{F}_{AB}^{(n)} = B^{A^n}$ and $\cl{F}_{AB} = \bigcup_{n \geq 1} \cl{F}_{AB}^{(n)}$.
For any $\cl{C} \subseteq \cl{F}_{AB}$, we denote $\cl{C}^{(n)} := \cl{C} \cap \cl{F}_{AB}^{(n)}$; this is called the \emph{$n$-ary part} of $\cl{C}$.

For integers $n$ and $i$ such that $1 \leq i \leq n$, the $i$-th $n$-ary \emph{projection} on $A$ is the operation $\pr_i^{(n)} \colon A^n \to A$, $\vect{a} \mapsto a_i$ for all $\vect{a} \in A^n$. 

Let $f \colon A^n \to B$. For $i \in \nset{n}$, the $i$-th argument of $f$ is \emph{essential,} or $f$ \emph{depends} on the $i$-th argument, if there exist elements $\vect{a}, \vect{b} \in A^n$ such that $a_j = b_j$ for all $j \in \nset{n} \setminus \{i\}$ and $f(\vect{a}) \neq f(\vect{b})$. Arguments that are not essential are \emph{inessential.}

We say that a function $f \colon A^n \to B$ is a \emph{minor} of another function $g \colon A^m \to B$, and we write $f \leq g$, if there exists a map $\sigma \colon \nset{m} \to \nset{n}$ such that $f(\vect{a}) = g(\vect{a} \sigma)$ for all $\vect{a} \in A^m$.
The minor relation $\leq$ is a quasiorder on $\cl{F}_{AB}$, and, as for all quasiorders, it induces an equivalence relation on $\cl{F}_{AB}$ by the following rule: $f \equiv g$ if and only if $f \leq g$ and $g \leq f$. We say that $f$ and $g$ are \emph{equivalent} if $f \equiv g$. Furthermore, $\leq$ induces a partial order on the quotient $\cl{F}_{AB} / {\equiv}$. (Informally speaking, $f$ is a minor of $g$, if $f$ can be obtained from $g$ by permutation of arguments, addition of inessential arguments, deletion of inessential arguments, and identification of arguments. If $f$ and $g$ are equivalent, then each one can be obtained from the other by permutation of arguments, addition of inessential arguments, and deletion of inessential arguments.) We will often distinguish between functions only up to equivalence, i.e., we are dealing with the $\equiv$-classes of functions. We denote the $\equiv$-class of $f$ by $f / {\equiv}$.
Note that equivalent functions have the same number of essential arguments and every nonconstant function is equivalent to a function with no inessential arguments.
Note also in particular that if $f, g \colon A^n \to B$, then $f \equiv g$ if and only if there exists a bijection $\sigma \colon \nset{n} \to \nset{n}$ such that $f(\vect{a}) = g(\vect{a} \sigma)$ for all $\vect{a} \in A^n$.

A set $\cl{C} \subseteq \cl{F}_{AB}$ of functions is \emph{closed under formation of minors} if for all $f, g \in \cl{F}_{AB}$, the conditions $f \leq g$ and $g \in \cl{C}$ together imply $f \in \cl{C}$. All clones (see \cite{DenWis,Lau,Szendrei}) are closed under formation of minors. A characterization of sets of functions closed under formation of minors in terms of a Galois connection between functions and so-called constraints was presented by Pippenger~\cite{Pippenger} for functions with finite domains, and this result was later extended to functions with arbitrary domains by Couceiro and Foldes~\cite{CouFol2005}.

Of particular interest to us are the following minors. Let $n \geq 2$, and let $f \colon A^n \to B$. For each $I \in \couples$, we define the function $f_I \colon A^{n-1} \to B$ by the rule
$f_I(\vect{a}) = f(\vect{a} \delta_I)$ for all $\vect{a} \in A^{n-1}$,
where $\delta_I \colon \nset{n} \to \nset{n - 1}$ is defined as
\begin{equation}
\label{eq:deltaI}
\delta_I(i) =
\begin{cases}
i, & \text{if $i < \max I$,} \\
\min I, & \text{if $i = \max I$,} \\
i - 1, & \text{if $i > \max I$.}
\end{cases}
\end{equation}
In other words, if $I = \{i, j\}$ with $i < j$, then
\[
f_I(a_1, \dots, a_{n-1}) =
f(a_1, \dots, a_{j-1}, a_i, a_j, \dots, a_{n-1}).
\]
Note that $a_i$ occurs twice on the right side of the above equality: both at the $i$-th and at the $j$-th position. We will refer to the function $f_I$ as an \emph{identification minor} of $f$. This name is motivated by the fact that $f_I$ is obtained from $f$ by identifying the arguments indexed by the couple $I$.

\begin{example}
Let $f \colon \mathbb{R}^4 \to \mathbb{R}$ be given by $f(x_1, x_2, x_3, x_4) = x_1^2 x_2 - x_1 x_3 + 2 x_3 x_4$. The identification minors of $f$ are
\begin{align*}
f_{\{1,2\}}(x_1, x_2, x_3) &= x_1^3 - x_1 x_2 + 2 x_2 x_3, &
  f_{\{2,3\}}(x_1, x_2, x_3) &= x_1^2 x_2 - x_1 x_2 + 2 x_2 x_3, \\
f_{\{1,3\}}(x_1, x_2, x_3) &= x_1^2 x_2 - x_1^2 + 2 x_1 x_3, &
  f_{\{2,4\}}(x_1, x_2, x_3) &= x_1^2 x_2 - x_1 x_3 + 2 x_2 x_3, \\
f_{\{1,4\}}(x_1, x_2, x_3) &= x_1^2 x_2 + x_1 x_3, &
  f_{\{3,4\}}(x_1, x_2, x_3) &= x_1^2 x_2 - x_1 x_3 + 2 x_3^2.
\end{align*}
\end{example}

\begin{example}
Let $g \colon \{0, 1\}^3 \to \{0, 1\}$ be given by $g(x_1, x_2, x_3) = x_1 x_2 + x_2 x_3$ (addition and multiplication modulo $2$). The identification minors of $g$ are
\begin{align*}
g_{\{1,2\}}(x_1, x_2) &= x_1 + x_1 x_2, & 
g_{\{1,3\}}(x_1, x_2) &= 0, &
g_{\{2,3\}}(x_1, x_2) &= x_1 x_2 + x_2.
\end{align*}
Note that $g_{\{1,2\}} \equiv g_{\{2,3\}}$.
\end{example}

\begin{example}
Let $n$ be an integer at least $2$, let $A := \{1, \dots n\}$, let $B$ be a set with at least two elements, and let $\alpha$ and $\beta$ be distinct elements of $B$. Define the function $h \colon A^n \to B$ by the rule
\[
h(a_1, \dots, a_n) =
\begin{cases}
\alpha, & \text{if $(a_1, \dots, a_n) = (1, \dots, n)$,} \\
\beta, & \text{otherwise.}
\end{cases}
\]
It is clear that $h$ depends on all of its arguments, and for every $I \in \couples$, the identification minor $h_I \colon A^{n-1} \to B$ is the constant map taking value $\beta$.
\end{example}

\subsection{Permutations and permutation groups}

We will use both the standard one-line and cycle notations to denote permutations. In \emph{one-line notation,} an $n$-tuple $(a_1, \dots, a_n) \in \nset{n}^n$ with pairwise distinct entries denotes the permutation $\sigma$ on $\nset{n}$ satisfying $\sigma(i) = a_i$ for all $i \in \nset{n}$. (Compare this with the discussion on $n$-tuples in Subsection~\ref{sec:preliminaries:general}.) In \emph{cycle notation,} the expression $(a_1 \; \cdots \; a_r)$, where $r \geq 1$ and $a_1, \dots, a_r$ are pairwise distinct elements of $A$, denotes the permutation $\sigma$ on $A$ satisfying $\sigma(a_i) = a_{i+1}$ whenever $1 \leq i \leq r - 1$, $\sigma(a_r) = a_1$, and $\sigma(m) = m$ whenever $m \in A \setminus \{a_1, \dots, a_r\}$; such a permutation is called a \emph{cycle}, and the number $r$ is its \emph{length.} Every permutation is a product of disjoint cycles. For example, consider the permutation $\sigma$ on $\{1, \dots, 7\}$ given by the following table.
\[
\begin{matrix}
i         & 1 & 2 & 3 & 4 & 5 & 6 & 7 \\
\sigma(i) & 5 & 1 & 7 & 4 & 2 & 6 & 3
\end{matrix}
\]
The representation of $\sigma$ in one-line notation is $(5, 1, 7, 4, 2, 6, 3)$ and one possible representation of $\sigma$ in cycle notation is $(1 \; 5 \; 2) (3 \; 7)$. We denote the identity permutation on any underlying set by $\id$.

The symmetric group and the alternating group on $\nset{n}$ are denoted by $\symm{n}$ and $\alt{n}$, respectively. For a subset $S \subseteq \nset{n}$, we denote by $\symm{S}$ the subgroup of $\symm{n}$ comprising those permutations that fix all elements of $\nset{n} \setminus S$. Similarly, we denote by $\alt{S}$ the subgroup of $\alt{n}$ comprising those even permutations that fix all elements of $\nset{n} \setminus S$.

We recall here some well-known facts about permutations and permutation groups (see, e.g., \cite{Lang}), and we will use these facts throughout the paper without explicit mention.
\begin{fact}
\mbox{}
\begin{enumerate}[\rm (i)]
\item If $1 < p < r \leq n$, then
\begin{itemize}
\item $(1 \; 2 \; \cdots \; p) (p \;\; p + 1 \; \cdots \; r) = (1 \; 2 \; \cdots \; r)$,
\item $(1 \; \cdots \; r) (2 \; \cdots \; r)^{-1} = (1 \; 2)$.
\end{itemize}

\item
Examples of generating sets of the symmetric group $\symm{n}$ include
\begin{itemize}
\item $(1 \; 2 \; \cdots \; n)$ and $(1 \; 2)$,
\item $(2 \; 3 \; \cdots \; n)$ and $(1 \; 2)$,
\item $(1 \; 2 \; \cdots \; n)$ and $(2 \; 3 \; \cdots \; n)$,
\item the set of all adjacent transpositions $(i \;\; i + 1)$, $1 \leq i \leq n - 1$,
\item $A_n$ and any odd permutation.
\end{itemize}

\item
Examples of generating sets of the alternating group $\alt{n}$ include
\begin{itemize}
\item $(1 \; 2 \; 3)$ and $(1 \; 2 \; \cdots \; n)$, if $n$ is odd and $n \geq 3$,
\item $(1 \; 2 \; 3)$ and $(2 \; 3 \; \cdots \; n)$, if $n$ is even and $n \geq 4$,
\item the set of all cycles $(i \;\; i + 1 \;\; i + 2)$, where $i$ is odd and $1 \leq i \leq n - 2$, if $n$ is odd and $n \geq 3$.
\end{itemize}
\end{enumerate}
\end{fact}

\subsection{Invariance groups of functions}

The symmetric group $\symm{n}$ acts on $A^n$ by the rule $\vect{a}^\sigma = \vect{a} \sigma$ for each $\sigma \in \symm{n}$. A function $f \colon A^n \to B$ is \emph{invariant} under a permutation $\sigma \in \symm{n}$, if for all $\vect{a} \in A^n$ it holds that $f(\vect{a}) = f(\vect{a} \sigma)$. Let $\Inv f$ denote the set of permutations of $\nset{n}$ under which $f$ is invariant. Clearly every function is invariant under the identity permutation. If $f$ is invariant under $\sigma$ and $\sigma'$, then $f$ is also invariant under $\sigma^{-1}$ and $\sigma \sigma'$. Hence $\Inv f$ constitutes a subgroup of $\symm{n}$, and it is called the \emph{invariance group} of $f$. If $f$ is invariant under all permutations, i.e., if $\Inv f = \symm{n}$, then $f$ is called \emph{totally symmetric.}

%%%%%%%%%%%%%%%%%%%%%%%%%%%%%%%%%%%%%%%%%%%%%%%%%%

\section{Reconstructing functions from identification minors}
\label{sec:reconstruction}

\subsection{Reconstruction problem for functions of several arguments}

We recall the usual terminology of reconstruction problems in the setting of functions of several arguments and identification minors.
Assume that $n \geq 2$ and let $f \colon A^n \to B$.

\begin{enumerate}[\rm (i)]
\item The \emph{deck} of $f$, denoted $\deck f$, is the multiset $\{f_I / {\equiv} : I \in \couples\}$ of the equivalence classes of the identification minors of $f$.
Any element of the deck of $f$ is called a \emph{card} of $f$.

\item A function $g \colon A^n \to B$ is a \emph{reconstruction} of $f$, if $\deck f = \deck g$, or, equivalently, if there exists a bijection $\phi \colon \couples \to \couples$ such that $f_I \equiv g_{\phi(I)}$ for all $I \in \couples$.

\item A function is \emph{reconstructible} if it is equivalent to all of its reconstructions, or, equivalently, if it is equivalent to all functions with the same deck.

\item A parameter defined for all functions is \emph{reconstructible} if it has the same value for all the reconstructions of any function.

\item A class $\cl{C} \subseteq \cl{F}_{AB}$ of functions is \emph{recognizable} if all reconstructions of the members of $\mathcal{C}$ are members of $\mathcal{C}$.

\item A class $\cl{C} \subseteq \cl{F}_{AB}$ of functions is \emph{weakly reconstructible} if for every $f \in \cl{C}$, all reconstructions of $f$ that are members of $\cl{C}$ are equivalent to $f$.

\item A class $\cl{C} \subseteq \cl{F}_{AB}$ of functions is \emph{reconstructible} if all members of $\cl{C}$ are reconstructible.
\end{enumerate}

Note that if a class of functions is recognizable and weakly reconstructible, then it is reconstructible.

A reconstruction problem for functions of several arguments can be formulated as follows.

\begin{question}
\label{q:rec1}
Let $A$ and $B$ be sets with at least two elements, and let $n$ be an integer greater than or equal to $2$. Is every function $f \colon A^n \to B$ reconstructible?
\end{question}

\begin{remark}
\label{rem:finiteA}
The answer to Question~\ref{q:rec1} is negative if $n$ is not sufficiently large. Namely, if $n \leq \card{A}$, then the set $A^n_{\neq}$ is nonempty (see Subsection~\ref{sec:preliminaries:general}), and the elements of $A^n_{\neq}$ do not contribute to any identification minor of $f$. Thus, if $f$ and $g$ are $n$-ary functions that coincide on $A^n \setminus A^n_{\neq}$, then $f_I = g_I$ for every $I \in \couples$ but $f$ and $g$ need not be equivalent---consider, for example, any functions $f$ and $g$ that coincide on $A^n \setminus A^n_{\neq}$ such that $f|_{A^n_{\neq}}$ is constant and $g|_{A^n_{\neq}}$ is nonconstant. This also shows that functions with infinite domains are not reconstructible. Furthermore, Example~\ref{ex:A+1} shows that the answer to Question~\ref{q:rec1} is negative if $n = \card{A} + 1$, and Proposition~\ref{prop:FnequivG} shows that the answer is negative if $\card{A} \equiv 0 \pmod{4}$ or $\card{A} \equiv 3 \pmod{4}$ and $n = \card{A} + 2$.
\end{remark}

\begin{remark}
Should we like to answer Question~\ref{q:rec1} in the negative for all $A$ and $B$, for a fixed $n$, it would suffice to find counterexamples among Boolean functions. For, let $A$, $B$, $C$, and $D$ be sets such that $A \subseteq C$ and $B \subseteq D$, and let $b$ be an arbitrary element of $B$. Any function $f \colon A^n \to B$ can be extended into a function $f^\dagger \colon C^n \to D$ as follows:
\[
f^\dagger(\vect{a}) =
\begin{cases}
f(\vect{a}), & \text{if $\vect{a} \in A^n$,} \\
b, & \text{otherwise.}
\end{cases}
\]
It is easy to verify that $(f^\dagger)_I = (f_I)^\dagger$ for every $I \in \binom{n}{2}$. Thus, if $f, g \colon A^n \to B$, $f \not\equiv g$ and $\deck f = \deck g$, then $f^\dagger \not\equiv g^\dagger$ and $\deck f^\dagger = \deck g^\dagger$.
\end{remark}

\subsection{Examples of reconstructible functions: constant functions}

\begin{example}
\label{ex:constant}
It is easy to verify that if $f \colon A^n \to B$ is a constant function and $n > \card{A}$, then $f$ is reconstructible.
For, if $\alpha \in B$ and $f(\vect{a}) = \alpha$ for all $\vect{a} \in A^n$, then for every $I \in \couples$, $f_I(\vect{b}) = \alpha$ for all $\vect{b} \in A^{n-1}$. Assume that $g \colon A^n \to B$ is a reconstruction of $f$. Then there exists a bijection $\sigma \colon \couples \to \couples$ such that $f_I \equiv g_{\sigma(I)}$ for all $I \in \couples$, and for each $I \in \couples$, there exists a permutation $\rho_I \in \symm{n-1}$ such that $f_I(\vect{b}) = g_{\sigma(I)}(\vect{b} \rho_I)$ for all $\vect{b} \in A^{n-1}$. Let $\vect{a} \in A^n$. Since $n > \card{A}$, there exist $\vect{b} \in A^{n-1}$ and $I \in \couples$ such that $\vect{a} = \vect{b} \delta_I$. Then $g(\vect{a}) = g(\vect{b} \delta_I) = g_I(\vect{b}) = f_{\sigma(I)}(\vect{b} \rho_I) = \alpha$. Thus $f = g$, and we conclude that $f$ is reconstructible.
\end{example}

\subsection{Examples of reconstructible functions: functions determined by $\supp$ and $\oddsupp$}

Following the definitions presented by Berman and Kisielewicz~\cite{BerKis}, the maps $\supp \colon \bigcup_{n \geq 1} A^n \to \mathcal{P}(A)$ and $\oddsupp \colon \bigcup_{n \geq 1} A^n \to \mathcal{P}(A)$ are given by the rules
\begin{align*}
\supp(a_1, \dots, a_n) &= \{a_1, \dots, a_n\}, \\
\oddsupp(a_1, \dots, a_n) &= \{a \in A : \text{$\card{\{i \in \nset{n} : a_i = a\}}$ is odd}\}.
\end{align*}
A function $f \colon A^n \to B$ is \emph{determined by $\supp$,} if there exists a map $f^* \colon \mathcal{P}(A) \to B$ such that $f = f^* \circ {\supp}|_{A^n}$. Similarly, a function $f \colon A^n \to B$ is \emph{determined by $\oddsupp$,} if there exists a map $f^* \colon \mathcal{P}(A) \to B$ such that $f = f^* \circ {\oddsupp}|_{A^n}$.

\begin{remark}
Functions determined by $\supp$ or $\oddsupp$ are totally symmetric. Hence, they depend on all of their arguments or on none of them.
\end{remark}

\begin{remark}
For all $\vect{a} \in A^{n-1}$ and for all $I \in \couples$, we have that $\supp(\vect{a}) = \supp(\vect{a} \delta_I)$.
\end{remark}

\begin{remark}
\label{rem:suppoddsupp}
Let $f \colon A^n \to B$. It is easy to verify that if $f = f^* \circ {\supp}|_{A^n}$ for some $f^* \colon \mathcal{P}(A) \to B$, then $f_I = f^* \circ {\supp}|_{A^{n-1}}$ for all $I \in \couples$. Also, if $f = f^* \circ {\oddsupp}|_{A^n}$, then for all $I \in \couples$, the $(\min I)$-th argument of $f_I$ is inessential and $f_I \equiv f^* \circ {\oddsupp}|_{A^{n-2}}$.

For $n \geq 1$, denote
\begin{align*}
\mathcal{P}_{\leq n}(A) &:= \{S \subseteq A : \card{S} \leq n\}, \\
\mathcal{P}'_n(A) &:= \{S \subseteq A : \card{S} \in \{n, n-2, n-4, \dots\}\}.
\end{align*}
The range of ${\supp}|_{A^n}$ equals $\mathcal{P}_{\leq n}(A)$. Thus, only the restriction of $f^*$ to $\mathcal{P}_{\leq n}(A)$ is relevant in the composition $f^* \circ {\supp}|_{A^n}$, and $f^* \circ {\supp}|_{A^n} = g^* \circ {\supp}|_{A^n}$ if and only if $f^*|_{\mathcal{P}_{\leq n}(A)} = g^*|_{\mathcal{P}_{\leq n}(A)}$.
Similarly, the range of ${\oddsupp}|_{A^n}$ equals $\mathcal{P}'_n(A)$. Thus, only the restriction of $f^*$ to $\mathcal{P}'_n(A)$ is relevant in the composition $f^* \circ {\oddsupp}|_{A^n}$, and $f^* \circ {\oddsupp}|_{A^n} = g^* \circ {\oddsupp}|_{A^n}$ if and only if $f^*|_{\mathcal{P}'_n(A)} = g^*|_{\mathcal{P}'_n(A)}$.
\end{remark}

Let us recall a few useful results about totally symmetric functions from the paper by Willard~\cite{Willard}.

\begin{lemma}[{Willard~\cite[Lemma~2.2]{Willard}}]
\label{lem:WillardLem2.2}
Assume that $f \colon A^n \to B$ is totally symmetric and depends on all of its arguments. If $n > 2$ and for some $I \in \couples$, $f_I$ is essentially $(n-1)$-ary and totally symmetric, then $f$ is determined by $\supp$.
\end{lemma}

\begin{lemma}[{Willard~\cite[Corollary~2.3]{Willard}}]
\label{lem:WillardCor2.3}
Assume that $n > \max(\card{A}, 3)$ and $f \colon A^n \to B$ depends on all of its arguments. If no identification minor of $f$ depends on all of its arguments, then $f$ is determined by $\oddsupp$.
\end{lemma}

The reconstructibility of functions determined by $\supp$ or $\oddsupp$ (of sufficiently large arity) follows almost immediately from the previous two lemmas.

\begin{proposition}
\label{prop:suppreconstructible}
Let $f \colon A^n \to B$, and assume that $n > \card{A}$ and $f$ is determined by $\supp$. Then the function $f$ is reconstructible.
\end{proposition}

\begin{proof}
Since functions determined by $\supp$ are totally symmetric, $f$ depends either on all of its arguments or on none of them. If $f$ has no essential arguments, then it is constant and hence it is reconstructible by Example~\ref{ex:constant}. Assume thus that $f$ depends on all of its arguments. Then there exists a map $f^* \colon \mathcal{P}(A) \to B$ such that $f = f^* \circ {\supp}|_{A^n}$ and $f^*|_{\mathcal{P}_{\leq n}(A)}$ is not a constant function. By Remark~\ref{rem:suppoddsupp}, $f_I = f^* \circ {\supp}|_{A^{n-1}}$ for all $I \in \couples$. Since $n > \card{A}$, we have that $\mathcal{P}_{\leq n-1}(A) = \mathcal{P}(A) = \mathcal{P}_{\leq n}(A)$; hence $f_I$ depends on all of its $n - 1$ arguments.
Let $g \colon A^n \to B$ be a reconstruction of $f$. Then $g_I \equiv f^* \circ {\supp}|_{A^{n-1}}$ for all $I \in \couples$. Since $f^* \circ {\supp}|_{A^{n-1}}$ is totally symmetric, we have in fact that $g_I = f^* \circ {\supp}|_{A^{n-1}}$ for all $I \in \couples$.

We claim that $g = f$. In order to prove this, let $\vect{a} \in A^n$. Since $n > \card{A}$, there exist $\vect{b} \in A^{n-1}$ and $I \in \couples$ such that $\vect{a} = \vect{b} \delta_I$. Then
\[
g(\vect{a})
= g(\vect{b} \delta_I)
= g_I(\vect{b})
= f^*(\supp(\vect{b}))
= f^*(\supp(\vect{b} \delta_I))
= f^*(\supp(\vect{a}))
= f(\vect{a}).
\]
Thus $g = f$, and we conclude that $f$ is reconstructible.
\end{proof}

\begin{proposition}
Let $f \colon A^n \to B$, and assume that $n > \max(\card{A}, 3)$ and $f$ is determined by $\oddsupp$. Then the function $f$ is reconstructible.
\end{proposition}

\begin{proof}
Since functions determined by $\oddsupp$ are totally symmetric, $f$ depends either on all of its arguments or on none of them. If $f$ has no essential arguments, then it is constant and hence it is reconstructible by Example~\ref{ex:constant}. Assume thus that $f$ depends on all of its arguments.
Then there exists a map $f^* \colon \mathcal{P}(A) \to B$ such that $f = f^* \circ {\oddsupp}|_{A^n}$ and $f^*|_{\mathcal{P}'_n(A)}$ is not a constant function.
By Remark~\ref{rem:suppoddsupp}, for all $I \in \couples$, $f_I$ has an inessential argument and $f_I \equiv f^* \circ {\oddsupp}|_{A^{n-2}}$. Let $g \colon A^n \to B$ be a reconstruction of $f$. Lemma~\ref{lem:WillardCor2.3} implies that $g$ is determined by $\oddsupp$. Thus, $g = g^* \circ {\oddsupp}|_{A^n}$ for some $g \colon \mathcal{P}(A) \to B$. By Remark~\ref{rem:suppoddsupp}, $g_I \equiv g^* \circ {\oddsupp}|_{A^{n-2}}$ for all $I \in \couples$. Then it holds that $f^* \circ {\oddsupp}|_{A^{n-2}} \equiv g^* \circ {\oddsupp}|_{A^{n-2}}$, because $\deck f = \deck g$. Since ${\supp}|_{A^{n-1}}$ is totally symmetric, we have that $f^* \circ {\oddsupp}|_{A^{n-2}} = g^* \circ {\oddsupp}|_{A^{n-2}}$, and Remark~\ref{rem:suppoddsupp} implies that $f^*|_{\mathcal{P}'_{n-2}(A)} = g^*|_{\mathcal{P}'_{n-2}(A)}$. Since $n > \card{A}$, it holds that $\mathcal{P}'_{n-2}(A) = \mathcal{P}'_n(A)$, and it follows that $f = f^*|_{\mathcal{P}'_n(A)} \circ {\oddsupp}|_{A^n} = g^*|_{\mathcal{P}'_n(A)} \circ {\oddsupp}|_{A^n} = g$. We conclude that $f$ is reconstructible.
\end{proof}

We are going to extend these results in Section~\ref{sec:symmetric}, in which we prove that all totally symmetric functions of sufficiently large arity are reconstructible.

\subsection{Examples of nonreconstructible functions}

We present here a scheme for producing functions of arity $\card{A} + 1$ with a predetermined deck of a special form. With a suitable choice of parameters, nonequivalent functions with the same deck will arise.

\begin{definition}
\label{def:fGPphi}
Assume that $n = k + 1$, and let $A$ be a set such that $\card{A} = k$.
Let $g^* \colon \mathcal{P}(A) \to B$ and let $g \colon A^k \to B$, $g = g^* \circ {\supp}|_{A^k}$. Let $G := (g^I)_{I \in \couples}$ be a family of functions $g^I \colon A^k \to B$ satisfying $g^I(\vect{a}) = g(\vect{a})$ whenever $\supp(\vect{a}) \neq A$, and let $P := (\rho_I)_{I \in \couples}$ be a family of permutations from $\symm{k}$. Let $\phi \colon \couples \to \couples$ be a bijection. Define $f_{G,P,\phi} \colon A^n \to B$ by the rule $f_{G,P,\phi}(\vect{b}) = g^{\phi(I)}(\vect{a} \rho_I)$ if $\vect{b} = \vect{a} \delta_I$. This definition is good, because if $\supp(\vect{b}) = A$, then there is a unique $\vect{a} \in A^k$ and a unique $I \in \couples$ such that $\vect{b} = \vect{a} \delta_I$; and if $\supp(\vect{b}) \neq A$, then for every $\vect{a} \in A^k$ and $I \in \couples$ satisfying $\vect{b} = \vect{a} \delta_I$, we have $\supp(\vect{a}) = \supp(\vect{b}) \neq A$ and $g^{\phi(I)}(\vect{a} \rho_I) = g^*(\supp(\vect{a} \rho_I)) = g^*(\supp(\vect{b}))$.
\end{definition}

\begin{lemma}
\label{lem:fGPphi}
For any family $G := (g^I)_{I \in \couples}$ of functions, any family $P := (\rho_I)_{I \in \couples}$ of permutations, and any bijection $\phi \colon \couples \to \couples$ as in Definition~\ref{def:fGPphi}, it holds that $(f_{G,P,\phi})_I \equiv g^{\phi(I)}$ for every $I \in \couples$. Consequently, $\deck f_{G,P,\phi}$ equals the multiset of the equivalence classes $g^I / {\equiv}$ with $I$ ranging in $\couples$.
\end{lemma}

\begin{proof}
It follows directly from the definitions that, for all $\vect{a} \in A^k$,
\[
(f_{G,P,\phi})_I(\vect{a}) = f_{G,P,\phi}(\vect{a} \delta_I) = g^{\phi(I)}(\vect{a} \rho_I),
\]
that is, $(f_{G,P,\phi})_I \equiv g^{\phi(I)}$. The claim about $\deck f_{G,P,\phi}$ follows immediately.
\end{proof}

Thus, for a fixed family $G$ of functions and for families $P$ and $P'$ of permutations and for bijections $\phi$ and $\phi'$, the functions $f_{G,P,\phi}$ and $f_{G,P',\phi'}$ are reconstructions of each other but they are not necessarily equivalent. As the following example illustrates, it is indeed possible that $f_{G,P,\phi} \not\equiv f_{G,P',\phi'}$ for a suitable choice of $G$, $P$, $P'$, $\phi$, and $\phi'$. Thus, the answer to Question~\ref{q:rec1} is negative if $n = \card{A} + 1$.

\begin{example}
\label{ex:A+1}
Let $n = k + 1$, and let $A$ and $B$ be sets such that $\card{A} = k$ and $A \subseteq B$. Let $\beta \in B$, and let $h \colon A^k \to B$ be the function
\[
h(\vect{a}) =
\begin{cases}
a_1, & \text{if $\supp(\vect{a}) = A$,} \\
\beta, & \text{otherwise.}
\end{cases}
\]
(Note that, letting $g = g^* \circ {\supp}|_{A^k}$, where $g^* \colon \mathcal{P}(A) \to B$ is the constant map $S \mapsto \beta$ for all $S \in \mathcal{P}(A)$, we have that $h$ satisfies the condition that $h(\vect{a}) = g(\vect{a})$ whenever $\supp(\vect{a}) \neq A$.) For each $I \in \couples$, let $g^I := h$; let $\rho_I$ be the identity permutation if $\min I = 1$, and let $\rho_I$ be the transposition of $1$ and $\min I$ if $\min I \neq 1$; and let $\rho'_I$ be the identity permutation. Let $\phi$ be the identity map on $\couples$. Denote $G := (g^I)_{I \in \couples}$, $P := (\rho_I)_{I \in \couples}$, and $P' := (\rho'_I)_{I \in \couples}$, and let $f := f_{G,P,\phi}$, $g := f_{G,P',\phi}$.
Then we have
\begin{align*}
f(\vect{a}) &=
\begin{cases}
b, & \text{if $\supp(\vect{a}) = A$ and $b$ occurs twice in $\vect{a}$ ($b \in A$),} \\
\alpha, & \text{otherwise,}
\end{cases}
\\
g(\vect{a}) &=
\begin{cases}
a_1, & \text{if $\supp(\vect{a}) = A$,} \\
\alpha, & \text{otherwise.}
\end{cases}
\end{align*}
By Lemma~\ref{lem:fGPphi}, we have that $f_I \equiv h \equiv g_I$ for all $I \in \couples$, and $\deck f = \deck g$.
Furthermore, it is easy to verify that $f \not\equiv g$. To see this, note that if $c$ and $d$ are distinct elements of $A$, then, on the one hand, there exists a tuple $\vect{a} \in A^n$ such that $\supp(\vect{a}) = A$, $\vect{a}$ has two occurrences of $c$ and $g(\vect{a}) = d$, but, on the other hand, for every tuple $\vect{b} \in A^n$ such that $\supp(\vect{b}) = A$ and $\vect{b}$ has two occurrences of $c$, it holds that $f(\vect{b}) = c$.

Note that the function $f$ is totally symmetric, and $g$ is determined by the order of first occurrence (see Section~\ref{subsec:ofo}).
\end{example}

\subsection{Examples of reconstructible parameters}

\begin{example}
Let $f \colon A^n \to B$. The \emph{diagonal} of $f$ is the map $\Delta_f \colon A \to B$ given by $\Delta_f(a) = f(a, a, \dots, a)$ for all $a \in A$. It is easy to verify that the diagonal of every minor of $f$ equals $\Delta_f$. It follows that the diagonal is a reconstructible parameter of functions.
\end{example}

\subsection{Examples of recognizable classes of functions}

Let $\cl{C} \subseteq \cl{F}_{AB}$ be a class of functions, and let $d$ and $r$ be positive integers. Let $S \subseteq (A^r)^d$ and $T \subseteq B^r$. The couple $(S, T)$ is called a \emph{$d$-dimensional nonmembership witness scheme} for $\cl{C}$, if it holds for every $n \geq 1$ and $f \in \cl{F}_{AB}^{(n)}$ that $f \notin \cl{C}$ if and only if there exists elements $a_{ij} \in A$ ($1 \leq i \leq r$, $1 \leq j \leq n$) and a map $\rho \colon \nset{d} \to \nset{n}$ such that $a_{1j} = a_{2j} = \dots = a_{rj}$ for every $j \in \nset{n} \setminus \range \rho$ and
\begin{gather*}
\bigl( (a_{1 \rho(1)}, a_{2 \rho(1)}, \dots, a_{r \rho(1)}), \dots, (a_{1 \rho(d)}, a_{2 \rho(d)}, \dots, a_{r \rho(d)}) \bigr) \in S, \\
\bigl( f(a_{11}, a_{12}, \dots, a_{1n}), \dots, f(a_{r1}, a_{r2}, \dots, a_{rn}) \bigr) \in T.
\end{gather*}

\begin{proposition}
\label{prop:witness}
Assume that $\cl{C} \subseteq \cl{F}_{AB}$ is a class of functions that is closed under formation of minors and there exists a $d$-dimensional nonmembership witness scheme $(S, T)$ for $\cl{C}$. Then the class $\cl{D} := \bigcup_{n \geq d + \card{A} + 1} \cl{C}_{AB}^{(n)}$ is recognizable.
\end{proposition}

\begin{proof}
Let $f \colon A^n \to B$ with $n \geq d + \card{A} + 1$. We claim that $f_I \in \cl{C}$ for every $I \in \couples$ if and only if $f \in \cl{C}$. The sufficiency is clear, because $\cl{C}$ is closed under formation of minors. For necessity, assume that $f \notin \cl{C}$. Then there exist a positive integer $r$ and elements $a_{ij} \in A$ ($1 \leq i \leq r$, $1 \leq j \leq n$) such that $a_{1j} = a_{2j} = \dots = a_{rj}$ for every $j \in \nset{n} \setminus \range \rho$ and
\begin{gather*}
\bigl( (a_{1 \rho(1)}, a_{2 \rho(1)}, \dots, a_{r \rho(1)}), \dots, (a_{1 \rho(d)}, a_{2 \rho(d)}, \dots, a_{r \rho(d)}) \bigr) \in S, \\
\bigl( f(a_{11}, a_{12}, \dots, a_{1n}), \dots, f(a_{r1}, a_{r2}, \dots, a_{rn}) \bigr) \in T.
\end{gather*}
Since $n - d \geq \card{A} + 1$, there exist indices $p, q \in \nset{n} \setminus \range \rho$ such that $a_{1p} = a_{1q}$. Set $J := \{p, q\}$, and let $b_{ij} \in A$ ($1 \leq i \leq r$, $1 \leq j \leq n - 1$) be the unique elements such that $a_{ij} = b_{i, \delta_J(j)}$ for all $1 \leq i \leq r$, $1 \leq j \leq n$, and let $\tau = \delta_J \circ \rho$. Then $b_{i,\tau(\ell)} = b_{i,\delta_J(\rho(\ell))} = a_{i,\rho(\ell)}$, and it is easy to verify that
$b_{1j} = b_{2j} = \dots = b_{rj}$ for every $j \in \nset{n - 1} \setminus \range \tau$ and
\begin{gather*}
\bigl( (b_{1 \tau(1)}, b_{2 \tau(1)}, \dots, b_{r \tau(1)}), \dots, (b_{1 \tau(d)}, b_{2 \tau(d)}, \dots, b_{r \tau(d)}) \bigr) \in S, \\
(f_J(b_{11}, b_{12}, \dots, b_{1,n-1}), \dots, f_J(b_{r1}, b_{r2}, \dots, b_{r,n-1})) \in T.
\end{gather*}
Thus, $f_J \notin \cl{C}$, because $(S, T)$ is a nonmembership witness scheme for $\cl{C}$.

This implies that every reconstruction of every member of $\cl{D}$ is again a member of $\cl{D}$, that is, $\cl{D}$ is recognizable.
\end{proof}

\begin{example}
Let $(A; \leq_A)$ and $(B; \leq_B)$ be partially ordered sets. A function $f \colon A^n \to B$ is \emph{order-preserving} if for all $\vect{a}, \vect{b} \in A^n$, the condition $\vect{a} \leq_A \vect{b}$ (i.e., $a_i \leq_A b_i$ for all $i \in \nset{n}$) implies $f(\vect{a}) \leq_B f(\vect{b})$. We claim that $(\leq_A, B^2 \setminus {\leq_B})$ is a $1$-dimensional nonmembership witness scheme for the class of order-preserving functions from $A$ to $B$. Consequently, by Proposition~\ref{prop:witness}, the class of order-preserving functions from $A$ to $B$ of arity at least $\card{A} + 2$ is recognizable.

In order to prove the claim, observe first that if $f \colon A^n \to B$ is order-preserving, then for all $a_{ij}$ ($1 \leq i \leq 2$, $1 \leq j \leq n$) and for every $\rho \colon \nset{1} \to \nset{n}$ such that $a_{1j} = a_{2j}$ for all $j \in \nset{n} \setminus \range \rho$ and $a_{1 \rho(1)} \leq_A a_{2 \rho(2)}$, we have that $(a_{11}, \dots, a_{1n}) \leq_A (a_{2n}, \dots, a_{2n})$; hence $f(a_{11}, \dots, a_{1n}) \leq_B f(a_{21}, \dots, a_{2n})$.

If $f \colon A^n \to B$ is not order-preserving, then there exist tuples $\vect{a}, \vect{b} \in A^n$ such that $\vect{a} \leq_A \vect{b}$ and $f(\vect{a}) \not\leq_B f(\vect{b})$. Consider the sequence
\begin{align*}
\vect{c}_1 &:= (a_1, a_2, a_3, \dots, a_n) = \vect{a}, \\
\vect{c}_2 &:= (b_1, a_2, a_3, \dots, a_n), \\
\vect{c}_3 &:= (b_1, b_2, a_3, \dots, a_n), \\
& \,\,\,\,\, \vdots \\
\vect{c}_n &:= (b_1, \dots, b_{n-1}, a_n), \\
\vect{c}_{n+1} &:= (b_1, \dots, b_{n-1}, b_n) = \vect{b}.
\end{align*}
It holds that $\vect{c}_\ell \leq_A \vect{c}_{\ell + 1}$ for all $\ell \in \nset{n}$. There exists an index $s \in \nset{n}$ such that $f(\vect{c}_s) \not\leq_B f(\vect{c}_{s+1})$ (otherwise we would have $f(\vect{a}) \leq_B f(\vect{b})$ by the transitivity of $\leq_B$, a contradiction). Choosing $a_{ij} := \vect{c}_{s - i + 1}(j)$ ($1 \leq i \leq 2$, $1 \leq j \leq n$) and $\rho \colon \nset{1} \to \nset{n}$, $1 \mapsto s$, the desired conditions for a nonmembership witness are satisfied.
\end{example}

%%%%%%%%%%%%%%%%%%%%%%%%%%%%%%%%%%%%%%%%%%%%%%%%%%

\section{Functions with a unique identification minor}
\label{sec:unique}

\subsection{Unique identification minors}

A function $f \colon A^n \to B$ has a \emph{unique identification minor} if $f_I \equiv f_J$ for all $I, J \in \couples$. The deck of such a function thus consists of a single element with multiplicity equal to the binomial coefficient $\binom{n}{2}$. The class of functions with a unique identification minor is obviously recognizable.  As a first step towards answering the reconstruction problem for functions, we focus on functions with a unique identification minor. In this section we present some wide classes of functions with a unique identification minor, namely the class of $2$-set-transitive functions (which include the totally symmetric functions) and the class of functions (weakly) determined by the order of first occurrence.

The property of a function's having a unique identification minor is closely related to the property of having a unique lower cover in the minor partial order. The former is a more restrictive property than the latter; if a function $f$ has a unique identification minor, then this minor is obviously the unique lower cover $f$. Boolean functions with unique lower covers have been studied by Bouaziz, Couceiro and Pouzet~\cite{BCP}.

\begin{lemma}
Let $f \colon A^n \to B$. Then $f$ has a unique identification minor if and only if for each $(I, J) \in \couples^2$ there exists a bijection $\tau_{IJ} \colon \nset{n - 1} \to \nset{n - 1}$ such that
$f(\vect{a} \delta_I) = f(\vect{a} \tau_{IJ} \delta_J)$ for all $\vect{a} \in A^{n-1}$.
\end{lemma}

\begin{proof}
The condition that $f$ has a unique identification minor is equivalent to the condition that $f_I \equiv f_J$ for all $I, J \in \couples$. This is equivalent to the condition that for each $(I, J) \in \couples^2$ there exists a bijection $\tau_{IJ} \colon \nset{n - 1} \to \nset{n - 1}$ such that
$f_I(\vect{a}) = f_J(\vect{a} \tau_{IJ})$ for all $\vect{a} \in A^{n-1}$. By the definition of identification minor, we have $f_I(\vect{a}) = f(\vect{a} \delta_I)$ and $f_J(\vect{a} \tau_{IJ}) = f(\vect{a} \tau_{IJ} \delta_J)$.
By putting these equivalences together, the claim follows.
\end{proof}

\subsection{$2$-set-transitive functions}

A permutation group $G \leq \symm{n}$ is called
\begin{itemize}
\item \emph{$m$-transitive,} if for all $m$-element subsets $\{a_1, \dots, a_m\}$ and $\{b_1, \dots, b_m\}$ of $\nset{n}$, there exists a permutation $\sigma \in G$ such that $\sigma(a_i) = b_i$ for all $1 \leq i \leq m$;
\item \emph{$m$-set-transitive,} if for all $m$-element subsets $\{a_1, \dots, a_m\}$ and $\{b_1, \dots, b_m\}$ of $\nset{n}$, there exists a permutation $\sigma \in G$ such that $\{\sigma(a_1), \dots, \sigma(a_m)\} = \{b_1, \dots, b_m\}$.
\end{itemize}
A function $f \colon A^n \to B$ is \emph{$m$-transitive} (\emph{$m$-set-transitive}) if the invariance group of $f$ is $m$-transitive ($m$-set-transitive, respectively).
Observe that $m$-transitivity implies $m$-set-transitivity for all $m \geq 1$, and $m$-transitivity implies $m'$-transitivity for all $1 \leq m' \leq m$, and total symmetry implies $m$-transitivity and hence $m$-set-transitivity for all $m \geq 1$.

The symmetric group $\symm{n}$ acts on $\couples$ by the rule $I^\sigma = I \sigma := \{\sigma(i) : i \in I\}$ for all $\sigma \in \symm{n}$.

\begin{lemma}
\label{lem:hatsigma}
Let $\sigma \in \symm{n}$ and $I \in \couples$. Then there exists a permutation $\hat{\sigma} \in \symm{n-1}$ that satisfies $\hat{\sigma} \circ \delta_{I \sigma^{-1}} = \delta_I \circ \sigma$ and $\hat{\sigma}(\min I \sigma^{-1}) = \min I$.
\end{lemma}

\begin{proof}
Observe first that for any $J \in \symm{n}$, the restriction of the map $\delta_J$ to the set $\nset{n} \setminus \{\max J\}$ is a bijection. Define the map $\beta_J \colon \nset{n-1} \to \nset{n}$ as $\beta_J(\ell) = ((\delta_J)|_{\nset{n} \setminus \{\max J\}})^{-1}(\ell)$ for all $\ell \in \nset{n-1}$. In other words, $\beta_J$ is obtained from $((\delta_J)|_{\nset{n} \setminus \{\max J\}})^{-1}$ by extending the codomain; both maps are given by the rule $\ell \mapsto \ell$ for $1 \leq \ell < \max J$ and $\ell \mapsto \ell + 1$ for $\max J \leq \ell \leq n - 1$. Then $\beta_J \circ \delta_J \colon \nset{n} \to \nset{n}$ is the map $\ell \mapsto \ell$ for $\ell \neq \max J$ and $\max J \mapsto \min J$.

Let $\sigma \in \symm{n}$ and $I \in \couples$. We have that $(\delta_I \circ \sigma)(\sigma^{-1}(\min I)) = (\delta_I \circ \sigma)(\sigma^{-1}(\max I)) = \min I$. Based on the above observations, it is easy to see that $\hat{\sigma} := \delta_I \circ \sigma \circ \beta_{I \sigma^{-1}}$ is a permutation of $\nset{n-1}$. Furthermore, $\hat{\sigma} \circ \delta_{I \sigma^{-1}} = \delta_I \circ \sigma \circ \beta_{I \sigma^{-1}} \circ \delta_{I \sigma^{-1}} = \delta_I \circ \sigma$.

For the last equality, it holds that
$\hat{\sigma}(\min I \sigma^{-1}) = \delta ( \sigma ( \beta_{I \sigma^{-1}} (\min I \sigma^{-1}))) = \delta_I ( \sigma (\min I \sigma^{-1}))$. Since $\sigma(\min I \sigma^{-1}) \in I$ and $\delta_I$ maps both elements of $I$ to $\min I$, we have that $\delta_I ( \sigma (\min I \sigma^{-1})) = \min I$.
\end{proof}

\begin{lemma}
\label{lem:invperm}
Assume that $f \colon A^n \to B$ is invariant under a permutation $\sigma \in \symm{n}$. Then $f_I \equiv f_{I \sigma}$ for all $I \in \binom{n}{2}$.
\end{lemma}

\begin{proof}
Let $\sigma \in \Inv f$ and $I \in \couples$.
Let $\hat{\sigma}$ be the permutation of $\nset{n-1}$ given by Lemma~\ref{lem:hatsigma}. We have that for all $\vect{a} \in A^{n-1}$,
\[
f_I(\vect{a}) =
f(\vect{a} \delta_I) =
f(\vect{a} \delta_I \sigma) =
f(\vect{a} \hat{\sigma} \delta_{I \sigma^{-1}}) =
f_{I \sigma^{-1}}(\vect{a} \hat{\sigma}).
\]
The first and the last equalities hold by the definition of identification minor. The second equality holds because $f$ is invariant under $\sigma$. The third equality holds by Lemma~\ref{lem:hatsigma}. We conclude that $f_I \equiv f_{I \sigma^{-1}}$, whence the claim follows.
\end{proof}

\begin{proposition}
\label{prop:2trans}
If $f \colon A^n \to B$ is $2$-set-transitive, then $f$ has a unique identification minor.
\end{proposition}
\begin{proof}
Let $I, J \in \couples$. By the $2$-set-transitivity of $f$, there exists a permutation $\sigma \in \Inv f$ such that $I \sigma = J$. Lemma~\ref{lem:invperm} implies that $f_I \equiv f_J$.
\end{proof}

\subsection{Functions determined by the order of first occurrence}
\label{subsec:ofo}

We will present another wide class of functions with a unique identification minor.
Let $A^\sharp := \bigcup_{n \geq 1} A^n_{\neq}$, where $A^n_{\neq}$ is as defined in Equation~\eqref{eq:Anneq} in Subsection~\ref{sec:preliminaries:general}.
We define the function $\ofo \colon \bigcup_{n \geq 1} A^n \to A^\sharp$ by the rule that
$(a_1, \dots, a_n)$ is mapped to the unique tuple $(b_1, \dots, b_r) \in A^\sharp$ that has the following properties:
\begin{itemize}
\item $\{a_1, \dots, a_n\} = \{b_1, \dots, b_r\}$ (i.e., $\supp(a_1, \dots, a_n) = \supp(b_1, \dots, b_r)$),
\item setting $\alpha_j := \min \{i \in \nset{n} : a_i = b_j\}$ for every $j \in \nset{r}$, it holds that $\alpha_1 < \alpha_ 2 < \dots < \alpha_r$.
\end{itemize}
Informally speaking, $\ofo(\vect{a})$ lists the elements occurring in $\vect{a}$ in the \fBold{order} of their \fBold{first} \fBold{occurrence} (hence the acronym $\ofo$). In other words, $\ofo(\vect{a})$ is obtained from $\vect{a}$ by removing all repetitions of elements, retaining only the first occurrence of each element in $\vect{a}$.
A function $f \colon A^n \to B$ is \emph{determined by the order of first occurrence,} if there exists a map $f^* \colon A^\sharp \to B$ such that $f = f^* \circ {\ofo}|_{A^n}$. A function $f \colon A^n \to B$ is \emph{weakly determined by the order of first occurrence,} if there exist a map $f^* \colon A^\sharp \to B$ and a permutation $\sigma \in \symm{n}$ such that $f(\vect{a}) = f^*(\ofo(\vect{a} \sigma))$ for all $\vect{a} \in A^n$, i.e., if $f$ is equivalent to an $n$-ary function determined by the order of first occurrence.

\begin{remark}
The function $\ofo$ is idempotent, i.e., $\ofo(\ofo(\vect{a})) = \ofo(\vect{a})$ for every tuple $\vect{a}$.
\end{remark}

\begin{remark}
For every $\vect{a} \in A^{n-1}$ and for every $I \in \couples$, it holds that $\ofo(\vect{a}) = \ofo(\vect{a} \delta_I)$.
\end{remark}

\begin{remark}
\label{rem:ofolambda}
If $A = \nset{k}$ and $\ofo(\vect{k} \sigma) = \vect{k} \tau$ for a map $\sigma \colon \nset{n} \to \nset{k}$ and an injective map $\tau \colon \nset{\ell} \to \nset{k}$, then $\ofo(\vect{k} \lambda \sigma) = \vect{k} \lambda \tau$ for every permutation $\lambda \in \symm{k}$.

More generally, if $\ofo(\vect{k} \sigma) = \ofo(\vect{k} \tau)$ for some maps $\sigma \colon \nset{n} \to \nset{k}$ and $\tau \colon \nset{\ell} \to \nset{k}$, then $\ofo(\vect{k} \lambda \sigma) = \ofo(\vect{k} \lambda \tau)$ for every map $\lambda \colon \nset{k} \to \nset{k}$.
\end{remark}

\begin{remark}
\label{rem:restrord}
The range of ${\ofo}|_{A^n}$ equals the set
\[
A^\sharp_n := A^\sharp \cap \bigcup_{i = 1}^n A^i.
\]
Thus, only the restriction of $f^*$ to $A^\sharp_n$ is relevant in the composition $f^* \circ {\ofo}|_{A^n}$, and $f^* \circ {\ofo}|_{A^n} = g^* \circ {\ofo}|_{A^n}$ if and only if $f^*|_{A^\sharp_n} = g^*|_{A^\sharp_n}$. If $\card{A} \leq n$, then obviously $A^\sharp_n = A^\sharp$.
\end{remark}

\begin{remark}
A function $f \colon \{0, 1\}^n \to B$ is determined by the order of first occurrence if and only if there exist unary functions $\phi, \gamma \colon \{0, 1\} \to B$ such that
\[
f(a_1, \dots, a_n) =
\begin{cases}
\phi(a_1), & \text{if $a_1 = \dots = a_n$,} \\
\gamma(a_1), & \text{otherwise.}
\end{cases}
\]
\end{remark}

The following lemma shows that the intersection of the class of $2$-set-transitive functions and the class of functions that are determined by the order of first occurrence is precisely the class of functions determined by $\supp$. It is also easy to see that none of the two classes is included in the other.

Before stating the result, let us introduce a notational device that will be used many times in the sequel. We write expressions such as
\[
(\dots, \pos{i}{a}, \dots, \pos{j}{b}, \dots)
\qquad \text{or} \qquad
(a_1, \dots, \pos{i}{a}, \dots, \pos{j}{b}, \dots, a_n)
\]
to denote an $n$-tuple whose $i$-th component is $a$ and the $j$-th component is $b$. The remaining components are irrelevant to the argument at hand and they are clear from the context. The indices $i$ and $j$ are always distinct and they may be equal to $1$ or $n$, but it does not necessarily hold that $i < j$; however, if it is known that $i < j$, then we usually write the $i$-th component to the left of the $j$-th one, especially in arguments involving functions determined by the order of first occurrence. Also, whenever possible, we write components indexed by $i$ and $i + 1$ next to each other, and we write components indexed by $1$ or $n$ at the beginning and at the end of the tuple, respectively, as in the following:
\[
(\dots, \pos{i}{a}, \pos{\; i+1}{b}, \dots, \pos{\ell}{c}, \dots, \pos{n}{d}).
\]

\begin{lemma}
\label{lem:suppord}
Assume that $n > \card{A} + 1$, and let $f \colon A^n \to B$. Then the following conditions are equivalent:
\begin{enumerate}[\rm (i)]
\item\label{lem:suppord:ts}
$f$ is totally symmetric and determined by the order of first occurrence.

\item\label{lem:suppord:2st}
$f$ is $2$-set-transitive and determined by the order of first occurrence.

\item\label{lem:suppord:piIJ}
$f$ is determined by the order of first occurrence and for all $I, J \in \couples$, there exists a bijection $\pi_{IJ} \colon \nset{n-1} \to \nset{n-1}$ such that $\pi_{IJ}(\min J) = \min I$ and
$f(\vect{a} \delta_I) = f(\vect{a} \pi_{IJ} \delta_J)$ for all $\vect{a} \in A^{n-1}$.

\item\label{lem:suppord:supp}
$f$ is determined by $\supp$.
\end{enumerate}
\end{lemma}

\begin{proof}
We will prove the implications $\eqref{lem:suppord:ts} \implies \eqref{lem:suppord:2st} \implies \eqref{lem:suppord:piIJ} \implies \eqref{lem:suppord:ts}$ and $\eqref{lem:suppord:ts} \implies \eqref{lem:suppord:supp} \implies \eqref{lem:suppord:ts}$.

$\eqref{lem:suppord:ts} \implies \eqref{lem:suppord:2st}$:
Total symmetry implies $2$-set-transitivity.

$\eqref{lem:suppord:2st} \implies \eqref{lem:suppord:piIJ}$:
Assume that $f$ is $2$-set-transitive, and let $I, J \in \couples$. Then there exists a permutation $\sigma \in \Inv f$ such that $I \sigma^{-1} = J$. By Lemma~\ref{lem:hatsigma}, there exists a permutation $\hat{\sigma} \in \symm{n-1}$ such that $\hat{\sigma} \circ \delta_J = \delta_I \circ \sigma$ and $\hat{\sigma}(\min J) = \min I$. Setting $\pi_{IJ} := \hat{\sigma}$, we have $f(\vect{a} \delta_I) = f(\vect{a} \delta_I \sigma) = f(\vect{a} \pi_{IJ} \delta_J)$ for all $\vect{a} \in A^{n-1}$.

$\eqref{lem:suppord:piIJ} \implies \eqref{lem:suppord:ts}$:
Assume that condition \eqref{lem:suppord:piIJ} holds. Observe first that for all integers $k$ and $\ell$ such that $1 \leq k < \ell \leq n$ and for all $a_1, \dots, a_n \in A$ and for $b, c \in \{a_1, \dots, a_{k-1}, a_{k+1}, \dots, a_{\ell-1}, a_{\ell}, \dots, a_{n-1}\}$, by choosing $I := \{k, \ell\}$ and $J := \{n-1, n\}$, we have that
\begin{equation}
\label{eq:swapbc}
\begin{split}
& f(a_1, \dots, a_{k-1}, b, a_{k+1}, \dots, a_{\ell-1}, b, a_{\ell}, \dots, a_{n-1}) = \\
& f(a_{\pi_{IJ}(1)}, \dots, a_{\pi_{IJ}(n-2)}, b, b) = \\
& f(a_{\pi_{IJ}(1)}, \dots, a_{\pi_{IJ}(n-2)}, c, c) = \\
& f(a_1, \dots, a_{k-1}, c, a_{k+1}, \dots, a_{\ell-1}, c, a_{\ell}, \dots, a_{n-1}),
\end{split}
\end{equation}
where the first and third equalities hold by condition \eqref{lem:suppord:piIJ}, and the second equality holds because $f$ is determined by the order of first occurrence and both $b$ and $c$ occur among $a_{\pi_{IJ}(1)}, \dots, a_{\pi_{IJ}(n-2)}$.

We will show that $f$ is totally symmetric. To this end, it is sufficient to show that $\Inv f$ contains all adjacent transpositions $(m \;\: m + 1)$, $1 \leq m \leq n - 1$, i.e., for every $m \in \nset{n - 1}$,
\begin{equation}
\label{eq:mm+1}
f(a_1, \dots, a_n) =
f(a_1, \dots, a_{m-1}, a_{m+1}, a_m, a_{m+2}, \dots, a_n),
\end{equation}
for all $a_1, \dots, a_n \in A$.

Let $m \in \nset{n-1}$, and let $(a_1, \dots, a_n) \in A^n$. If $a_m = a_{m+1}$, then equality~\eqref{eq:mm+1} obviously holds, so let us assume that $a_m \neq a_{m+1}$; let $x := a_m$, $y := a_{m+1}$. Since $n > \card{A} + 1$, there exist indices $i < j$ and $i' < i$, $j' < j$ such that $\alpha := a_{i'} = a_i$ and $\beta := a_{j'} = a_j$. We need to consider several cases according to the order of elements $i$, $j$ and $m$.
In what follows, we will write
\begin{itemize}
\item ``$\stackrel{\ofo}{=}$'' to indicate that the equality holds because $f$ is determined by the order of first occurrence,
\item ``$\stackrel{pq}{=}$'', where $p, q \in \nset{n}$, to indicate that the equality holds by \eqref{eq:swapbc} for $I = \{p, q\}$.
\end{itemize}

\begin{asparaenum}[\it {Case} 1:]
\item $\{i, j\} \cap \{m, m+1\} = \emptyset$. We only give the details in the case when $i < m$, $m + 1 < j$. The other two cases ($i < j < m$; $m + 1 < i < j$) are proved in a similar way.
\begin{align*}
&
f(\dots, \pos{i}{\alpha}, \dots, \pos{m}{x}, \pos{\;\; m+1}{y}, \dots, \pos{j}{\beta}, \dots)
\stackrel{\ofo}{=}
f(\dots, \pos{i}{\alpha}, \dots, \pos{m}{x}, \pos{\;\; m+1}{y}, \dots, \pos{j}{\alpha}, \dots)
\stackrel{ij}{=} \\ &
f(\dots, \pos{i}{x}, \dots, \pos{m}{x}, \pos{\;\; m+1}{y}, \dots, \pos{j}{x}, \dots)
\stackrel{mj}{=}
f(\dots, \pos{i}{x}, \dots, \pos{m}{y}, \pos{\;\; m+1}{y}, \dots, \pos{j}{y}, \dots)
\stackrel{m+1, j}{=} \\ &
f(\dots, \pos{i}{x}, \dots, \pos{m}{y}, \pos{\;\; m+1}{x}, \dots, \pos{j}{x}, \dots)
\stackrel{ij}{=}
f(\dots, \pos{i}{\alpha}, \dots, \pos{m}{y}, \pos{\;\; m+1}{x}, \dots, \pos{j}{\alpha}, \dots)
\stackrel{\ofo}{=} \\ &
f(\dots, \pos{i}{\alpha}, \dots, \pos{m}{y}, \pos{\;\; m+1}{x}, \dots, \pos{j}{\beta}, \dots)
.
\end{align*}

\item $\{i, j\} \cap \{m, m+1\} \neq \emptyset$. Then $x$ or $y$ occurs before the $m$-th position and we clearly have
\[
f(a_1, \dots, \pos{m}{x}, \pos{\;\; m+1}{y}, \dots, a_n)
\stackrel{\ofo}{=}
f(a_1, \dots, \pos{m}{y}, \pos{\;\; m+1}{x}, \dots, a_n)
.
\]
\end{asparaenum}

We conclude that \eqref{eq:mm+1} holds for all $a_1, \dots, a_n \in A$, for any $m \in \nset{n}$, i.e., $\Inv f$ contains all adjacent transpositions $(m \;\: m + 1)$. This implies that $f$ is totally symmetric, as claimed.

$\eqref{lem:suppord:ts} \implies \eqref{lem:suppord:supp}$:
Assume that $f$ is determined by the order of first occurrence. Then there exists $f' \colon A^\sharp \to B$ such that $f = f' \circ {\ofo}|_{A^n}$. Since $f$ is totally symmetric, $f'(\vect{a}) = f(\vect{a} \sigma)$ for any permutation $\sigma$ of $\nset{r}$, for all $\vect{a} \in A^\sharp \cap A^r$, $r \geq 1$. Hence the function $f^* \colon \mathcal{P}(A) \to B$ given by setting $f^*(S) := f'(\vect{a})$, where $\vect{a}$ is any element of $A^\sharp$ such that $\supp(\vect{a}) = S$, is well defined. We have
\begin{multline*}
(f^* \circ \supp) (\vect{a}) =
f^*(\supp(\vect{a})) =
f^*(\supp(\ofo(\vect{a})) = \\
f'(\ofo(\vect{a})) =
(f' \circ \ofo) (\vect{a}) = 
f(\vect{a}),
\end{multline*}
for all $\vect{a} \in A^n$. Thus, $f = f^* \circ {\supp}|_{A^n}$, i.e., $f$ is determined by $\supp$.

$\eqref{lem:suppord:supp} \implies \eqref{lem:suppord:ts}$:
Assume that $f$ is determined by $\supp$. Then $f$ is totally symmetric. Furthermore, $f = f^* \circ {\supp}|_{A^n}$ for some $f^* \colon \mathcal{P}(A) \to B$. Define $f' \colon A^\sharp \to B$ as $f'(a_1, \dots, a_r) := f^*(\{a_1, \dots, a_r\})$, for all $(a_1, \dots, a_r) \in A^\sharp$. Then
\begin{multline*}
(f' \circ \ofo) (\vect{a}) = 
f'(\ofo(\vect{a})) =
f^*(\supp(\ofo(\vect{a})) = \\
f^*(\supp(\vect{a})) =
(f^* \circ \supp) (\vect{a}) =
f(\vect{a}),
\end{multline*}
for all $\vect{a} \in A^n$. Thus, $f = f' \circ {\ofo}|_{A^n}$, i.e., $f$ is determined by the order of first occurrence.
\end{proof}

\begin{proposition}
\label{prop:minord}
Let $f^* \colon A^\sharp \to B$, and let $f \colon A^n \to B$.
\begin{enumerate}[\rm (i)]
\item\label{prop:minord:item1}
If $f = f^* \circ {\ofo}|_{A^n}$, then $f_I = f^* \circ {\ofo}|_{A^{n-1}}$ for all $I \in \couples$.

\item\label{prop:minord:item2}
If $f_I = f^* \circ {\ofo}|_{A^{n-1}}$ for all $I \in \couples$, then there exists a map $h^* \colon A^\sharp \to B$ such that $f = h^* \circ {\ofo}|_{A^n}$ and $f^*(\vect{a}) = h^*(\vect{a})$ whenever $\card{\supp(\vect{a})} < n$.

\item\label{prop:minord:item3}
If $n > \card{A}$ and $f_I = f^* \circ {\ofo}|_{A^{n-1}}$ for all $I \in \couples$, then $f = f^* \circ {\ofo}|_{A^n}$.

\item\label{prop:minord:item4}
If $f$ is weakly determined by the order of first occurrence, then $f$ has a unique identification minor.
\end{enumerate}
\end{proposition}

\begin{proof}
\eqref{prop:minord:item1}
For any $I \in \couples$ and for any $\vect{a} \in A^{n-1}$, we have
\[
f_I(\vect{a}) =
f(\vect{a} \delta_I) =
f^* \circ \ofo (\vect{a} \delta_I) =
f^* \circ \ofo (\vect{a}).
\]
Therefore $f_I = f^* \circ {\ofo}|_{A^{n-1}}$ for all $I \in \couples$.

\eqref{prop:minord:item2}
Assume that $f_I = f^* \circ {\ofo}|_{A^{n-1}}$ for all $I \in \couples$. Define $h^* \colon A^\sharp \to B$ as follows:
\[
h^*(\vect{a}) =
\begin{cases}
f^*(\vect{a}), & \text{if $\card{\supp(\vect{a})} < n$,} \\
f(\vect{a}), & \text{if $\card{\supp(\vect{a})} = n$.}
\end{cases}
\]
Let $\vect{b} \in A^n$. If $\vect{b} = \vect{a} \delta_I$ for some $\vect{a} \in A^{n-1}$ and $I \in \couples$, then $\card{\supp(\vect{a})} < n$ and we have
\[
f(\vect{b}) = f(\vect{a} \delta_I) = f_I(\vect{a}) = f^*(\ofo(\vect{a})) = h^*(\ofo(\vect{a})) = h^*(\ofo(\vect{a} \delta_I)) = h^*(\ofo(\vect{b})).
\]
Otherwise $\vect{b}$ has no repeated entries, and we have $\ofo(\vect{b}) = \vect{b}$, $\card{\supp(\vect{b})} = n$, and
\[
f(\vect{b}) = h^*(\vect{b}) = h^*(\ofo(\vect{b})).
\]
Thus, we have $f = h^* \circ {\ofo}|_{A^n}$. By the definition of $h^*$, we also have $f^*(\vect{a}) = h^*(\vect{a})$ whenever $\card{\supp(\vect{a})} < n$.

\eqref{prop:minord:item3}
If $n > \card{A}$, then the function $h^*$ constructed in part \eqref{prop:minord:item2} equals $f^*$, and the claim follows.

\eqref{prop:minord:item4}
Follows immediately from part~\eqref{prop:minord:item1}.
\end{proof}

As the following example illustrates, the equalities involving $\ofo$ that occur in the conditions of Proposition~\ref{prop:minord}~\eqref{prop:minord:item2} and~\eqref{prop:minord:item3} cannot be relaxed to equivalences, at least not in the case when $n = \card{A} + 1$. (However, as we will see in Corollary~\ref{cor:gofo}, the equalities can be relaxed to equivalences if $\card{A} = 2$ and $n \geq 4$.) We apply again the scheme of Definition~\ref{def:fGPphi} to show that if $n = k + 1$ and $f \colon A^n \to B$ is a function such that $f_I \equiv f^* \circ {\ofo}|_{A^{n-1}}$ for all $I \in \couples$, then $f$ is not necessarily weakly determined by the order of first occurrence.

\begin{example}
\label{ex:notofonor2settransitive}
Assume that $n = k + 1$, let $A = \nset{k}$, and assume that $\alpha$ and $\beta$ are distinct elements of $B$. Define the function $h \colon A^k \to B$ by the rule
\[
h(\vect{a}) =
\begin{cases}
\alpha, & \text{if $\vect{a} = \vect{k}$,} \\
\beta, & \text{otherwise.}
\end{cases}
\]
For $I \in \couples$, let $g^I = h$ and let $\rho_I = (1 \; 2 \; 3 \; \cdots \; k)^i = (i, i+1, \dots, k, 1, \dots, i-1)$, where $i = \min I$. Let $\phi$ be the identity map on $\couples$. Denote $G := (g^I)_{I \in \couples}$, $P := (\rho_I)_{I \in \couples}$. Let $f \colon A^n \to B$ be the function $f_{G,P,\phi}$ as in Definition~\ref{def:fGPphi}. Then $f(\vect{b}) = \alpha$ if and only if $\vect{b} = \vect{k} \rho_I \delta_I$ for some $I \in \couples$.

For example, if $k = 4$ and $n = 5$, then $f$ is the function that takes on value $\alpha$ at points
\begin{align*}
&
(1, 1, 2, 3, 4),
(1, 2, 1, 3, 4),
(1, 2, 3, 1, 4),
(1, 2, 3, 4, 1),
(4, 1, 1, 2, 3), \\ &
(4, 1, 2, 1, 3),
(4, 1, 2, 3, 1),
(3, 4, 1, 1, 2),
(3, 4, 1, 2, 1),
(2, 3, 4, 1, 1)
\end{align*}
and value $\beta$ elsewhere.

We claim that $f$ is not weakly determined by the order of first occurrence. To see this, suppose on the contrary that $f(\vect{b}) = f^*(\ofo(\vect{b} \sigma))$ for some $f^* \colon A^\sharp \to B$ and $\sigma \in \symm{k}$. Let $I \in \couples$, and let $\vect{c}$ be the unique tuple that has two occurrences of $2$ and satisfies $\ofo(\vect{c} \sigma) = \ofo(\vect{k} \rho_I \delta_I \sigma)$. Then we have
\[
\beta = f(\vect{c}) = f^*(\ofo(\vect{c} \sigma)) = f^*(\ofo(\vect{k} \rho_I \delta_I \sigma)) = f(\vect{k} \rho_I \delta_I) = \alpha,
\]
a contradiction.

We also claim that if $k > 2$, then the only permutation under which $f$ is invariant is the identity permutation (and hence, in particular, $f$ is not $2$-set-transitive). To see this, let $\sigma \in \symm{n}$ and assume that $f(\vect{b}) = f(\vect{b} \sigma)$ for all $\vect{b} \in A^n$. Then $\vect{b} \mapsto \vect{b} \sigma$ maps the set $\{\vect{k} \rho_I \delta_I : I \in \couples\}$ onto itself. Let $I = \{1, 2\}$, and let $J$ be the unique couple in $\couples$ such that $\vect{k} \rho_I \delta_I \sigma = \vect{k} \rho_J \delta_J$. Assume that $J = \{p, q\}$ with $p < q$.

Suppose first that $p \geq 2$ and $p + 1 < q < n$. Then
\[
(1, 1, 2, 3, \dots, k-1, k) \sigma = (\dots, \pos{p-1}{k}, \pos{p}{1}, \pos{p+1}{2}, \dots, p-q, \pos{q}{1}, p-q+1, \dots),
\]
and, depending on whether $\sigma(p) = 1$ and $\sigma(q) = 2$, or $\sigma(p) = 2$ and $\sigma(q) = 1$, it holds that $(k, 1, 2, 3, \dots, k-1, 1) \sigma$ equals 
\[
(\dots, \pos{p-1}{1}, \pos{p}{k}, \pos{p+1}{2}, \dots, p-q, \pos{q}{1}, p-q+1, \dots) 
\quad\text{or}\quad
(\dots, \pos{p-1}{1}, \pos{p}{1}, \pos{p+1}{2}, \dots, p-q, \pos{q}{k}, p-q+1, \dots),
\]
and in both cases we arrive at a contradiction.

Suppose then that $p = 2$ and $q = n$. Then
\[
(1, 1, 2, 3, \dots, k-1, k) \sigma = (k, 1, 2, \dots, k-1, 1).
\]
Thus $\sigma$ fixes all elements in $\{3, \dots, n-1\}$ and $\sigma(1) = n$, and we have that either $\sigma = (1 \; n)$ or $\sigma = (1 \; n \; 2)$.
If $n > 4$, then
\begin{gather*}
(k-1, k, 1, 1, 2, \dots, k-3, k-2) (1 \; n) = (k-2, k, 1, 1, 2, \dots, k-3, k-1), \\
(k-1, k, 1, 1, 2, \dots, k-3, k-2) (1 \; n \; 2) = (k-2, k-1, 1, 1, 2, \dots, k-3, k),
\end{gather*}
and both possibilities for $\sigma$ give rise to a contradiction.
If $n = 4$, then
\[
(3, 1, 1, 2) (1 \; 4) = (2, 1, 1, 3),
\qquad
(1, 2, 3, 1) (1 \; 4 \; 2) = (1, 1, 3, 2),
\]
and we arrive again at a contradiction.

Suppose then that $p \geq 2$ and $q = p + 1$.
Then
\[
(1, 1, 2, 3, \dots, k-1, k) \sigma = (\dots, \pos{p-1}{k}, \pos{p}{1}, \pos{p+1}{1}, 2, \dots),
\]
and, depending on whether $\sigma(p) = 1$ and $\sigma(p+1) = 2$, or $\sigma(p) = 2$ and $\sigma(p+1) = 2$, it holds that
$(k, 1, 2, 3, \dots, k-1, 1) \sigma$ equals
\[
(\dots, \pos{p-1}{1}, \pos{p}{k}, \pos{p+1}{1}, 2, \dots)
\qquad \text{or} \qquad
(\dots, \pos{p-1}{1}, \pos{p}{1}, \pos{p+1}{k}, 2, \dots),
\]
and in both cases we arrive at a contradiction.

Suppose then that $p = 1$ and $q \geq 4$. Then
\[
(1, 1, 2, 3, \dots, k-1, k) \sigma = (\pos{1}{1}, 2, 3, \dots, q-1, \pos{q}{1}, q, \dots, k),
\]
and, depending on whether $\sigma(1) = 1$ and $\sigma(q) = 2$, or $\sigma(1) = 2$ and $\sigma(q) = 1$, it holds that $(1, 2, 1, 3, \dots, k-1, k) \sigma$ equals
\[
(\pos{1}{1}, 1, 3, \dots, q-1, \pos{q}{2}, q, \dots, k)
\qquad \text{or} \qquad
(\pos{1}{2}, 1, 3, \dots, q-1, \pos{q}{1}, q, \dots, k),
\]
and in both cases we arrive at a contradiction.

Suppose then that $p = 1$ and $q = 3$. Then
\[
(1, 1, 2, 3, \dots, k-1, k) \sigma = (1, 2, 1, 3, 4, \dots, k-1, k),
\]
and, depending on whether $\sigma(1) = 1$ and $\sigma(3) = 2$, or $\sigma(1) = 2$ and $\sigma(3) = 1$, it holds that
$(k, 1, 2, 3, \dots, k-1, 1) \sigma$ equals
\[
(k, 2, 1, 3, 4, \dots, k-1, 1)
\qquad \text{or} \qquad
(1, 2, k, 3, 4, \dots, k-1, 1),
\]
and in both cases we arrive at a contradiction.

Finally, suppose that $p = 1$ and $q = 2$. Then
\[
(1, 1, 2, 3, \dots, k-1, k) \sigma = (1, 1, 2, 3, \dots, k-1, k),
\]
and we have that either $\sigma = (1 \; 2)$ or $\sigma$ is the identity permutation. If $\sigma = (1 \; 2)$, then
\[
(1, 2, 1, 3, \dots, n-1) \sigma = (2, 1, 1, 3, \dots, n-1),
\]
a contradiction.
The only remaining possibility is that $\sigma$ is the identity permutation.
\end{example}

\subsection{Open problems}

Functions with a unique identification minor are a topic that is worth investigating on its own right. We have seen in this section examples of wide classes of functions with a unique identification minor: $2$-set-transitive functions and functions weakly determined by the order of first occurrence (see Propositions~\ref{prop:2trans} and~\ref{prop:minord}). Example~\ref{ex:notofonor2settransitive} shows that
\begin{itemize}
\item if $n = \card{A} + 1$, then there exist functions $f \colon A^n \to B$ such that $f$ has a unique identification minor and $f$ is not $2$-set-transitive nor weakly determined by the order of first occurrence; and
\item if $n = \card{A} + 1$, then there exist functions $f \colon A^n \to B$ such that $f$ is determined by the order of first occurrence and $f$ is not reconstructible.
\end{itemize}
Whether the above statements remain true also in the case when $n > \card{A} + 1$, however, remains an open problem.

\begin{question}
Describe all functions that have a unique identification minor.
\end{question}

\begin{question}
\label{que:ofo}
Let $\mathcal{C}$ be the class of functions $f \colon A^n \to B$ such that $n > \card{A} + 1$ and $f$ is weakly determined by the order of first occurrence. Is $\mathcal{C}$ recognizable?
\end{question}

We will see in Corollary~\ref{cor:Booleanwofo-reconstructible} that the answer to Question~\ref{que:ofo} is positive if $\card{A} = 2$.

%%%%%%%%%%%%%%%%%%%%%%%%%%%%%%%%%%%%%%%%%%%%%%%%%%

\section{Totally symmetric functions are reconstructible}
\label{sec:symmetric}

In this section, we investigate totally symmetric functions, an especially important class of functions with a unique identification minor. Our main result, Theorem~\ref{thm:symmetric-reconstructible}, asserts that totally symmetric functions $f \colon A^n \to B$ are reconstructible, provided $n \geq \card{A} + 2$. We also show that totally symmetric functions $f \colon A^n \to B$ are weakly reconstructible if $n > \max(\card{A}, 3)$ (Proposition~\ref{prop:symmetric-weakly-reconstructible}). The lower bound of Theorem~\ref{thm:symmetric-reconstructible} is sharp: in Example~\ref{ex:A+1} we have seen totally symmetric functions of arity $n = \card{A} + 1$ that are not reconstructible.

Let $\msupp \colon \bigcup_{n \geq 1} A^n \to \mathcal{M}(A)$ be the map given by the rule $\vect{a} \mapsto \{a_i : i \in \nset{n}\}$ for every $\vect{a} \in A^n$ ($n \geq 1$), i.e., $\msupp$ maps each tuple to the multiset of its entries.
It is easy to verify that a function $f \colon A^n \to B$ is totally symmetric if and only if $f = \phi \circ {\msupp}|_{A^n}$ for some $\phi \colon \mathcal{M}_n(A) \to B$.

\begin{theorem}
\label{thm:symmetric-reconstructible}
Assume that $n \geq k + 2$ and $\card{A} = k$. If $f \colon A^n \to B$ is totally symmetric, then $f$ is reconstructible.
\end{theorem}

\begin{proof}
Since $f$ is totally symmetric, there exists a map $\phi \colon \mathcal{M}_n(A) \to B$ such that $f = \phi \circ {\msupp}|_{A^n}$. Let $h \colon A^{n-1} \to B$ be the function given by the rule $h(\vect{b}) = \phi(\msupp(\vect{b}) \uplus \{b_1\})$ for all $\vect{b} \in A^{n-1}$. Then for every $I \in \couples$, it holds that $f_I \equiv h$.
Let $g \colon A^n \to B$ be a reconstruction of $f$. Then for every $I \in \couples$, it holds that $g_I \equiv h$ and hence there exists a permutation $\rho_I \in \symm{n-1}$ such that $g_I(\vect{b}) = h(\vect{b} \rho_I)$ for all $\vect{b} \in A^{n-1}$. Let $q_I := \min \delta_I^{-1}(\rho_I(1))$.

Let $\vect{a} \in A^n$. Since $n \geq k + 2$, there exist $I \in \couples$ and $\vect{b} \in A^{n-1}$ such that $\vect{a} = \vect{b} \delta_I$. It holds that $g(\vect{a}) = g(\vect{b} \delta_I) = g_I(\vect{b}) = h(\vect{b} \rho_I) = \phi(\msupp(\vect{b} \rho_I) \uplus \{b_{\rho_I(1)}\})$. Since $a_i = b_{\delta_I(i)}$ for every $i \in \nset{n}$ and $\delta_I(q_I) = \rho_I(1)$, we have $b_{\rho_I(1)} = a_{q_I}$. Therefore, for any $I \in \couples$ and for every $\vect{a} \in A^n$ such that $a_{\min I} = a_{\max I}$, it holds that
\begin{equation}
\label{eq:gvalues}
g(\vect{a}) =
\phi(\msupp(\vect{a}) \setminus \{a_{\max I}\} \uplus \{a_{q_I}\}).
\end{equation}
(In the sequel, we write ``$\stackrel{J}{=}$'' for some $J \in \couples$ to indicate that the equality in question holds by Equation~\eqref{eq:gvalues} for $I = J$.)

\begin{claim}
\label{clm:IJdisjoint1}
If there exist $I, J \in \couples$ such that $I \cap J = \emptyset$ and $q_I \notin I$ and $q_J \notin J$, then $\phi(S) = \phi(T)$ for all $S, T \in \mathcal{M}_n(A)$ such that $\set(S) = \set(T)$.
\end{claim}

\begin{pfclaim}[Proof of Claim~\ref{clm:IJdisjoint1}]
Let $I = \{i, j\}$, $J = \{p, q\}$, and assume that $I \cap J = \emptyset$ and $q_I \notin I$ and $q_J \notin J$. We split the analysis into several cases.

\begin{asparaenum}[\it {Case} 1:]
\item\label{clm:IJdisjoint1:case1}
$q_I, q_J \notin I \cup J$, $q_I \neq q_J$. In this case $n \geq 6$.
We have for any $\alpha, \beta, \gamma, \delta \in A$ and for any $\vect{u} \in A^{n - 6}$ that
\[
\phi(\{\alpha, \alpha, \beta, \gamma, \delta, \delta\} \uplus \msupp(\vect{u}))
\stackrel{I}{=} g(\pos{q_I}{\alpha}, \pos{i}{\beta}, \pos{j}{\beta}, \pos{q_J}{\gamma}, \pos{p}{\delta}, \pos{q}{\delta}, \vect{u}) \stackrel{J}{=}
\phi(\{\alpha, \beta, \beta, \gamma, \gamma, \delta\} \uplus \msupp(\vect{u})).
\]
Let $K = \{q_I, q_J\}$. If $q_K \in K$, then for any $\alpha, \beta, \gamma, \delta \in A$ and for any $\vect{u} \in A^{n-6}$,
\begin{multline*}
\phi(\{\alpha, \alpha, \alpha, \beta, \gamma, \delta\} \uplus \msupp(\vect{u}))
\stackrel{I}{=} g(\pos{q_I}{\alpha}, \pos{i}{\delta}, \pos{j}{\delta}, \pos{q_J}{\alpha}, \pos{p}{\beta}, \pos{q}{\gamma}, \vect{u}) \stackrel{K}{=} \\
\phi(\{\alpha, \alpha, \beta, \gamma, \delta, \delta\} \uplus \msupp(\vect{u}))
= \phi(\{\alpha, \beta, \beta, \gamma, \gamma, \delta\} \uplus \msupp(\vect{u})).
\end{multline*}
If $q_K \notin K \cup I$, then for any $\alpha, \beta, \gamma \in A$ and for any $\vect{u} \in A^{n-5}$,
\[
\phi(\{\alpha, \alpha, \alpha, \beta, \gamma, \msupp(\vect{u})\})
\stackrel{I}{=} g(\pos{q_I}{\alpha}, \pos{i}{\beta}, \pos{j}{\beta}, \pos{q_J}{\alpha}, \pos{q_K}{\gamma}, \vect{u}) \stackrel{K}{=}
\phi(\{\alpha, \beta, \beta, \gamma, \gamma\} \uplus \msupp(\vect{u})).
\]
If $q_K \notin K \cup J$, then for any $\alpha, \beta, \gamma \in A$ and for any $\vect{u} \in A^{n-5}$,
\[
\phi(\{\alpha, \alpha, \alpha, \beta, \gamma, \msupp(\vect{u})\})
\stackrel{J}{=} g(\pos{q_I}{\alpha}, \pos{q_J}{\alpha}, \pos{p}{\beta}, \pos{q}{\beta}, \pos{q_K}{\gamma}, \vect{u}) \stackrel{K}{=}
\phi(\{\alpha, \beta, \beta, \gamma, \gamma\} \uplus \msupp(\vect{u})).
\]

Thus, for all $\alpha, \beta, \gamma, \delta \in A$ and for all $\vect{u} \in A^{n-6}$, it holds that
\begin{multline}
\label{eq:alphabetagammadelta}
\phi(\{\alpha, \alpha, \alpha, \beta, \gamma, \delta\} \uplus \msupp(\vect{u})) =
\phi(\{\alpha, \alpha, \beta, \gamma, \delta, \delta\} \uplus \msupp(\vect{u})) = \\
\phi(\{\alpha, \beta, \beta, \gamma, \gamma, \delta\} \uplus \msupp(\vect{u})).
\end{multline}

Let $S \in \mathcal{M}_n(A)$, let $E := \set(S)$, and fix an element $e \in E$. Let $F$ be the multiset on $A$ given by the multiplicity function
\[
\mathbf{1}_F(x) =
\begin{cases}
\card{E} - n + 1, & \text{if $x = e$,} \\
1, & \text{if $x \in E \setminus \{e\}$,} \\
0, & \text{if $x \notin E$.}
\end{cases}
\]
We will construct a sequence $S = S_0, S_1, \dots, S_r = F$ ($r \geq 1$) of multisets in $\mathcal{M}_n(A)$ that satisfy $\phi(S_i) = \phi(S_{i-1})$ for all $i \in \nset{r}$. Let $S_0 := S$, and define $S_1$ by the following rules.
\begin{itemize}
\item If $\mathbf{1}_S(e) > 1$, then let $S_1 := S_0$. In this case obviously $\phi(S_0) = \phi(S_1)$.
\item If $\mathbf{1}_S(e) = 1$ and there exists $a \in A$ such that $\mathbf{1}_S(a) \geq 3$, then let $S_1 := S_0 \uplus \{e\} \setminus \{a\}$. Then $\phi(S_0) = \phi(S_1)$ by \eqref{eq:alphabetagammadelta} (take $\alpha = a$ and $\delta = e$, and consider the first and second expression in \eqref{eq:alphabetagammadelta}).
\item Otherwise we have that $\mathbf{1}_S(e) = 1$ and, since $n \geq \card{A} + 2$, there exist distinct elements $a$ or $b$ of $A$ such that $\mathbf{1}_S(a) = \mathbf{1}_S(b) = 2$. Let $S_1 := S_0 \uplus \{e, e\} \setminus \{a, b\}$. Then $\phi(S_0) = \phi(S_1)$ by \eqref{eq:alphabetagammadelta} (take $\alpha = e$, $\beta = a$, and $\gamma = b$, and consider the first and third expression in \eqref{eq:alphabetagammadelta}).
\end{itemize}
Thus $S_1$ is a multiset with $\mathbf{1}_{S_1}(e) > 1$. We proceed by the following recursion.
\begin{itemize}
\item If $i \geq 1$ and there exists $a \in A$ such that $\mathbf{1}_{S_i}(a) > 1$ and $a \neq e$, then let $S_{i+1} := S_i \uplus \{e\} \setminus \{a\}$. Then $\phi(S_i) = \phi(S_{i+1})$ by \eqref{eq:alphabetagammadelta} (take $\alpha = e$ and $\delta = a$, and consider the first and second expression in \eqref{eq:alphabetagammadelta}). Furthermore, $\mathbf{1}_{S_{i+1}}(e) > 1$, and we can apply the recursive step again.
\item Otherwise $S_i = U$, and we let $r := i$ and stop the recursion.
\end{itemize}
The recursion will stop after a finite number of steps, and we have that $\phi(S) = \phi(S_0) = \phi(S_1) = \dots = \phi(S_r) = \phi(F)$. We conclude that if $S, T \in \mathcal{M}_N(A)$ are multisets such that $\set(S) = \set(T)$, then we have $\phi(S) = \phi(F) = \phi(T)$, as claimed.

\item $q_I, q_J \notin I \cup J$, $q_I = q_J$. In this case $n \geq 5$.
We have for any $\alpha, \beta, \gamma \in A$ and for any $\vect{u} \in A^{n - 5}$ that
\[
\phi(\{\alpha, \alpha, \beta, \gamma, \gamma\} \uplus \msupp(\vect{u}))
\stackrel{I}{=}
g(\pos{q_I = q_J \;\;\;\;}{\alpha}, \pos{i}{\beta}, \pos{j}{\beta}, \pos{p}{\gamma}, \pos{q}{\gamma}, \vect{u})
\stackrel{J}{=}
\phi(\{\alpha, \alpha, \beta, \beta, \gamma\} \uplus \msupp(\vect{u})).
\]

Let $K = \{q_I, p\}$. If $q_K \in K$, then for all $\alpha, \beta \in A$ and for all $\vect{u} \in A^{n-4}$,
\[
\phi(\{\alpha, \alpha, \alpha, \beta\} \uplus \msupp(\vect{u}))
\stackrel{I}{=} 
g(\pos{q_I}{\alpha}, \pos{i}{\beta}, \pos{j}{\beta}, \pos{p}{\alpha}, \vect{u})
\stackrel{K}{=}
\phi(\{\alpha, \alpha, \beta, \beta\} \uplus \msupp(\vect{u})).
\]
If $q_K \notin J \cup K$, then for all $\alpha, \beta \in A$ and for all $\vect{u} \in A^{n-4}$,
\[
\phi(\{\alpha, \alpha, \alpha, \beta\} \uplus \msupp(\vect{u}))
\stackrel{J}{=}
g(\pos{q_I}{\alpha}, \pos{p}{\alpha}, \pos{q}{\alpha}, \pos{q_K}{\beta}, \vect{u})
\stackrel{K}{=}
\phi(\{\alpha, \alpha, \beta, \beta\} \uplus \msupp(\vect{u})).
\]
If $q_K \notin I \cup K$, then for all $\alpha, \beta, \gamma \in A$ and for all $\vect{u} \in A^{n-5}$,
\[
\phi(\{\alpha, \alpha, \alpha, \beta, \gamma\} \uplus \msupp(\vect{u}))
\stackrel{I}{=}
g(\pos{q_I}{\alpha}, \pos{i}{\beta}, \pos{j}{\beta}, \pos{p}{\alpha}, \pos{q_K}{\gamma}, \vect{u})
\stackrel{K}{=}
\phi(\{\alpha, \beta, \beta, \gamma, \gamma\} \uplus \msupp(\vect{u})).
\]
Thus, for all $\alpha, \beta, \gamma \in A$ and for all $\vect{u} \in A^{n-5}$, it holds that
\begin{multline*}
\phi(\{\alpha, \alpha, \alpha, \beta, \gamma\} \uplus \msupp(\vect{u})) =
\phi(\{\alpha, \alpha, \beta, \beta, \gamma\} \uplus \msupp(\vect{u})) = \\
\phi(\{\alpha, \alpha, \beta, \gamma, \gamma\} \uplus \msupp(\vect{u})) =
\phi(\{\alpha, \beta, \beta, \gamma, \gamma\} \uplus \msupp(\vect{u})).
\end{multline*}
Proceeding in a similar way as in Case~\ref{clm:IJdisjoint1:case1}, we can show that the above identities imply that $\phi(S) = \phi(T)$ for all $S, T \in \mathcal{M}_n(A)$ such that $\set(S) = \set(T)$.

\item $q_I \notin I \cup J$, $q_J \in I$. In this case $n \geq 5$.
We have for any $\alpha, \beta, \gamma \in A$ and for any $\vect{u} \in A^{n - 5}$ that
\[
\phi(\{\alpha, \alpha, \beta, \gamma, \gamma\} \uplus \msupp(\vect{u}))
\stackrel{I}{=}
g(\pos{q_I}{\alpha}, \pos{i}{\beta}, \pos{j}{\beta}, \pos{p}{\gamma}, \pos{q}{\gamma}, \vect{u})
\stackrel{J}{=}
\phi(\{\alpha, \beta, \beta, \beta, \gamma\} \uplus \msupp(\vect{u})).
\]
Let $K = \{q_I, p\}$. If $q_K \in K$, then for all $\alpha, \beta, \gamma \in A$ and for all $\vect{u} \in A^{n-5}$,
\[
\phi(\{\alpha, \alpha, \alpha, \beta, \gamma\} \uplus \msupp(\vect{u}))
\stackrel{I}{=}
g(\pos{q_I}{\alpha}, \pos{i}{\beta}, \pos{j}{\beta}, \pos{p}{\alpha}, \pos{q}{\gamma}, \vect{u})
\stackrel{K}{=}
\phi(\{\alpha, \alpha, \beta, \beta, \gamma\} \uplus \msupp(\vect{u})).
\]
If $q_K \in I$, then for all $\alpha, \beta, \gamma \in A$ and for all $\vect{u} \in A^{n-5}$,
\[
\phi(\{\alpha, \alpha, \alpha, \beta, \gamma\} \uplus \msupp(\vect{u}))
\stackrel{I}{=}
g(\pos{q_I}{\alpha}, \pos{i}{\beta}, \pos{j}{\beta}, \pos{p}{\alpha}, \pos{q}{\gamma}, \vect{u})
\stackrel{K}{=}
\phi(\{\alpha, \beta, \beta, \beta, \gamma\} \uplus \msupp(\vect{u})).
\]
If $q_K = q$, then for all $\alpha, \beta, \gamma \in A$ and for all $\vect{u} \in A^{n-5}$,
\[
\phi(\{\alpha, \alpha, \beta, \gamma, \gamma\} \uplus \msupp(\vect{u}))
\stackrel{J}{=}
g(\pos{q_I}{\alpha}, \pos{i}{\beta}, \pos{j}{\gamma}, \pos{p}{\alpha}, \pos{q}{\alpha}, \vect{u})
\stackrel{K}{=}
\phi(\{\alpha, \alpha, \alpha, \beta, \gamma\} \uplus \msupp(\vect{u})).
\]
If $q_K \notin I \cup J \cup \{q_I\}$, then for all $\alpha, \beta, \gamma \in A$ and for all $\vect{u} \in A^{n-5}$,
\[
\phi(\{\alpha, \alpha, \beta, \beta, \gamma\} \uplus \msupp(\vect{u}))
\stackrel{J}{=}
g(\pos{q_I}{\alpha}, \pos{q_J}{\beta}, \pos{p}{\alpha}, \pos{q}{\alpha}, \pos{q_K}{\gamma}, \vect{u})
\stackrel{K}{=}
\phi(\{\alpha, \alpha, \beta, \gamma, \gamma\} \uplus \msupp(\vect{u})).
\]
Thus, for all $\alpha, \beta, \gamma \in A$ and for all $\vect{u} \in A^{n-5}$, it holds that
\begin{align*}
&
\phi(\{\alpha, \alpha, \alpha, \beta, \gamma\} \uplus \msupp(\vect{u})) =
\phi(\{\alpha, \beta, \beta, \beta, \gamma\} \uplus \msupp(\vect{u})) = \\ &
\phi(\{\alpha, \beta, \gamma, \gamma, \gamma\} \uplus \msupp(\vect{u})) =
\phi(\{\alpha, \alpha, \beta, \beta, \gamma\} \uplus \msupp(\vect{u})) = \\ &
\phi(\{\alpha, \alpha, \beta, \gamma, \gamma\} \uplus \msupp(\vect{u})) =
\phi(\{\alpha, \beta, \beta, \gamma, \gamma\} \uplus \msupp(\vect{u})).
\end{align*}
As in the previous cases, we can show that $\phi(S) = \phi(T)$ for all $S, T \in \mathcal{M}_n(A)$ such that $\set(S) = \set(T)$.

\item\label{clm:IJdisjoint1:case4}
$q_I \in J$, $q_J \in I$. Without loss of generality, we may assume that $q_I = p$ and $q_J = j$.
We have for any $\alpha, \beta \in A$ and for any $\vect{u} \in A^{n-4}$,
\[
\phi(\{\alpha, \beta, \beta, \beta\} \uplus \msupp(\vect{u}))
\stackrel{I}{=}
g(\pos{i}{\alpha}, \pos{j}{\alpha}, \pos{p}{\beta}, \pos{q}{\beta}, \vect{u})
\stackrel{J}{=}
\phi(\{\alpha, \alpha, \alpha, \beta\} \uplus \msupp(\vect{u})).
\]
Let $K = \{i, q\}$. If $q_K \in \{i, j, q\}$, then for all $\alpha, \beta \in A$ and for all $\vect{u} \in A^{n-4}$,
\[
\phi(\{\alpha, \alpha, \beta, \beta\} \uplus \msupp(\vect{u}))
\stackrel{I}{=}
g(\pos{i}{\alpha}, \pos{j}{\alpha}, \pos{p}{\beta}, \pos{q}{\alpha}, \vect{u})
\stackrel{K}{=}
\phi(\{\alpha, \alpha, \alpha, \beta\} \uplus \msupp(\vect{u})).
\]
If $q_K \in \{i, p, q\}$, then for all $\alpha, \beta \in A$ and for all $\vect{u} \in A^{n-4}$,
\[
\phi(\{\alpha, \alpha, \beta, \beta\} \uplus \msupp(\vect{u}))
\stackrel{J}{=}
g(\pos{i}{\alpha}, \pos{j}{\beta}, \pos{p}{\alpha}, \pos{q}{\alpha}, \vect{u})
\stackrel{K}{=}
\phi(\{\alpha, \alpha, \alpha, \beta\} \uplus \msupp(\vect{u})).
\]
If $q_K \notin \{i, j, p, q\}$, then $n \geq 5$. Let $L = \{q_K, i\}$. If $q_L \in \{q_K, i, j\}$, then for all $\alpha, \beta, \gamma \in A$ and for all $\vect{u} \in A^{n-5}$,
\[
\phi(\{\alpha, \alpha, \beta, \beta, \gamma\} \uplus \msupp(\vect{u}))
\stackrel{I}{=}
g(\pos{q_K}{\alpha}, \pos{i}{\alpha}, \pos{j}{\alpha}, \pos{p}{\beta}, \pos{q}{\gamma}, \vect{u})
\stackrel{L}{=}
\phi(\{\alpha, \alpha, \alpha, \beta, \gamma\} \uplus \msupp(\vect{u})).
\]
If $q_L \in \{p, q\}$, then for all $\alpha, \beta, \gamma \in A$ and for all $\vect{u} \in A^{n-5}$,
\[
\phi(\{\alpha, \alpha, \beta, \beta, \gamma\} \uplus \msupp(\vect{u}))
\stackrel{J}{=}
g(\pos{q_K}{\alpha}, \pos{i}{\alpha}, \pos{j}{\beta}, \pos{p}{\gamma}, \pos{q}{\gamma}, \vect{u})
\stackrel{L}{=}
\phi(\{\alpha, \beta, \gamma, \gamma, \gamma\} \uplus \msupp(\vect{u})).
\]
If $q_L \notin \{i, j, p, q\}$, then for all $\alpha, \beta \in A$ and for all $\vect{u} \in A^{n-4}$,
\[
\phi(\{\alpha, \alpha, \alpha, \beta\} \uplus \msupp(\vect{u}))
\stackrel{K}{=}
g(\pos{q_K}{\alpha}, \pos{i}{\alpha}, \pos{q}{\alpha}, \pos{q_L}{\gamma}, \vect{u})
\stackrel{L}{=}
\phi(\{\alpha, \alpha, \beta, \beta\} \uplus \msupp(\vect{u})).
\]
Thus, for all $\alpha, \beta \in A$ and for all $\vect{u} \in A^{n-4}$, it holds that
\[
\phi(\{\alpha, \alpha, \alpha, \beta\} \uplus \msupp(\vect{u})) =
\phi(\{\alpha, \alpha, \beta, \beta\} \uplus \msupp(\vect{u})) =
\phi(\{\alpha, \beta, \beta, \beta\} \uplus \msupp(\vect{u})).
\]
As in the previous cases, we can show that $\phi(S) = \phi(T)$ for all $S, T \in \mathcal{M}_n(A)$ such that $\set(S) = \set(T)$.
\end{asparaenum}

Cases~\ref{clm:IJdisjoint1:case1}--\ref{clm:IJdisjoint1:case4} exhaust all possibilities, and the proof of the claim is complete.
\end{pfclaim}

\begin{claim}
\label{clm:IJdisjoint2}
If there exist $I, J \in \couples$ such that $I \cap J = \emptyset$ and $q_I, q_J \in J$, then $\phi(S) = \phi(T)$ for all $S, T \in \mathcal{M}_n(A)$ such that $\set(S) = \set(T)$.
\end{claim}

\begin{pfclaim}[Proof of Claim~\ref{clm:IJdisjoint2}]
Let $I = \{i, j\}$, $J = \{p, q\}$, and assume that $I \cap J = \emptyset$ and $q_I, q_J \in J$. Then for any $\alpha, \beta \in A$ and for any $\vect{u} \in A^{n - 4}$ we have
\[
\phi(\{\alpha, \beta, \beta, \beta\} \uplus \msupp(\vect{u}))
\stackrel{I}{=}
g(\pos{i}{\alpha}, \pos{j}{\alpha}, \pos{p}{\beta}, \pos{q}{\beta}, \vect{u})
\stackrel{J}{=}
\phi(\{\alpha, \alpha, \beta, \beta\} \uplus \msupp(\vect{u})).
\]
Proceeding as we did in the proof of Claim~\ref{clm:IJdisjoint1}, we can show that $\phi(S) = \phi(T)$ for all $S, T \in \mathcal{M}_n(A)$ such that $\set(S) = \set(T)$.
\end{pfclaim}

\begin{claim}
\label{clm:qInotinI}
If there exists $I \in \couples$ such that $q_I \notin I$, then $\phi(S) = \phi(T)$ for all $S, T \in \mathcal{M}_n(A)$ such that $\set(S) = \set(T)$.
\end{claim}

\begin{pfclaim}[Proof of Claim~\ref{clm:qInotinI}]
Assume that $I \in \couples$ is such that $q_I \notin I$. Since $n \geq 4$, there exists $p \in \nset{n} \setminus (I \cup \{q_I\})$. Let $J = \{p, q_I\}$. Depending on whether $q_J \notin J$ or $q_J \in J$, Claim~\ref{clm:IJdisjoint1} or Claim~\ref{clm:IJdisjoint2} implies that $\phi(S) = \phi(T)$ for all $S, T \in \mathcal{M}_n(A)$ such that $\set(S) = \set(T)$.
\end{pfclaim}

(Proof of Theorem~\ref{thm:symmetric-reconstructible} continued) If there exists $I \in \couples$ such that $q_I \notin I$, then by Claim~\ref{clm:qInotinI}, we have that $\phi(S) = \phi(T)$ for all $S, T \in \mathcal{M}_n(A)$ such that $\set(S) = \set(T)$. Then $f$ is determined by $\supp$, and it is reconstructible by Proposition~\ref{prop:suppreconstructible}. Otherwise $q_I \in I$ for all $I \in \couples$, and we have for every $\vect{a} \in A^n$ that if $a_{\min I} = a_{\max I}$ for some $I \in \couples$, then $a_{q_I} = a_{\min I} = a_{\max I}$. Then Equation~\eqref{eq:gvalues} yields $g(\vect{a}) = \phi(\msupp(\vect{a}) \setminus \{a_{\max I}\} \uplus \{a_{q_I}\}) = \phi(\msupp(\vect{a}))$ for all $\vect{a} \in A^n$, i.e., $g = f$. We conclude that $f$ is reconstructible.
\end{proof}

\begin{proposition}
\label{prop:symmetric-weakly-reconstructible}
Assume that $n > \max (k, 3)$ and $f, g \colon A^n \to B$ are totally symmetric. If $\deck f = \deck g$, then $f = g$.
\end{proposition}

\begin{proof}
Proposition~\ref{prop:2trans} and the assumption that $\deck f = \deck g$ imply that $f_I \equiv g_J$ for all $I, J \in \couples$.
In particular, setting $N := \{n-1, n\}$, we have that $f_N \equiv g_N$; hence there exists a permutation $\tau \in \symm{n-1}$ such that $f_N(\vect{a}) = g_N(\vect{a} \tau)$ for all $\vect{a} \in A^{n-1}$. By the definition of identification minor, we have
\begin{equation}
f(\vect{a} \delta_N) = f_N(\vect{a}) = g_N(\vect{a} \tau) = g(\vect{a} \tau \delta_N)
\label{eq:f2gtau2}
\end{equation}
for all $\vect{a} \in A^{n-1}$.

We want to show that $f = g$, that is $f(\vect{a}) = g(\vect{a})$ for all $\vect{a} \in A^n$. Let $\vect{a} \in A^n$ be arbitrary. Since $n > k$, there is an element $\alpha \in A$ that has at least two occurrences in $\vect{a}$. By the total symmetry of $f$ and $g$, we may assume that the last two components of $\vect{a}$ are equal to $\alpha$, i.e., $a_{n-1} = a_n = \alpha$. Let $\vect{b}$ be the $(n-1)$-tuple that is obtained by removing the last entry from $\vect{a}$. We clearly have $\vect{a} = \vect{b} \delta_N$.

We need to distinguish between two cases depending on whether $\tau(n-1) = n-1$ or not. Consider first the case that $\tau(n-1) = n-1$. By Equation~\eqref{eq:f2gtau2} and the total symmetry (TS) of $g$ we have
\[
f(\vect{a})
= f(\vect{b} \delta_N)
\stackrel{\eqref{eq:f2gtau2}}{=} g(\vect{b} \tau \delta_N)
\stackrel{\text{TS}}{=} g(\vect{a}).
\]

Consider then the case that $\tau(n-1) = r \neq n-1$. Fix an element $s$ of $\nset{n-1} \setminus \{r, n-1\}$; this set is nonempty since $n > 3$. Let $\beta := a_r$, $\gamma := a_s$. Repeated applications of \eqref{eq:f2gtau2} and the total symmetry of $f$ and $g$ yield
\begin{align*}
f(\vect{a})
&= f(a_1, \dots, \pos{r}{\beta}, \dots, \pos{s}{\gamma}, \dots, a_{n-2}, \pos{n-1 \;\;}{\alpha}, \pos{n}{\alpha}) 
\stackrel{\eqref{eq:f2gtau2}}{=}
g(a_{\tau(1)}, \dots, \pos{\tau^{-1}(n-1) \;\;}{\alpha}, \dots, \pos{\;\; \tau^{-1}(s)}{\gamma}, \dots, a_{\tau(n-2)}, \pos{n-1 \;\;}{\beta}, \pos{n}{\beta}) \\
&\stackrel{\text{TS}}{=}
g(a_{\tau(1)}, \dots, \pos{\tau^{-1}(n-1) \;\;}{\gamma}, \dots, \pos{\;\; \tau^{-1}(s)}{\alpha}, \dots, a_{\tau(n-2)}, \pos{n-1 \;\;}{\beta}, \pos{n}{\beta})
\stackrel{\eqref{eq:f2gtau2}}{=}
f(a_1, \dots, \pos{r}{\beta}, \dots, \pos{s}{\alpha}, \dots, a_{n-2}, \pos{n-1 \;\;}{\gamma}, \pos{n}{\gamma}) \\
&\stackrel{\text{TS}}{=}
f(a_1, \dots, \pos{r}{\alpha}, \dots, \pos{s}{\beta}, \dots, a_{n-2}, \pos{n-1 \;\;}{\gamma}, \pos{n}{\gamma})
\stackrel{\eqref{eq:f2gtau2}}{=}
g(a_{\tau(1)}, \dots, \pos{\tau^{-1}(n-1) \;\;}{\gamma}, \dots, \pos{\;\; \tau^{-1}(s)}{\beta}, \dots, a_{\tau(n-2)}, \pos{n-1 \;\;}{\alpha}, \pos{n}{\alpha}) \\
&\stackrel{\text{TS}}{=}
g(a_1, \dots, \pos{r}{\beta}, \dots, \pos{s}{\gamma}, \dots, a_{n-2}, \pos{n-1 \;\;}{\alpha}, \pos{n}{\alpha})
= g(\vect{a}).
\end{align*}
This shows that $f(\vect{a}) = g(\vect{a})$ for all $\vect{a} \in A^n$, i.e., $f = g$.
\end{proof}

The following example shows that, in Theorem~\ref{thm:symmetric-reconstructible} and Proposition~\ref{prop:symmetric-weakly-reconstructible}, the bound $n \geq 4$ is sharp when $k = 2$.

\begin{example}
\label{ex:symmofo}
Let $A = \{0, 1\}$, let $a, b, c, d \in B$, and let $f, g, h \colon A^3 \to B$ be given by the following table.

\begin{center}
\begin{tabular}{|ccc|ccc|}
\hline
$x$ & $y$ & $z$ & $f(x, y, z)$ & $g(x, y, z)$ & $h(x, y, z)$ \\
\hline
$0$ & $0$ & $0$ & $a$ & $a$ & $a$ \\
$0$ & $0$ & $1$ & $b$ & $c$ & $b$ \\
$0$ & $1$ & $0$ & $b$ & $c$ & $b$ \\
$0$ & $1$ & $1$ & $c$ & $b$ & $b$ \\
$1$ & $0$ & $0$ & $b$ & $c$ & $c$ \\
$1$ & $0$ & $1$ & $c$ & $b$ & $c$ \\
$1$ & $1$ & $0$ & $c$ & $b$ & $c$ \\
$1$ & $1$ & $1$ & $d$ & $d$ & $d$ \\
\hline
\end{tabular}
\end{center}
Functions $f$ and $g$ are totally symmetric, and $h$ is determined by the order of first occurrence, and these functions are pairwise nonequivalent (unless $b = c$). It is not difficult to verify that for every $I \in \couples$, both $f_I$, $g_I$, and $h_I$ are equivalent to the function $(0, 0) \mapsto a$, $(0, 1) \mapsto b$, $(1, 0) \mapsto c$, $(1, 1) \mapsto d$. Hence $f$, $g$, and $h$ are reconstructions of each other.
\end{example}

%%%%%%%%%%%%%%%%%%%%%%%%%%%%%%%%%%%%%%%%%%%%%%%%%%

\section{Functions determined by $(\pr, \supp)$ are reconstructible}
\label{sec:prsupp}

The remainder of this paper deals with the reconstructibility of functions that are weakly determined by the order of first occurrence. We first consider a special subclass, namely the functions determined by $(\pr, \supp)$, and we show that this subclass is reconstructible. In the case of functions defined on a two-element set, this subclass actually coincides with the class of all functions weakly determined by the order of first occurrence.

A function $f \colon A^n \to B$ is \emph{determined by $(\pr, \supp)$} if there exists a map $f^* \colon A \times \mathcal{P}(A) \to B$ and $i \in \nset{n}$ such that $f = f^* \circ(\pr_i^{(n)}, {\supp}|_{A^n})$.

\begin{lemma}
\label{lem:prsuppofo}
\begin{inparaenum}[\rm (i)]
\item\label{lem:prsuppofo:minors}
If $f = f^* \circ (\pr_i^{(n)}, {\supp}|_{A^n})$ for some $i \in \nset{n}$ and $f^* \colon A \times \mathcal{P}(A) \to B$, then $f_I = f^* \circ (\pr_{\delta_I(i)}^{(n-1)}, {\supp}|_{A^{n-1}})$ for all $I \in \couples$.

\item\label{lem:prsuppofo:tool}
Let $f^* \colon A \times \mathcal{P}(A) \to B$ and $f^\sharp \colon A^\sharp \to B$ be functions satisfying $f^\sharp(\vect{a}) = f^*(a_1, \supp(\vect{a}))$ for all $\vect{a} \in A^\sharp$. Then $f^\sharp \circ {\ofo}|_{A^n} = f^* \circ (\pr_1^{(n)}, {\supp}|_{A^n})$.

\item\label{lem:prsuppofo1}
Every function determined by $(\pr, \supp)$ is weakly determined by the order of first occurrence.

\item\label{lem:prsuppofoBoolean}
If $\card{A} = 2$, then a function $f \colon A^n \to B$ is determined by $(\pr, \supp)$ if and only if it is weakly determined by the order of first occurrence.
\end{inparaenum}
\end{lemma}

\begin{proof}
\begin{inparaenum}[\rm (i)]
\item
Let $f = f^* \circ (\pr_i^{(n)}, {\supp}|_{A^n})$, and let $I \in \couples$. Then for any $\vect{b} \in A^{n-1}$, we have
\[
f_I(\vect{b}) = f(\vect{b} \delta_I) = f^*(\pr_i^{(n)}(\vect{b} \delta_I), \supp(\vect{b} \delta_I)) = f^*(\pr_{\delta_I(i)}^{(n-1)}(\vect{b}), \supp(\vect{b})).
\]

\item
Let $\vect{a} \in A^n$. Since $\ofo(\vect{a}) \in A^\sharp$ and the first component of $\ofo(\vect{a})$ is $a_1$, it holds that $f^\sharp(\ofo(\vect{a})) = f^*(a_1, \supp(\ofo(\vect{a})))$. Since $\supp(\ofo(\vect{a})) = \supp(\vect{a})$, we have $f^*(a_1, \supp(\ofo(\vect{a}))) = f^*(a_1, \supp(\vect{a}))$. Consequently, $f^\sharp \circ {\ofo}|_{A^n} = f^* \circ (\pr_1^{(n)}, {\supp}|_{A^n})$.

\item
Assume that $f \colon A^n \to B$ be determined by $(\pr, \supp)$. Then there exists a map $f^* \colon A \times \mathcal{P}(A) \to B$ and $i \in \nset{n}$ such that $f = f^* \circ(\pr_i^{(n)}, {\supp}|_{A^n})$. Let $\sigma \in \symm{n}$ be any permutation that maps $1$ to $i$, and let $f^\sharp \colon A^\sharp \to B$ be given by the rule $f^\sharp(\vect{a}) = f^*(a_1, \supp(\vect{a}))$, for all $\vect{a} \in A^\sharp$. It follows from part~\eqref{lem:prsuppofo:tool} and from the fact that $\supp(\vect{a}) = \supp(\vect{a} \sigma)$ that for all $\vect{a} \in A^n$, it holds that
\begin{multline*}
f(\vect{a})
= f^* \circ (\pr_i^{(n)}, {\supp}|_{A^n})(\vect{a})
= f^* \circ (\pr_1^{(n)}, {\supp}|_{A^n})(\vect{a} \sigma)
= f^\sharp \circ {\ofo}|_{A^n} (\vect{a} \sigma).
\end{multline*}
Thus, $f \equiv f^\sharp \circ {\ofo}|_{A^n}$, i.e., $f$ is weakly determined by the order of first occurrence.

\item
Without loss of generality, we assume that $A = \{0, 1\}$.
In light of part~\eqref{lem:prsuppofo1}, we only need to show that if $f \colon A^n \to B$ is weakly determined by the order of first occurrence, then it is determined by $(\pr, \supp)$. Let $f^\sharp \colon A^\sharp \to B$, $\sigma \in \symm{n}$, and assume that $f(\vect{a}) = f^\sharp(\ofo(\vect{a} \sigma))$ for all $\vect{a} \in A^n$. Let $f^* \colon A \times \mathcal{P}(A) \to B$ be any function satisfying
\begin{align*}
f^*(0, \{0\}) &= f^\sharp(0), &
f^*(0, \{0, 1\}) &= f^\sharp(0, 1), \\
f^*(1, \{1\}) &= f^\sharp(1), &
f^*(1, \{0, 1\}) &= f^\sharp(1, 0).
\end{align*}
Then $f^*$ satisfies the condition $f^*(a_1, \supp(\vect{a})) = f^\sharp(\vect{a})$ for all $\vect{a} \in A^\sharp$. By part~\eqref{lem:prsuppofo:tool}, we have
\begin{multline*}
f(\vect{a})
= f^\sharp(\ofo(\vect{a} \sigma))
= f^* \circ (\pr_1^{(n)}, {\supp}|_{A^n})(\vect{a} \sigma)
= f^* \circ (\pr_{\sigma(1)}^{(n)}, {\supp}|_{A^n}) (\vect{a}),
\end{multline*}
that is, $f$ is determined by $(\pr, \supp)$.
\end{inparaenum}
\end{proof}

\begin{theorem}
\label{thm:prsupp-reconstructible}
Assume that $n \geq k + 2$, $\card{A} = k$, and $f \colon A^n \to B$ is determined by $(\pr, \supp)$. Then $f$ is reconstructible.
\end{theorem}

\begin{proof}
Assume that $f = f^* \circ (\pr_i^{(n)}, {\supp}|_{A^n})$ for some $i \in \nset{n}$ and $f^* \colon A \times \mathcal{P}(A) \to B$. By Lemma~\ref{lem:prsuppofo}~\eqref{lem:prsuppofo:minors}, $f_I \equiv f^* \circ (\pr_1^{(n-1)}, {\supp}|_{A^{n-1}})$ for all $I \in \couples$. Let $g \colon A^n \to B$ be a reconstruction of $f$. Then for every $I \in \couples$, there exists a bijection $\rho_I \in \symm{n-1}$ such that $g_I(\vect{a}) = f^* \circ (\pr_1^{(n-1)}, {\supp}|_{A^{n-1}})(\vect{a} \rho_I)$ for all $\vect{a} \in A^{n-1}$, i.e., $g_I = f^* \circ (\pr_{\rho_I(1)}^{(n-1)}, {\supp}|_{A^{n-1}})$. Let $q_I := \min \delta_I^{-1}(\rho_I(1))$.

Let $\vect{a} \in A^n$. Since $n \geq k + 2$, there exist $I \in \couples$ and $\vect{b} \in A^{n-1}$ such that $\vect{a} = \vect{b} \delta_I$. It holds that $g(\vect{a}) = g(\vect{b} \delta_I) = g_I(\vect{b}) = f^*(b_{\rho_I(1)}, \supp(\vect{b}))$. Since $a_i = b_{\delta_I(i)}$ for every $i \in \nset{n}$ and $\delta_I(1) = \rho_I(1)$, we have $b_{\rho_I(1)} = a_{q_I}$. Therefore,
\begin{equation}
\label{eq:f*aqIsuppa}
g(\vect{a}) = f^*(a_{q_I}, \supp(\vect{a})).
\end{equation}
(In the sequel, we will write ``$\stackrel{J}{=}$'' for a couple $J \in \couples$ to indicate that the equality holds by Equation~\eqref{eq:f*aqIsuppa} for $I = J$.)

\begin{claim}
\label{clm:qIinI}
If $q_I \in I$ for all $I \in \couples$, then $f$ is determined by $\supp$.
\end{claim}
\begin{pfclaim}[Proof of Claim~\ref{clm:qIinI}]
Assume that $q_I \in I$ for all $I \in \couples$. Since $n \geq k + 2$, for every $\alpha, \beta \in A$ and for every $S \subseteq A$ with $\card{S} \geq 2$, there exists a tuple $\vect{u} \in A^n$ such that $\alpha \neq \beta$, $a_1 = a_2 = \alpha$, $a_3 = a_4 = \beta$, $\supp(\vect{u}) = S$. Then
\[
\phi(\alpha, S)
\stackrel{\{1,2\}}{=}
g(\vect{u})
\stackrel{\{3,4\}}{=}
\phi(\beta, S).
\]
Varying $\alpha$ and $\beta$, we have that $f^*(\alpha, S) = f^*(\beta, S)$ for all $\alpha, \beta \in S$, and we can conclude that $f$ is determined by $\supp$.
\end{pfclaim}

\begin{claim}
\label{clm:qIout}
If there exists $I \in \couples$ such that $q_I \notin I$, then $f$ is determined by $\supp$ or for all $J \in \couples$ either $q_J = q_I$ or $J = \{q_I, q_J\}$.
\end{claim}
\begin{pfclaim}[Proof of Claim~\ref{clm:qIout}]
Assume that $I \in \couples$ is a couple that satisfies $q_I \notin I$. Let $J \in \couples$ be a couple. We split the analysis in two cases according to whether $q_I \in J$.

\begin{inparaenum}[\it {Case} 1:]
\item\label{clm:qIout:case1}
$q_I \notin J$; in this case it is not possible that $J = \{q_I, q_J\}$. If $q_J = q_I$, then $J$ has the desired property, so assume that $q_J \neq q_I$. Since $n \geq k + 2$, then for every $S \subseteq A$ with $\card{S} \geq 2$ and for all $\alpha, \beta \in A$ such that $\alpha \neq \beta$, there exists a tuple $\vect{u} \in A^n$ that satisfies $u_{q_I} = \alpha$, $u_{q_J} = \beta$, $u_{\min I} = u_{\max I}$, $u_{\min J} = u_{\max J}$, $\supp(\vect{u}) = S$. Then we have
\[
f^*(\alpha, S)
\stackrel{I}{=}
g(\vect{u})
\stackrel{J}{=}
f^*(\beta, S).
\]
Varying $\alpha$ and $\beta$, we have that $f^*(\alpha, S) = f^*(\beta, S)$ for all $\alpha, \beta \in S$, and we can conclude that $f$ is determined by $\supp$.

\item
$q_I \in J$. If $q_J \in J$, then either $q_J = q_I$ or $J = \{q_I, q_J\}$, and $J$ has the desired property. Assume thus that $q_J \notin J$. Since $n \geq 4$, there exists a couple $K \in \couples$ that is disjoint from $J$; thus $q_I \notin K$. By Case~\ref{clm:qIout:case1}, we have that $f$ is determined by $\supp$ or $q_K = q_I$; thus we assume that $q_K = q_I$. Since $n \geq k + 2$, then for every $S \subseteq A$ with $\card{S} \geq 2$ and for all $\alpha, \beta \in A$ such that $\alpha \neq \beta$, there exists a tuple $\vect{u} \in A^n$ that satisfies $u_{q_J} = \alpha$, $u_{q_K} = \beta$, $u_{\min J} = u_{\max J}$, $u_{\min K} = u_{\max K}$, $\supp(\vect{u}) = S$. Then we have
\[
f^*(\alpha, S)
\stackrel{J}{=}
g(\vect{u})
\stackrel{K}{=}
f^*(\beta, S).
\]
Varying $\alpha$ and $\beta$, we have that $f^*(\alpha, S) = f^*(\beta, S)$ for all $\alpha, \beta \in S$, and we can conclude that $f$ is determined by $\supp$.
\end{inparaenum}
\end{pfclaim}

(Proof of Theorem~\ref{thm:prsupp-reconstructible} continued) By Claims~\ref{clm:qIinI} and~\ref{clm:qIout}, $f$ is determined by $\supp$ or there exists an element $s \in \nset{n}$ such that for all $I \in \couples$, $q_I = s$ or $I = \{q_I, s\}$. In the former case, $f$ is reconstructible by Proposition~\ref{prop:suppreconstructible}. We claim that $g \equiv f$ in the latter case. Let $\vect{a} \in A^n$. Since $n > \card{A}$, there exist $I \in \couples$ and $\vect{b} \in A^{n-1}$ such that $\vect{a} = \vect{b} \delta_I$. Since $a_{q_I} = a_s$, Equation~\eqref{eq:f*aqIsuppa} yields $g(\vect{a}) = f^*(a_s, \supp(\vect{a}))$. We conclude that $g = (\pr_s^{(n)}, {\supp}|_{A^n})$; hence $g \equiv f$, and $f$ is reconstructible.
\end{proof}

\begin{corollary}
\label{cor:Booleanwofo-reconstructible}
Assume that $n \geq 4$ and $\card{A} = 2$. If $f \colon A^n \to B$ is weakly determined by the order of first occurrence, then $f$ is reconstructible.
\end{corollary}

\begin{proof}
If $f$ is as described, then by Lemma~\ref{lem:prsuppofo}~\eqref{lem:prsuppofoBoolean}, $f$ is determined by $(\pr, \supp)$, and by Theorem~\ref{thm:prsupp-reconstructible}, $f$ is reconstructible.
\end{proof}

By Corollary~\ref{cor:Booleanwofo-reconstructible}, the answer to Question~\ref{que:ofo} is positive if $\card{A} = 2$. Example~\ref{ex:symmofo} shows that the bound $n \geq 4$ in Corollary~\ref{cor:Booleanwofo-reconstructible} is sharp. The following corollary shows that the equalities in Proposition~\ref{prop:minord}~\eqref{prop:minord:item3} can be relaxed into equivalences in the case when $\card{A} = 2$ and $n \geq 4$.

\begin{corollary}
\label{cor:gofo}
Assume that $n \geq 4$ and $\card{A} = 2$. Let $f \colon A^n \to B$ and $f^* \colon A^\sharp \to B$. If $f_I \equiv f^* \circ {\ofo}|_{A^{n-1}}$ for all $I \in \couples$, then $f \equiv f^* \circ {\ofo}|_{A^n}$.
\end{corollary}

\begin{proof}
Let $g = f^* \circ {\ofo}|_{A^n}$. By Proposition~\ref{prop:minord}~\eqref{prop:minord:item1}, $g_I = f^* \circ {\ofo}|_{A^n}$. By Corollary~\ref{cor:Booleanwofo-reconstructible}, $g$ is reconstructible. Since $f$ is a reconstruction of $g$, we have $f \equiv f^* \circ {\ofo}|_{A^n}$.
\end{proof}

%%%%%%%%%%%%%%%%%%%%%%%%%%%%%%%%%%%%%%%%%%%%%%%%%%

\section{Equivalence-equalizing couples for functions determined by the order of first occurrence}
\label{sec:eqeq}

We now investigate conditions under which functions determined by the order of first occurrence that are equivalent are actually equal. The main result of this section, Theorem~\ref{thm:eqeq}, finds applications in Section~\ref{sec:results}, where we show that the class of functions weakly determined by the order of first occurrence is weakly reconstructible.

Let $n$ and $k$ be integers greater than or equal to $2$. The couple $(n, k)$ is \emph{equivalence-equalizing for functions determined by the order of first occurrence} (or briefly \emph{equalizing}), if for all $f, g \colon A^n \to B$ with $\card{A} = k$ it holds that if $f$ and $g$ are determined by the order of first occurrence and $f \equiv g$, then $f = g$. Without loss of generality, we will assume throughout this section that $A = \nset{k}$.

\begin{theorem}
\label{thm:eqeq}
Let $n$ and $k$ be integers greater than or equal to $2$. The couple $(n, k)$ is equalizing if and only if
$k \equiv 1, 2 \pmod{4}$ and $n \geq k + 1$; or
$k \equiv 0, 3 \pmod{4}$ and $n \geq k + 2$.
\end{theorem}

The proof of Theorem~\ref{thm:eqeq} is rather long and somewhat technical. We proceed in small steps, establishing several auxiliary lemmas. The general plan is the following. We are going to proceed by induction on $k$. We first establish in Lemma~\ref{lem:equalizing} the basis of the induction and in Lemma~\ref{lem:indstep} the arguments needed in the inductive step. The inductive argument relies on Lemma~\ref{lem:k-equalizing}, a necessary and sufficient condition for a permutation $\sigma$ to be $k$-equalizing (see definition below). The necessity of the condition is established in Lemmas~\ref{lem:thetannotkequalizing}, \ref{lem:equalizing}~\eqref{lem:equalizing:itemnleqk}, and \ref{lem:indstep}. Proving sufficiency occupies most of this section. We formulate a sufficient condition for a permutation $\sigma$ to be $k$-equalizing (see Lemma~\ref{lem:translation}). This condition is expressed in terms of a permutation group that is derived in a certain way from $\sigma$. Then we develop tools with which we can find some elements of the group the presence of which guarantees that the sufficient condition is satisfied, and we verify that all permutations satisfying the condition of Lemma~\ref{lem:k-equalizing} do also satisfy the condition of Lemma~\ref{lem:translation}.

\begin{definition}
In the argument that follows, the following permutations play a very special role. For $n \geq 2$, define $\theta_n \in \symm{n}$ as the following product of disjoint adjacent transpositions:
\[
\theta_n :=
\begin{cases}
(1 \; 2) (3 \; 4) \cdots (n - 1 \;\; n), & \text{if $n$ is even,} \\
(2 \; 3) (4 \; 5) \cdots (n - 1 \;\; n), & \text{if $n$ is odd.}
\end{cases}
\]
If $\ell$ and $k$ have the same parity and $1 \leq \ell \leq k$, then we define $\lambda^\ell_k$ as follows:
\[
\lambda^\ell_k :=
\begin{cases}
(1 \; 2)(3 \; 4) \cdots (\ell - 2 \;\; \ell - 1) (\ell + 1 \;\; \ell + 2) \cdots (k - 1 \;\; k), & \text{if $k$ is odd,} \\
(2 \; 3)(4 \; 5) \cdots (\ell - 2 \;\; \ell - 1) (\ell + 1 \;\; \ell + 2) \cdots (k - 1 \;\; k), & \text{if $k$ is even.}
\end{cases}
\]
\end{definition}

\begin{remark}
\label{rem:thetanlambdaparities}
Note that $\theta_n$ is an even permutation if and only if $n \equiv 0, 1 \pmod{4}$. Note also that for any $\ell$, the permutation $\lambda^\ell_{n-1}$ has parity opposite to that of $\theta_n$.
\end{remark}

\begin{lemma}
\label{lem:lambda}
Let $n$ and $k$ be positive integers such that $n = k + 1 \geq 3$, and let $I \in \couples$. Let $\ell := \min(\max I, \theta_n(\max I))$. Then $\ofo(\vect{k} \delta_I \theta_n) = \vect{k} \lambda^\ell_k$.
\end{lemma}

\begin{proof}
The tuple $\vect{k} \delta_I$ contains exactly one repetition of elements: the entries at the positions indexed by $I$ are equal to $\min I$, while the remaining entries are pairwise distinct and also distinct from $\min I$.
If $n$ is even, then $\vect{k} \delta_I$ is of one of the following two forms for some odd $\ell$:
\newlength{\widmaxI}\settowidth{\widmaxI}{$\max I$}
\begin{gather*}
(1, 2, 3, 4, \dots, \ell - 2, \ell - 1, \pos{\smash[b]{\theta_n(\max I)}}{\makebox[\widmaxI][c]{$\ell$}}, \pos{\max I}{\min I}, \ell + 1, \ell + 2, \dots, k - 1, k), \\
(1, 2, 3, 4, \dots, \ell - 2, \ell - 1, \pos{\max I}{\min I}, \pos{\smash[b]{\theta_n(\max I)}}{\makebox[\widmaxI][c]{$\ell$}}, \ell + 1, \ell + 2, \dots, k - 1, k).
\end{gather*}
In fact, $\ell = \min \{\max I, \theta_n(\max I)\}$. Then $\vect{k} \delta_I \theta_n$ is of one of the following two forms:
\begin{gather*}
(2, 1, 4, 3, \dots, \ell - 1, \ell - 2, \pos{\smash[b]{\theta_n(\max I)}}{\min I}, \pos{\max I}{\makebox[\widmaxI][c]{$\ell$}}, \ell + 2, \ell + 1, \dots, k, k - 1), \\
(2, 1, 4, 3, \dots, \ell - 1, \ell - 2, \pos{\max I}{\makebox[\widmaxI][c]{$\ell$}}, \pos{\smash[b]{\theta_n(\max I)}}{\min I}, \ell + 2, \ell + 1, \dots, k, k - 1).
\end{gather*}
Since $\min I \leq \ell$, we have that $\ofo(\vect{k} \delta_I \theta_n)$ equals
\[
(2, 1, 4, 3, \dots, \ell - 1, \ell - 2, \ell, \ell + 2, \ell + 1, \dots, k, k - 1) = \vect{k} \lambda^\ell_k.
\]
The argument in the case when $n$ is odd is similar.
\end{proof}

We say that a permutation $\sigma \in \symm{n}$ is \emph{$k$-equalizing} if for every $f^+, g^+ \colon A^k_{\neq} \to B$, the condition that $f^+(\ofo(\vect{a})) = g^+(\ofo(\vect{a} \sigma))$ for every $\vect{a} \in A^n$ with $\supp(\vect{a}) = A$ implies $f^+ = g^+$.

\begin{definition}
\label{def:phipsi}
Let $\alpha$ and $\beta$ be distinct elements of $B$, and let $\gamma$ be an element of $B$ distinct from $\alpha$. Define the functions $\phi^+_k, \psi^+_k \colon A^k_{\neq} \to B$, as follows.
\begin{itemize}
\item If $k$ is odd, then let
\begin{align}
\label{eq:4f+}
\phi^+_k(\vect{a}) &=
\begin{cases}
\alpha, & \text{if $\vect{a} = \vect{k} \rho$ for some even $\rho \in \symm{k}$,} \\
\beta,  & \text{if $\vect{a} = \vect{k} \rho$ for some odd $\rho \in \symm{k}$.}
\end{cases} \\
\label{eq:4g+}
\psi^+_k(\vect{a}) &=
\begin{cases}
\alpha, & \text{if $\vect{a} = \vect{k} \rho$ for some odd $\rho \in \symm{k}$,} \\
\beta,  & \text{if $\vect{a} = \vect{k} \rho$ for some even $\rho \in \symm{k}$.}
\end{cases}
\end{align}

\item If $k$ is even, then let
\begin{align}
\label{eq:5f+}
\phi^+_k(\vect{a}) &=
\begin{cases}
\alpha, & \text{if $\vect{a} = \vect{k} \rho$ for some even $\rho \in \symm{k}$ with $\rho(1) = 1$,} \\
\beta,  & \text{if $\vect{a} = \vect{k} \rho$ for some odd $\rho \in \symm{k}$ with $\rho(1) = 1$,} \\
\gamma, & \text{otherwise.}
\end{cases} \\
\label{eq:5g+}
\psi^+_k(\vect{a}) &=
\begin{cases}
\alpha, & \text{if $\vect{a} = \vect{k} \rho$ for some odd $\rho \in \symm{k}$ with $\rho(1) = 1$,} \\
\beta,  & \text{if $\vect{a} = \vect{k} \rho$ for some even $\rho \in \symm{k}$ with $\rho(1) = 1$,} \\
\gamma, & \text{otherwise.}
\end{cases}
\end{align}
\end{itemize}
\end{definition}

\begin{lemma}
\label{lem:thetannotkequalizing}
Let $n$ and $k$ be positive integers such that $n = k + 1 \equiv 0, 1 \pmod{4}$. Then $\phi^+_k(\ofo(\vect{a})) = \psi^+_k(\ofo(\vect{a} \theta_n))$ for all $\vect{a} \in A^n$ with $\supp(\vect{a}) = A$. Consequently, $\theta_n$ is not $k$-equalizing.
\end{lemma}

\begin{proof}
Let $\vect{a} \in A^n$ such that $\supp(\vect{a}) = A$. Since $n = k + 1$, we have that $\vect{a} = \vect{k} \pi \delta_I$ for some $I \in \couples$ and for some permutation $\pi \in \symm{k}$. Let $\ell = \min(\max I, \theta_n(\max I))$. Then $\ofo(\vect{a}) = \vect{k} \pi$ and $\ofo(\vect{a} \theta_n) = \ofo(\vect{k} \pi \delta_I \theta_n) =  \vect{k} \pi \lambda^\ell_k$ by Lemma~\ref{lem:lambda} and Remark~\ref{rem:ofolambda}. Since $n \equiv 0, 1 \pmod{4}$, $\lambda^\ell_k$ is an odd permutation by Remark~\ref{rem:thetanlambdaparities}. Hence, the permutations $\pi$ and $\pi \lambda^\ell_k$ have opposite parities. Moreover, if $n \equiv 1 \pmod{4}$, i.e., $k \equiv 0 \pmod{4}$, then $\lambda^\ell_k$ fixes $1$, and it holds that $\pi(1) = \pi \lambda^\ell_k (1)$. Therefore, it holds that $\phi^+_k(\ofo(\vect{a})) = \phi^+_k(\vect{k} \pi) = \psi^+_k(\vect{k} \pi \lambda^\ell_k) = \psi^+_k(\ofo(\vect{a} \theta_n))$. Since $\phi^+_k \neq \psi^+_k$, we conclude that $\theta_n$ is not $k$-equalizing.
\end{proof}

\begin{lemma}
\label{lem:equalizing}
Let $n$ and $k$ be integers greater than or equal to $2$.
\begin{enumerate}[\rm (i)]
\item\label{lem:equalizing:item2k-1}
If $n \geq 2k - 1$, then $(n, k)$ is equalizing.
\item\label{lem:equalizing:itemnleqk}
If $n \leq k$, then $(n, k)$ is not equalizing.
\end{enumerate}
\end{lemma}

\begin{proof}
\eqref{lem:equalizing:item2k-1} Assume that $n \geq 2k - 1$. Let $f, g \colon A^n \to B$ be functions determined by the order of first occurrence. Then there exist maps $f^*, g^* \colon A^\sharp \to B$ such that $f = f^* \circ {\ofo}|_{A^n}$, $g = g^* \circ {\ofo}|_{A^n}$. Assume that $f \equiv g$; thus there exists a permutation $\sigma \in \symm{n}$ such that $f(\vect{a}) = g(\vect{a} \sigma)$ for all $\vect{a} \in A^n$.

Let $\vect{a} \in A^n$, and let $(\alpha_1, \dots, \alpha_r) := \ofo(\vect{a})$. Define the tuple $\vect{b} \in A^n$ as follows.
\begin{compactenum}
\item Set $N_1 := \nset{n}$.
\item For $\ell = 1, \dots, r - 1$, set $M_\ell := \{\min(N_\ell), \sigma(\min(\sigma^{-1}(N_\ell)))\}$ and $N_{\ell+1} := N_\ell \setminus M_\ell$.
\item Set $M_r := N_r$.
\item For every $i \in \nset{n}$, set $b_i := \alpha_\ell$ if $i \in M_\ell$.
\end{compactenum}
Note that $1 \leq \card{M_\ell} \leq 2$ for every $\ell \in \nset{r-1}$; hence $\card{N_\ell} \geq n + 2 - 2 \ell$ for every $\ell \in \nset{r}$. Since $r \leq k$ and we assume that $n \geq 2k - 1$, it holds that $\card{N_r} \geq n + 2 - 2k \geq 1$. Thus $N_1 \supset N_2 \supset \dots \supset N_r \neq \emptyset$, and $(M_1, \dots, M_r)$ is a partition of $\nset{n}$. The tuple $\vect{b} = (b_1, \dots, b_n)$ is thus well defined.
Moreover, it is clear from the construction that $\ofo(\vect{b}) = (\alpha_1, \dots, \alpha_r) = \ofo(\vect{b} \sigma)$.
It follows that $f(\vect{a}) = f(\vect{b}) = g(\vect{b} \sigma) = g(\vect{a})$ for all $\vect{a} \in A^n$, i.e., $f = g$.

\eqref{lem:equalizing:itemnleqk} Let $\sigma \in \symm{n}$ be a non-identity permutation. Let $\alpha$ and $\beta$ be distinct elements of $B$. Define the functions $f, g \colon A^n \to B$ by the rules
\begin{align*}
f(\vect{a}) &=
\begin{cases}
\alpha, & \text{if $\vect{a} = (1, 2, \dots, n)$,} \\
\beta, & \text{otherwise,}
\end{cases} \\
g(\vect{a}) &=
\begin{cases}
\alpha, & \text{if $\vect{a} = (\sigma(1), \sigma(2), \dots, \sigma(n))$,} \\
\beta, & \text{otherwise.}
\end{cases}
\end{align*}
It is clear that both $f$ and $g$ are determined by the order of first occurrence, $f \neq g$, and $f(\vect{a}) = g(\vect{a} \sigma)$ for all $\vect{a} \in A^n$, that is, $f \equiv g$.
\end{proof}

\begin{lemma}
\label{lem:indstep}
For $k \geq 3$, the couple $(n, k)$ is equalizing if and only if $(n, k - 1)$ is equalizing and every $\sigma \in \symm{n}$ is $k$-equalizing.
\end{lemma}

\begin{proof}
Assume first that $(n, k)$ is equalizing. Let $f^+, g^+ \colon A^k_{\neq} \to B$, and let $\sigma \in \symm{n}$. Assume that $f^+(\ofo(\vect{a})) = g^+(\ofo(\vect{a} \sigma))$ for every $\vect{a} \in A^n$ with $\supp(\vect{a}) = A$. Let $\beta$ be an arbitrary element of $B$, and extend $f^+$ and $g^+$ into functions $f^*, g^* \colon A^\sharp \to B$ as follows:
\[
f^*(\vect{c}) =
\begin{cases}
f^+(\vect{c}), & \text{if $\vect{c} \in A^k_{\neq}$,} \\
\beta, & \text{otherwise,}
\end{cases}
\qquad
g^*(\vect{c}) =
\begin{cases}
g^+(\vect{c}), & \text{if $\vect{c} \in A^k_{\neq}$,} \\
\beta, & \text{otherwise.}
\end{cases}
\]
It clearly holds that $f^*(\ofo(\vect{a})) = g^*(\ofo(\vect{a} \sigma))$ for all $\vect{a} \in A^n$. By the assumption that $(n, k)$ is equalizing, it holds that $f^* \circ {\ofo}|_{A^n} = g^* \circ {\ofo}|_{A^n}$; hence $f^* = g^*$ by Remark~\ref{rem:restrord}. Since $f^+$ and $g^+$ are the restrictions of $f^*$ and $g^*$ to the subset $A^k_{\neq}$ of $A^\sharp$, we have that $f^+ = g^+$. Thus, every $\sigma \in \symm{n}$ is $k$-equalizing.

In order to show that $(n, k - 1)$ is equalizing, denote $E := A \setminus \{k\} = \nset{k-1}$, and let $f, g \colon E^n \to B$ be functions satisfying $f = f^* \circ {\ofo}|_{E^n}$, $g = g^* \circ {\ofo}|_{E^n}$ for some $f^*, g^* \colon E^\sharp \to B$ and $f \equiv g$. Thus there exists a permutation $\sigma$ of $\nset{n}$ such that $f(\vect{a}) = g(\vect{a} \sigma)$ for all $\vect{a} \in E^n$.

Let $\beta$ be an arbitrary element of $B$. Extend $f$ and $g$ into functions $f^\dagger, g^\dagger \colon A^n \to B$ as follows:
\[
f^\dagger(\vect{a}) =
\begin{cases}
f(\vect{a}), & \text{if $\vect{a} \in E^n$,} \\
\beta, & \text{otherwise,}
\end{cases}
\qquad
g^\dagger(\vect{a}) =
\begin{cases}
g(\vect{a}), & \text{if $\vect{a} \in E^n$,} \\
\beta, & \text{otherwise.}
\end{cases}
\]
It is easy to verify that $f^\dagger(\vect{a}) = g^\dagger(\vect{a} \sigma)$ for all $\vect{a} \in A^n$; hence $f^\dagger \equiv g^\dagger$. Moreover, $f^\dagger = f^\diamond \circ {\ofo}|_{A^n}$, $g^\dagger = g^\diamond \circ {\ofo}|_{A^n}$, where $f^\diamond, g^\diamond \colon A^\sharp \to B$ are defined as
\[
f^\diamond(\vect{c}) =
\begin{cases}
f^*(\vect{c}), & \text{if $\vect{c} \in E^\sharp$,} \\
\beta, & \text{otherwise,}
\end{cases}
\qquad
g^\diamond(\vect{c}) =
\begin{cases}
g^*(\vect{c}), & \text{if $\vect{c} \in E^\sharp$,} \\
\beta, & \text{otherwise.}
\end{cases}
\]
By the assumption that $(n, k)$ is equalizing, we have that $f^\dagger = g^\dagger$. Since $f$ and $g$ are restrictions of $f^\dagger$ and $g^\dagger$ to the subset $E^n$ of $A^n$, it obviously holds that $f = g$. We conclude that $(n, k - 1)$ is equalizing.

For the converse implication, assume that $(n, k - 1)$ is equalizing and every $\sigma \in \symm{n}$ is $k$-equalizing. Let $f, g \colon A^n \to B$, and assume that $f = f^* \circ {\ofo}|_{A^n}$, $g = g^* \circ {\ofo}|_{A^n}$, and $f \equiv g$. Then there exists a permutation $\sigma \in \symm{n}$ such that $f(\vect{a}) = g(\vect{a} \sigma)$ for all $\vect{a} \in A^n$. From our assumptions it follows immediately that $f^*|_{A^k_{\neq}} = g^*|_{A^k_{\neq}}$, that is, $f(\vect{a}) = g(\vect{a})$ for all $\vect{a} \in A^n$ with $\supp(\vect{a}) = A$. Furthermore, for any subset $E$ of $A$ with $\card{E} = k - 1$, it clearly holds that $f|_{E^n} = f^* \circ {\ofo}|_{E^n}$, $g|_{E^n} = g^* \circ {\ofo}|_{E^n}$, and $f|_{E^n} \equiv g|_{E^n}$. By the assumption that $(n, k - 1)$ is equalizing, it holds that $f|_{E^n} = g|_{E^n}$. Thus $f$ and $g$ coincide at all points in $A^n$, i.e., $f = g$. We conclude that $(n, k)$ is equalizing.
\end{proof}

Let $\sigma \in \symm{n}$. Define the relation $\triangleright^\sigma_k$ on $\symm{k}$ by the following rule: $\pi \triangleright^\sigma_k \tau$ if and only if there exists a tuple $\vect{a} \in A^n$ such that $\ofo(\vect{a}) = \vect{k} \pi$ and $\ofo(\vect{a} \sigma) = \vect{k} \tau$.
Denote the converse relation of $\triangleright^\sigma_k$ by $\triangleleft^\sigma_k$, i.e., $\pi \triangleleft^\sigma_k \tau$ if and only if $\tau \triangleright^\sigma_k \pi$. Define the relation $\sim^\sigma_k$ on $\symm{k}$ as $\pi \sim^\sigma_k \tau$ if and only if $\pi (\triangleright^\sigma_k \circ (\triangleleft^\sigma_k \circ \triangleright^\sigma_k)^m) \tau$ for some $m \geq 0$, i.e., there exist a positive integer $m$ and $\pi_1, \pi_2, \dots, \pi_{2m} \in \symm{k}$ such that $\pi = \pi_1$, $\tau = \pi_{2m}$, and $\pi_{2i-1} \triangleright^\sigma_k \pi_{2i}$ whenever $1 \leq i \leq m$, and $\pi_{2i} \triangleleft^\sigma_k \pi_{2i+1}$ whenever $1 \leq i \leq m-1$.

\begin{lemma}
\label{lem:pitaulambda}
For every $\pi, \tau, \lambda \in \symm{k}$, the condition $\pi \triangleright^\sigma_k \tau$ implies $\lambda \pi \triangleright^\sigma_k \lambda \tau$. Consequently, for every $\pi, \tau, \lambda \in \symm{k}$, the condition $\pi \sim^\sigma_k \tau$ implies $\lambda \pi \sim^\sigma_k \lambda \tau$.
\end{lemma}

\begin{proof}
If $\pi \triangleright^\sigma_k \tau$, then there exists a tuple $\vect{a} \in A^n$ such that $\ofo(\vect{a}) = \vect{k} \pi$ and $\ofo(\vect{a} \sigma) = \vect{k} \tau$. As $\vect{a} = \vect{k} \gamma$ for some $\gamma \colon \nset{n} \to \nset{k}$, we have $\ofo(\vect{k} \gamma) = \vect{k} \pi$ and $\ofo(\vect{k} \gamma \sigma) = \vect{k} \tau$. By Remark~\ref{rem:ofolambda}, we have $\ofo(\vect{k} \lambda \gamma) = \vect{k} \lambda \pi$ and $\ofo(\vect{k} \lambda \gamma \sigma) = \vect{k} \lambda \tau$. Hence $\lambda \pi \triangleright^\sigma_k \lambda \tau$. The claim about $\sim^\sigma_k$ follows immediately.
\end{proof}

Let
\[
U^\sigma_k := \{\pi \in \symm{k} : \id \triangleright^\sigma_k \pi\},
\qquad
\Delta^\sigma_k := \{\pi^{-1} \tau \mid \pi, \tau \in U^\sigma_k\},
\]
and let $G^\sigma_k$ be the subgroup of $\symm{k}$ generated by $\Delta^\sigma_k$. Note that since $(\pi^{-1} \tau)^{-1} = \tau^{-1} \pi$, the set $\Delta^\sigma_k$ contains the inverse permutation of each element of $\Delta^\sigma_k$.

\begin{lemma}
\label{lem:pirho}
If $\pi$ and $\rho$ are permutations such that $\pi \sim^\sigma_k \rho$ and $\tau \in \Delta^\sigma_k$, then $\pi \sim^\sigma_k \rho \tau$.
Consequently, if $\pi$ and $\rho$ are permutations such that $\pi \sim^\sigma_k \rho$ and $\tau \in G^\sigma_k$, then $\pi \sim^\sigma_k \rho \tau$.
\end{lemma}

\begin{proof}
Since $\pi \sim^\sigma_k \rho$, there exist a positive integer $m$ and permutations $\pi_1, \dots, \pi_{2m}$ such that $\pi = \pi_1$, $\rho = \pi_{2m}$, and $\pi_{2i-1} \triangleright^\sigma_k \pi_{2i}$ whenever $1 \leq i \leq m$, and $\pi_{2i} \triangleleft^\sigma_k \pi_{2i+1}$ whenever $1 \leq i \leq m-1$.
Since $\tau \in \Delta^\sigma_k$, there exist permutations $\tau_1, \tau_2 \in U^\sigma_k$ such that $\tau = \tau_1^{-1} \tau_2$. By the definition of $U^\sigma_k$, we have $\id \triangleright^\sigma_k \tau_1$ and $\id \triangleright^\sigma_k \tau_2$. It follows from Lemma~\ref{lem:pitaulambda} that $\rho \tau_1^{-1} \triangleright^\sigma_k \rho \tau_1^{-1} \tau_1 = \rho$ and $\rho \tau_1^{-1} \triangleright^\sigma_k \rho \tau_1^{-1} \tau_2 = \rho \tau$. Thus $\pi \sim^\sigma_k \rho \tau$.

If $\tau \in G^\sigma_k$, then there exist a nonnegative integer $m$ and $\tau_1, \tau_2, \dots, \tau_m \in \Delta^\sigma_k$ such that $\tau = \tau_1 \tau_2 \cdots \tau_m$. An easy induction on $m$ then shows that $\pi \sim^\sigma_k \rho$ implies $\pi \sim^\sigma_k \rho \tau$.
\end{proof}

\begin{lemma}
\label{lem:simreflUG}
Let $\sigma \in \symm{n}$. The relation $\sim^\sigma_k$ is reflexive if and only if $U^\sigma_k \cap G^\sigma_k \neq \emptyset$.
\end{lemma}

\begin{proof}
By Lemma~\ref{lem:pitaulambda}, the reflexivity of $\sim^\sigma_k$ is equivalent to the condition that $\pi \sim^\sigma_k \pi$ for some $\pi \in \symm{k}$; in particular, this is equivalent to the condition that $\id \sim^\sigma_k \id$.

Assume first that $\id \sim^\sigma_k \id$. Then there exist a positive integer $m$ and permutations $\pi_1, \pi_2, \dots, \pi_{2m} \in \symm{k}$ such that $\id = \pi_1$, $\id = \pi_{2m}$, and $\pi_{2i-1} \triangleright^\sigma_k \pi_{2i}$ whenever $1 \leq i \leq m$, and $\pi_{2i} \triangleleft^\sigma_k \pi_{2i+1}$ whenever $1 \leq i \leq m - 1$. By Lemma~\ref{lem:pitaulambda} and by the definition of $U^\sigma_k$, it holds that $\pi_{2i-1}^{-1} \pi_{2i} \in U^\sigma_k$ whenever $1 \leq i \leq m$ and $\pi_{2i+1}^{-1} \pi_{2i} \in U^\sigma_k$ whenever $1 \leq i \leq m - 1$. By the definition of $\Delta^\sigma_k$, we have $(\pi_{2i-1}^{-1} \pi_{2i-2})^{-1} (\pi_{2i-1}^{-1} \pi_{2i}) = \pi_{2i-2}^{-1} \pi_{2i} \in \Delta^\sigma_k$ whenever $2 \leq i \leq m$. Thus, $U^\sigma_k$ contains $\pi_1^{-1} \pi_2 = \id^{-1} \pi_2 = \pi_2$, and $G^\sigma_k$ contains the permutation
\[
(\pi_2^{-1} \pi_4) (\pi_4^{-1} \pi_6) \cdots (\pi_{2m-2}^{-1} \pi_{2m}) = \pi_2^{-1} \pi_{2m} = \pi_2^{-1} \id = \pi_2^{-1}.
\]
Being a group, $G^\sigma_k$ contains the inverses of its members. Therefore $\pi_2 \in G^\sigma_k$, and we conclude that $U^\sigma_k \cap G^\sigma_k \neq \emptyset$.

Assume then that $U^\sigma_k \cap G^\sigma_k \neq \emptyset$, and let $\rho \in U^\sigma_k \cap G^\sigma_k$. Then $\id \sim^\sigma_k \rho$ by the definition of $U^\sigma_k$. Moreover, $\rho^{-1} \in G^\sigma_k$ because $G^\sigma_k$ is a group and hence contains the inverses of its elements. It then follows from Lemma~\ref{lem:pirho} that $\id \sim^\sigma_k \rho \rho^{-1} = \id$.
\end{proof}

\begin{lemma}
\label{lem:UG}
$U^\sigma_k \cap G^\sigma_k \neq \emptyset$ if and only if $U^\sigma_k \subseteq G^\sigma_k$.
\end{lemma}

\begin{proof}
The set $U^\sigma_k$ is nonempty by definition, so if $U^\sigma_k \subseteq G^\sigma_k$, then $U^\sigma_k \cap G^\sigma_k = U^\sigma_k \neq \emptyset$.

Assume then that $U^\sigma_k \cap G^\sigma_k \neq \emptyset$, and let $\pi \in U^\sigma_k \cap G^\sigma_k$. Let $\rho \in U^\sigma_k$. Then $\pi^{-1} \rho \in \Delta^\sigma_k$ and hence $\rho = \pi (\pi^{-1} \rho) \in G^\sigma_k$.
\end{proof}

\begin{lemma}
Assume that $\sim^\sigma_k$ is reflexive. Then for every $\tau \in \symm{k}$, $\id \sim^\rho_k \tau$ if and only if $\tau \in G^\sigma_k$.
\end{lemma}

\begin{proof}
Assume that $\tau \in G^\sigma_k$. By the reflexivity of $\sim^\sigma_k$, we have $\id \sim^\sigma_k \id$, and Lemma~\ref{lem:pirho} implies that $\id \sim^\sigma_k \tau$.

Assume then that $\id \sim^\sigma_k \tau$. Then there exist a positive integer $m$ and permutations $\pi_1, \dots, \pi_{2m}$ such that $\id = \pi_1 \triangleright^\sigma_k \pi_2 \triangleleft^\sigma_k \pi_3 \triangleright^\sigma_k \cdots \triangleleft^\sigma_k \pi_{2m-1} \triangleright^\sigma_k \pi_{2m} = \tau$. Then $\pi_{2i-2}^{-1} \pi_{2i} \in \Delta^\sigma_k$ for all $i \in \{2, \dots, m\}$ and $\pi_2 \in U^\sigma_k$. By Lemmas~\ref{lem:simreflUG} and~\ref{lem:UG} and by the assumption that $\sim^\sigma_k$ is reflexive, we have $U^\sigma_k \subseteq G^\sigma_k$. Thus,
\[
\tau = \pi_{2m} =
\pi_2 (\pi_2^{-1} \pi_4) \cdots (\pi_{2i-2}^{-1} \pi_{2i}) \cdots (\pi_{2m-4}^{-1} \pi_{2m-2}) (\pi_{2m-2}^{-1} \pi_{2m})
\in G^\sigma_k.
\qedhere
\]
\end{proof}

\begin{lemma}
\label{lem:translation}
Let $\sigma \in \symm{n}$. If $U^\sigma_k \cap G^\sigma_k \neq \emptyset$ (or, equivalently, if $\sim^\sigma_k$ is reflexive), then $\sigma$ is $k$-equalizing.
\end{lemma}

\begin{proof}
Let $f^+, g^+ \colon A^k_{\neq} \to B$, and assume that $f^+(\ofo(\vect{a})) = g^+(\ofo(\vect{a} \sigma))$ for all $\vect{a} \in A^n$ with $\supp(\vect{a}) = A$. By the definition of $\triangleright^\sigma_k$, we have that $\pi \triangleright^\sigma_k \tau$ implies $f^+(\vect{k} \pi) = g^+(\vect{k} \tau)$, and an easy induction shows that $\pi \sim^\sigma_k \tau$ implies $f^+(\vect{k} \pi) = g^+(\vect{k} \tau)$. The case $m = 0$ holds by the definition of $\triangleright^\sigma_k$. If $\pi_1 (\triangleright^\sigma_k \circ (\triangleleft^\sigma_k \circ \triangleright^\sigma_k)^m) \tau_1 \triangleleft^\sigma_k \pi_2 \triangleright^\sigma_k \tau_2$, then $f^+(\vect{k} \pi_1) = g^+(\vect{k} \tau_1)$ by the induction hypothesis, and, by the definition of $\triangleright^\sigma_k$, there exist tuples $\vect{a}, \vect{b} \in A^n$ with $\supp(\vect{a}) = A = \supp(\vect{b})$ such that $\ofo(\vect{a}) = \vect{k} \pi_2$, $\ofo(\vect{a} \sigma) = \vect{k} \tau_1$, $\ofo(\vect{b}) = \vect{k} \pi_2$, and $\ofo(\vect{b} \sigma) = \vect{k} \tau_2$; hence $f^+(\vect{k} \pi_2) = g^+(\vect{k} \tau_1)$ and $f^+(\vect{k} \pi_2) = g^+(\vect{k} \tau_2)$; thus $f^+(\vect{k} \pi_1) = g^+(\vect{k} \tau_2)$.

Thus, if $\sim^\sigma_k$ is reflexive (or, equivalently, by Lemma~\ref{lem:simreflUG}, if $U^\sigma_k \cap G^\sigma_k \neq \emptyset$), then $f^+ = g^+$.
\end{proof}

Thus, in order to show that $\sigma \in \symm{n}$ is $k$-equalizing, we want to exhibit a permutation that belongs to $U^\sigma_k \cap G^\sigma_k$. We will shortly define a permutation $\pi^\sigma_k$, and our aim is to show that $\pi^\sigma_k \in U^\sigma_k$ (Lemma~\ref{lem:vsigmaU}) and $\pi^\sigma_k \in G^\sigma_k$ (Lemma~\ref{lem:vsigmaG}).

For integers $x$ and $\alpha$, denote
\[
x^{(\alpha)} :=
\begin{cases}
x, & \text{if $x \leq \alpha$,} \\
x - 1, & \text{it $x > \alpha$.}
\end{cases}
\]
We also denote $(a_1, \dots, a_n)^{(\alpha)} := (a_1^{(\alpha)}, \dots, a_n^{(\alpha)})$.

If $n = k + r$ for some $r \geq 1$ and $I_1, \dots, I_r$ are elements of $\couples$ such that $\max I_i \neq \max I_j$ whenever $i \neq j$, then we denote by $\vect{k}_{I_1, \dots, I_r}$ the unique tuple $\vect{a} \in A^n$ such that $\ofo(\vect{a}) = \vect{k}$ and $a_i = a_j$ for all $i, j \in I_\ell$ ($1 \leq \ell \leq r$). Every tuple $\vect{a} \in A^n$ such that $\ofo(\vect{a}) = \vect{k}$ is of the form $\vect{k}_{I_1, \dots, I_r}$ for some $I_1, \dots, I_r$. Note that the condition on distinct maxima is necessary: not every set of $r$ pairwise distinct couples specifies $r$ repetitions (consider, for example, the couples $\{1, 2\}$, $\{1, 3\}$, and $\{2, 3\}$).

For a permutation $\sigma \in \symm{n}$, let $S_\sigma := \{i \in \nset{n} : \sigma(i) \leq k + 1\}$, $b_\sigma := \min S_\sigma$ and $c_\sigma := \sigma(b_\sigma)$. Let $I \in \couples[k+1]$. If $c_\sigma \in I$, then let $d := \min I$; otherwise let $d := c_\sigma^{(\max I)}$. Let $I_1 = I$, and for $i \in \{2, \dots, r\}$, let $I_i = \{d, k + i\}$. Define $\sigma_I$ to be the unique permutation $\rho \in \symm{k}$ such that $\ofo(\vect{k}_{I_1, \dots, I_r} \sigma) = \vect{k} \rho$.

We define $\pi^\sigma_k$ as $\sigma_I$, where
\[
I =
\begin{cases}
\{k + 1,\, c_\sigma\}, & \text{if $c_\sigma \neq k + 1$,} \\
\{k + 1,\, \sigma(\min S_\sigma \setminus \{b_\sigma\})\}, & \text{if $c_\sigma = k + 1$.}
\end{cases}
\]

\begin{remark}
\label{rem:sigmaI}
Observe that for the choice of the sets $I_1, \dots, I_r$ as in the definition of $\sigma_I$ above, we have
\[
\vect{k}_{I_1, \dots, I_r}(i) =
\begin{cases}
i^{(\max I)}, & \text{if $1 \leq i \leq k + 1$ and $i \neq \max I$,} \\
\min I, & \text{if $i = \max I$,} \\
d, & \text{if $i > k + 1$.}
\end{cases}
\]
Hence
\[
\vect{k}_{I_1, \dots, I_r} \sigma (i) =
\begin{cases}
(\sigma(i))^{(\max I)}, & \text{if $1 \leq \sigma(i) \leq k + 1$ and $\sigma(i) \neq \max I$,} \\
\min I, & \text{if $\sigma(i) = \max I$,} \\
d, & \text{if $\sigma(i) > k + 1$.}
\end{cases}
\]
Let us determine $\ofo(\vect{k}_{I_1, \dots, I_r} \sigma)$ by finding the first occurrence of each element of $A$ in $\vect{k}_{I_1, \dots, I_r} \sigma$ and deleting the repetitions. The element $d$ occurs at position $b_\sigma$ and at all positions $i$ such that $\sigma(i) > k + 1$, and, in the case that $b_\sigma \in I$ (i.e., $d = \min I$), also at position $\sigma^{-1}(\max I)$. Since $\sigma(i) > k + 1$ whenever $1 \leq i < b_\sigma$, we actually have that $\vect{k}_{I_1, \dots, I_r} \sigma (i) = d$ whenever $i \leq b_\sigma$, and we can delete the occurrences of $d$ at every position $i$ such that $\sigma(i) > k + 1$. The element $\min I$ occurs at positions $\sigma^{-1}(\min I)$ and $\sigma^{-1}(\max I)$, so the occurrence at the rightmost of these two positions, i.e., at position $\max \sigma^{-1}(I)$, can be deleted. The remaining entries of $\vect{k}_{I_1, \dots, I_r} \sigma$ are pairwise distinct and also distinct from $d$ and $\min I$, so each one of them must remain.

Thus we obtained the following simple procedure for determining $\sigma_I$ directly from the representation $\vect{s} = (\sigma(1), \dots, \sigma(n))$ of $\sigma$ in one-line notation. Remove all entries greater than $k + 1$ from $\vect{s}$. Find the positions where $\min I$ and $\max I$ occur, substitute $\min I$ in the leftmost of these two positions, and remove the entry in the rightmost of these two positions. Finally, decrease by one each remaining entry that is greater than $\max I$. The resulting $k$-tuple is $\ofo(\vect{k}_{I_1, \dots, I_r} \sigma)$, or, equivalently, the representation of $\sigma_I$ in one-line notation.

The representation of $\pi^\sigma_k$ in one-line notation is obtained from that of $\sigma$ simply by removing all entries greater than $k$.
\end{remark}

To facilitate easy working the procedure described in Remark~\ref{rem:sigmaI}, let us denote by $\check{\sigma}_{\ell}$ the permutation on $\nset{\ell}$ whose representation in one-line notation is obtained from that of $\sigma$ by removing all entries greater than $\ell$. Note that if $n = \ell$, then $\check{\sigma}_{\ell} = \sigma$.

\begin{lemma}
\label{lem:vsigmaU}
For any $\sigma \in \symm{n}$ and $I \in \couples[k+1]$, it holds that $\sigma_I \in U^\sigma_k$.
In particular, $\pi^\sigma_k \in U^\sigma_k$.
\end{lemma}

\begin{proof}
The claim follows immediately from the definitions. For each $I \in \couples[k+1]$, there exists a tuple $\vect{a} \in A^n$ with $\ofo(\vect{a}) = \vect{k}$ such that $\ofo(\vect{a} \sigma) = \vect{k} \sigma_I$. Therefore, $\id \triangleright^\sigma_k \sigma_I$, that is, $\sigma_I \in U^\sigma_k$. The claim about $\pi^\sigma_k$ is a particular case, $\pi^\sigma_k$ being of the form $\sigma_I$ for a suitable choice of $I$.
\end{proof}

\begin{example}
Let us illustrate the definitions and constructions described above. Let $n := 6$, $k := 4$, $\sigma := (3, 1, 5, 2, 6, 4) \in \symm{n}$. Then $S_\sigma = \{1, 2, 3, 4, 6\}$, $b_\sigma = 1$, $c_\sigma = \sigma(1) = 3$, and $\check{\sigma}_{k+1} = (3, 1, 5, 2, 4)$. Below are listed the permutations $\sigma_I \in \symm{k}$ for each $I \in \couples[k+1]$.

\medskip
\begin{center}
\begin{tabular}{|c|ccccc|}
\hline
$I$        & $\{1, 2\}$     & $\{1, 3\}$     & $\{1, 4\}$     & $\{1, 5\}$     & $\{2, 3\}$     \\
$\sigma_I$ & $(2, 1, 4, 3)$ & $(1, 4, 2, 3)$ & $(3, 1, 4, 2)$ & $(3, 1, 2, 4)$ & $(2, 1, 4, 3)$ \\
\hline
$I$        & $\{2, 4\}$     & $\{2, 5\}$     & $\{3, 4\}$     & $\{3, 5\}$     & $\{4, 5\}$     \\
$\sigma_I$ & $(3, 1, 4, 2)$ & $(3, 1, 2, 4)$ & $(3, 1, 4, 2)$ & $(3, 1, 2, 4)$ & $(3, 1, 4, 2)$ \\
\hline
\end{tabular}
\end{center}
\medskip

The permutation $\pi^\sigma_k$ equals $\sigma_I$ for $I = \{c_\sigma, k+1\} = \{3, 5\}$, or, equivalently, $\check{\sigma}_{k}$, that is, $\pi^\sigma_k = (3, 1, 2, 4)$.

Thus, $U^\sigma_k$ includes the set
\[
\textstyle
\{\sigma_I \in \symm{k} \colon I \in \couples[k+1]\} = \{(1, 4, 2, 3), (2, 1, 4, 3), (3, 1, 2, 4), (3, 1, 4, 2)\}.
\]
The differences of the elements of the set $\{\sigma_I \in \symm{k} \colon I \in \couples[k+1]\}$ are contained in $\Delta^\sigma_k$, and they are indicated in the table below (the entry in row $\rho$ in column $\tau$ is the permutation $\rho^{-1} \tau$).

\medskip
\begin{center}
\begin{tabular}{c|cccc}
               & $(1, 4, 2, 3)$ & $(2, 1, 4, 3)$ & $(3, 1, 2, 4)$ & $(3, 1, 4, 2)$ \\
\hline
$(1, 4, 2, 3)$ & $\id$          & $(3, 1, 2, 4)$ & $(4, 1, 3, 2)$ & $(4, 1, 2, 3)$ \\
$(2, 1, 4, 3)$ & $(2, 3, 1, 4)$ & $\id$          & $(4, 2, 1, 3)$ & $(4, 2, 3, 1)$ \\
$(3, 1, 2, 4)$ & $(2, 4, 3, 1)$ & $(3, 2, 4, 1)$ & $\id$          & $(1, 2, 4, 3)$ \\
$(3, 1, 4, 2)$ & $(2, 3, 4, 1)$ & $(4, 2, 3, 1)$ & $(1, 2, 4, 3)$ & $\id$
\end{tabular}
\end{center}
\medskip

Since $\Delta^\sigma_k$ contains $(2, 3, 4, 1)$ and $(1, 2, 4, 3)$, $G^\sigma_k$ is the full symmetric group $\symm{k}$.

The reader is invited to verify that $U^\sigma_k$ and $\Delta^\sigma_k$ actually contain a few other permutations than the ones found above, namely
\begin{align*}
U^\sigma_k &=
\{
(1, 2, 4, 3),
(1, 3, 2, 4),
(1, 3, 4, 2),
(1, 4, 2, 3),
(2, 1, 3, 4),
(2, 1, 4, 3), \\
{}&\phantom{{}= \{}
(3, 1, 2, 4),
(3, 1, 4, 2)
\},
\\
\Delta^\sigma_k &=
\{
(1, 2, 3, 4),
(1, 2, 4, 3),
(1, 3, 2, 4),
(1, 3, 4, 2),
(1, 4, 2, 3),
(1, 4, 3, 2), \\
{}&\phantom{{}= \{}
(2, 1, 3, 4),
(2, 1, 4, 3),
(2, 3, 1, 4),
(2, 3, 4, 1),
(2, 4, 1, 3),
(2, 4, 3, 1), \\
{}&\phantom{{}= \{}
(3, 1, 2, 4),
(3, 1, 4, 2),
(3, 2, 1, 4),
(3, 2, 4, 1),
(4, 1, 2, 3),
(4, 1, 3, 2), \\
{}&\phantom{{}= \{}
(4, 2, 1, 3),
(4, 2, 3, 1)
\}.
\end{align*}
\end{example}

We can now apply Lemma~\ref{lem:vsigmaU} and Remark~\ref{rem:sigmaI} to establish rules for determining some members of $G^\sigma_k$ directly from the one-line notation for $\sigma$. We are going to work with permutations of the form $\sigma_I$ where $I \in \couples[k+1]$. Therefore, by Remark~\ref{rem:sigmaI}, only the entries of $\sigma$ that are less than or equal to $k + 1$ are relevant, that is, we are only concerned about $\check{\sigma}_{k+1}$.
The following lemma is summarized for easy reference in Table~\ref{table:Gsigma} in the form of deduction rules: if $\check{\sigma}_{k+1}$ has the configuration indicated in the middle column, then $G^\sigma_k$ contains the permutations indicated in the last column.

\begin{lemma}
\label{lem:sigmarules}
Let $\alpha, \beta, \gamma, \delta \in \nset{k+1}$, and assume that $\alpha > x$ whenever $x \in \{\beta, \gamma, \delta\}$. Let $p, q, r, s \in \nset{k+1}$ and assume that $p < q < r < s$.
\begin{enumerate}[\rm (i)]
\item\label{alphabetagamma}
If $\check{\sigma}_{k+1}(p) = \alpha$, $\check{\sigma}_{k+1}(q) = \beta$ and $\check{\sigma}_{k+1}(r) = \gamma$, then $(p \;\; q \;\; \cdots \;\; r-1) \in \Delta^\sigma_k$.

\item\label{alphabetagammadelta}
If $\check{\sigma}_{k+1}(p) = \alpha$, $\check{\sigma}_{k+1}(q) = \beta$, $\check{\sigma}_{k+1}(q+1) = \gamma$, and $\check{\sigma}_{k+1}(r) = \delta$, then $\symm{\{p, q, \dots, r-1\}} \subseteq G^\sigma_k$.

\item\label{betaalphagamma}
If $\check{\sigma}_{k+1}(p) = \beta$, $\check{\sigma}_{k+1}(q) = \alpha$ and $\check{\sigma}_{k+1}(r) = \gamma$, then $(q \;\; \cdots \;\; r-1) \in \Delta^\sigma_k$.

\item\label{alphaalpha+1beta}
If $\check{\sigma}_{k+1}(p) = \alpha$, $\check{\sigma}_{k+1}(q) = \alpha + 1$ and $\check{\sigma}_{k+1}(r) = \beta$ with $\alpha \leq k$, then $(p \;\; q \;\; \cdots \;\; r-1) \in \Delta^\sigma_k$ and consequently $\symm{\{q, \dots, r-1\}} \subseteq G^\sigma_k$.

\item\label{alphabetaalpha+1}
If $\check{\sigma}_{k+1}(p) = \alpha$, $\check{\sigma}_{k+1}(q) = \beta$ and $\check{\sigma}_{k+1}(r) = \alpha + 1$ with $\alpha \leq k$, then $(p \;\; q \;\; \cdots \;\; r-1) \in \Delta^\sigma_k$.

\item\label{alphabetaalpha+1gamma}
If $\check{\sigma}_{k+1}(p) = \alpha$, $\check{\sigma}_{k+1}(q) = \beta$, $\check{\sigma}_{k+1}(r) = \alpha + 1$, and $\check{\sigma}_{k+1}(s) = \gamma$ with $\alpha \leq k$, then $\symm{\{p, q, \dots, s - 1\}} \subseteq G^\sigma_k$.

\item\label{betaalphaalpha+2alpha+1}
If $\check{\sigma}_{k+1}(p) = \beta$, $\check{\sigma}_{k+1}(q) = \alpha$, $\check{\sigma}_{k+1}(q+1) = \alpha + 2$, and $\check{\sigma}_{k+1}(q+2) = \alpha + 1$ with $\alpha \leq k - 1$, then $(q \;\; q + 1) \in \Delta^\sigma_k$.
\end{enumerate}
\end{lemma}

\begin{table}
\begin{tabular}{|ccc|}
\hline
Lemma~\ref{lem:sigmarules} & If $\check{\sigma}_{k+1}$ is of the form\ldots, & then $G^\sigma_k$ contains\ldots \\
\hline
\eqref{alphabetagamma} &
$(\dots, \pos{p}{\alpha}, \dots, \pos{q}{\beta}, \dots, \pos{r}{\gamma}, \dots)$ &
$(p \;\; q \;\; \cdots \;\; r-1)$ \\
\eqref{alphabetagammadelta} &
$(\dots, \pos{p}{\alpha}, \dots, \pos{q}{\beta}, \pos{q+1}{\gamma}, \dots, \pos{r}{\delta}, \dots)$ &
$\symm{\{p, q, \dots, r-1\}}$ \\
\eqref{betaalphagamma} &
$(\dots, \pos{p}{\beta}, \dots, \pos{q}{\alpha}, \dots, \pos{r}{\gamma}, \dots)$ &
$(q \;\; \cdots \;\; r-1)$ \\
\eqref{alphaalpha+1beta} &
$(\dots, \pos{p}{\alpha}, \dots, \pos{q}{\alpha + 1}, \dots, \pos{r}{\beta}, \dots)$ &
$\symm{\{p, q, \dots, r-1\}}$ \\
\eqref{alphabetaalpha+1} &
$(\dots, \pos{p}{\alpha}, \dots, \pos{q}{\beta}, \dots, \pos{r}{\alpha + 1}, \dots)$ &
$(p \;\; q \;\; \cdots \;\; r-1)$ \\
\eqref{alphabetaalpha+1gamma} &
$(\dots, \pos{p}{\alpha}, \dots, \pos{q}{\beta}, \dots, \pos{r}{\alpha + 1}, \dots, \pos{s}{\gamma}, \dots)$ &
$\symm{\{p, q, \dots, s-1\}}$ \\
\eqref{betaalphaalpha+2alpha+1} &
$(\dots, \pos{p}{\beta}, \dots, \pos{q}{\alpha}, \pos{q+1}{\alpha + 2}, \pos{q+2}{\alpha + 1}, \dots)$ &
$(q \;\; q+1)$ \\
\hline
\end{tabular}
\bigskip
\caption{Rules for deducing the membership of certain permutations in $G^\sigma_k$ from the one-line notation for $\sigma$. We assume that $p < q < r < s$ and $\alpha > x$ whenever $x \in \{\beta, \gamma, \delta\}$.}
\label{table:Gsigma}
\end{table}

\begin{proof}
\eqref{alphabetagamma}
We have
\begin{align*}
\check{\sigma}_{k+1} &= (\dots, \pos{p}{\alpha}, \dots, \pos{q}{\beta}, \dots, \pos{r}{\gamma}, \dots). \\
\intertext{By Lemma~\ref{lem:vsigmaU}, the following permutations are members of $U^\sigma_k$:}
\sigma_{\{\alpha, \beta\}} &= (\dots, \pos{p}{\beta}, \dots, \pos{q}{\phantom{\beta}} \phantom{,} \dots, \pos{r-1}{\gamma}, \dots)^{(\alpha)}, \\
\sigma_{\{\alpha, \gamma\}} &= (\dots, \pos{p}{\gamma}, \dots, \pos{q}{\beta}, \dots, \pos{r}{\phantom{\gamma}}\phantom{,} \dots)^{(\alpha)}.
\end{align*}
Consequently, $(\sigma_{\{\alpha, \gamma\}})^{-1} \sigma_{\{\alpha, \beta\}} = (p \;\; q \;\; \cdots \;\; r-1)$ is a member of $\Delta^\sigma_k$.

\eqref{alphabetagammadelta}
By part \eqref{alphabetagamma}, $\Delta^\sigma_k$ contains $(p \;\; q \;\; \cdots \;\; r-1)$ and $(p \;\; q+1 \;\; \cdots \;\; r-1)$. These permutations constitute a generating set of $\symm{\{p, q, \dots, r-1\}}$.

\eqref{betaalphagamma}
We have
\begin{align*}
\check{\sigma}_{k+1} &= (\dots, \pos{p}{\beta}, \dots, \pos{q}{\alpha}, \dots, \pos{r}{\gamma}, \dots). \\
\intertext{By Lemma~\ref{lem:vsigmaU}, the following permutations are members of $U^\sigma_k$:}
\sigma_{\{\alpha, \beta\}} &= (\dots, \pos{p}{\beta}, \dots, \pos{q}{\phantom{\alpha}} \phantom{,} \dots, \pos{r-1}{\gamma}, \dots)^{(\alpha)}, \\
\sigma_{\{\alpha, \gamma\}} &= (\dots, \pos{p}{\beta}, \dots, \pos{q}{\gamma}, \dots, \pos{r}{\phantom{\gamma}}\phantom{,} \dots)^{(\alpha)}.
\end{align*}
Consequently, $(\sigma_{\{\alpha, \gamma\}})^{-1} \sigma_{\{\alpha, \beta\}} = (q \;\; \cdots \;\; r-1)$ is a member of $\Delta^\sigma_k$.

\eqref{alphaalpha+1beta}
\newlength{\wida}\settowidth{\wida}{$\alpha + 1$}
We have
\begin{align*}
\check{\sigma}_{k+1} &= (\dots, \pos{p}{\alpha}, \dots, \pos{q}{\alpha + 1}, \dots, \pos{r}{\beta}, \dots). \\
\intertext{By Lemma~\ref{lem:vsigmaU}, the following permutations are members of $U^\sigma_k$:}
\sigma_{\{\alpha, \alpha+1\}} &= (\dots, \pos{p}{\alpha}, \dots, \pos{q}{\phantom{\alpha + 1}}\phantom{,} \dots, \pos{r-1}{\beta}, \dots)^{(\alpha+1)}, \\
\sigma_{\{\alpha, \beta\}} &= (\dots, \pos{p}{\beta}, \dots, \pos{q}{\alpha + 1}, \dots, \pos{r}{\phantom{\beta}}\phantom{,} \dots)^{(\alpha)} \\
&= (\dots, \pos{p}{\beta}, \dots, \pos{q}{\makebox[\wida][c]{$\alpha$}}, \dots, \pos{r}{\phantom{\beta}}\phantom{,} \dots)^{(\alpha+1)}.
\end{align*}
Consequently, $(\sigma_{\{\alpha, \beta\}})^{-1} \sigma_{\{\alpha, \alpha+1\}} = (p \;\; q \;\; \cdots \;\; r-1)$ is a member of $\Delta^\sigma_k$.
By part \eqref{betaalphagamma}, $\Delta^\sigma_k$ also contains $(q \;\; \cdots \;\; r-1)$. These permutations constitute a generating set of $\symm{\{p, q, \dots, r-1\}}$.

\eqref{alphabetaalpha+1}
We have
\begin{align*}
\check{\sigma}_{k+1} &= (\dots, \pos{p}{\alpha}, \dots, \pos{q}{\beta}, \dots, \pos{r}{\alpha + 1}, \dots). \\
\intertext{By Lemma~\ref{lem:vsigmaU}, the following permutations are members of $U^\sigma_k$:}
\sigma_{\{\alpha, \beta\}} &= (\dots, \pos{p}{\beta}, \dots, \pos{q}{\phantom{\beta}}\phantom{,} \dots, \pos{r-1}{\alpha + 1}, \dots)^{(\alpha)} \\
&= (\dots, \pos{p}{\beta}, \dots, \pos{q}{\phantom{\beta}}\phantom{,} \dots, \pos{r-1}{\makebox[\wida][c]{$\alpha$}}, \dots)^{(\alpha + 1)}, \\
\sigma_{\{\alpha, \alpha + 1\}} &= (\dots, \pos{p}{\alpha}, \dots, \pos{q}{\beta}, \dots, \pos{r}{\phantom{\alpha + 1}}\phantom{,} \dots)^{(\alpha + 1)}.
\end{align*}
Consequently, $(\sigma_{\{\alpha, \alpha + 1\}})^{-1} \sigma_{\{\alpha, \beta\}} = (p \;\; q \;\; \cdots \;\; r-1)$ is a member of $\Delta^\sigma_k$.

\eqref{alphabetaalpha+1gamma}
By part \eqref{alphaalpha+1beta}, $\symm{\{p, r, \dots, s-1\}} \subseteq G^\sigma_k$; hence $G^\sigma_k$ contains $(p \;\, r)$. By part \eqref{alphabetaalpha+1}, $\Delta^\sigma_k$ contains $(p \;\, q \;\, \cdots \;\, r-1)$. These permutations constitute a generating set of $\symm{\{p, q, \dots, r\}}$. Thus, $G^\sigma_k$ contains both $\symm{\{p, r, \dots, s-1\}}$ and $\symm{\{p, q, \dots, r\}}$, which together generate $\symm{\{p, q, \dots, s-1\}}$.

\eqref{betaalphaalpha+2alpha+1}
We have
\begin{align*}
\check{\sigma}_{k+1} &= (\dots, \pos{p}{\beta}, \dots, \pos{q}{\alpha}, \pos{q+1}{\alpha + 2}, \pos{q+2}{\alpha + 1}, \dots). \\
\intertext{By Lemma~\ref{lem:vsigmaU}, the following permutations are members of $U^\sigma_k$:}
\sigma_{\{\alpha, \beta\}} &= (\dots, \pos{p}{\beta}, \dots, \phantom{\alpha,} \pos{q}{\alpha + 2}, \pos{q+1}{\alpha + 1}, \dots)^{(\alpha)} \\
&= (\dots, \pos{p}{\beta}, \dots, \phantom{\alpha,} \pos{q}{\alpha + 1}, \pos{q+1}{\makebox[\wida][c]{$\alpha$}}, \dots)^{(\alpha + 2)}, \\
\sigma_{\{\beta, \alpha + 1\}} &= (\dots, \pos{p}{\beta}, \dots, \pos{q}{\alpha}\phantom{,} \pos{q+1}{\alpha + 2}, \pos{q+2}{\phantom{\alpha + 1}}\phantom{,} \dots)^{(\alpha + 1)} \\
&= (\dots, \pos{p}{\beta}, \dots, \pos{q}{\alpha}\phantom{,} \pos{q+1}{\alpha + 1}, \pos{q+2}{\phantom{\alpha + 1}}\phantom{,} \dots)^{(\alpha + 2)}.
\end{align*}
Consequently, $(\sigma_{\{\beta, \alpha + 1\}})^{-1} \sigma_{\{\alpha, \beta\}} = (q \;\, q+1)$ is a member of $\Delta^\sigma_k$.
\end{proof}

Let $\rho \in \symm{m}$ and $S \subseteq \nset{m}$. We say that $\rho$ \emph{preserves} $S$, if $\rho(i) \in S$ for all $i \in S$. Let $\Pi$ be a partition of $\nset{m}$. We say that $\rho$ is \emph{compatible} with $\Pi$, if $\rho$ preserves every block of $\Pi$. We say that $\Pi$ is a \emph{partition into intervals} if each block of $\Pi$ is an interval, i.e., a set of the form $\{i \in \nset{m} : a \leq i \leq b\}$ for some $a, b \in \nset{m}$ with $a \leq b$.

The \emph{fundamental partition} of a permutation $\rho \in \symm{m}$ is the most refined partition of $\nset{m}$ into intervals with which $\rho$ is compatible. (It is easy to verify that this definition is good: the most refined compatible partition into intervals is unique. For, if $\Pi$ and $\Pi'$ are partitions of $\nset{m}$ into intervals with which $\rho$ is compatible, then the partition whose blocks are the nonempty intersections of the blocks of $\Pi$ and $\Pi'$ is a common refinement of $\Pi$ and $\Pi'$, its blocks are intervals, and $\rho$ it compatible with it.)

\begin{example}
We present here the fundamental partitions of some permutations on the set $\nset{7}$. The permutations are given in one-line notation, and the partitions are indicated with boxes grouping together the members of each block.
\[
\begin{gathered}
\left( \framebox{\smash[b]{1,}}\, \framebox{\smash[b]{2,}}\, \framebox{\smash[b]{3,}}\, \framebox{\smash[b]{4,}}\, \framebox{\smash[b]{5,}}\, \framebox{\smash[b]{6,}}\, \framebox{\smash[b]{7}} \right) \\
\left( \framebox{\smash[b]{1,}}\, \framebox{\smash[b]{2,}}\, \framebox{\smash[b]{3,}}\, \framebox{\smash[b]{5, 6, 4,}}\, \framebox{\smash[b]{7}} \right)
\end{gathered}
\quad
\begin{gathered}
\left( \framebox{\smash[b]{2, 1,}}\, \framebox{\smash[b]{3,}}\, \framebox{\smash[b]{4,}}\, \framebox{\smash[b]{7, 6, 5}} \right) \\
\left( \framebox{\smash[b]{1,}}\, \framebox{\smash[b]{3, 2,}}\, \framebox{\smash[b]{5, 4,}}\, \framebox{\smash[b]{7, 6}} \right)
\end{gathered}
\quad
\begin{gathered}
\left( \framebox{\smash[b]{7, 2, 3, 4, 5, 6, 1}} \right) \\
\left( \framebox{\smash[b]{4, 3, 1, 7, 6, 2, 5}} \right)
\end{gathered}
\]
\end{example}

Let $S$ and $T$ be sets of integers. We write $S < T$ if $s < t$ for every $s \in S$ and $t \in T$.

\begin{lemma}
\label{lem:fundamentalthetan}
Let $\rho \in \symm{n}$, and let $\Pi = \{B_1, \dots, B_r\}$ be the fundamental partition of $\rho$, and assume that $B_i < B_j$ whenever $i < j$.
\begin{itemize}[\rm (i)]
\item $\card{B_i} = 2$ for all $i \in \nset{r}$ if and only if $n$ is even and $\rho = \theta_n$.
\item $\card{B_1} = 1$ and $\card{B_i} = 2$ for all $i \in \{2, \dots, r\}$ if and only if $n$ is odd and $\rho = \theta_n$.
\end{itemize}
\end{lemma}

\begin{proof}
It is straightforward to verify that the fundamental partition of $\theta_n$ has the prescribed block sizes. On the other hand, if the fundamental partition of $\rho$ has a block $B$ of cardinality $2$, then $\rho$ must transpose the elements of $B$, for otherwise the partition could be refined by replacing $B$ with the two singleton blocks $\{\min B\}$ and $\{\max B\}$. Moreover, the element of any singleton block is fixed by $\rho$. Thus any permutation having the prescribed block sizes is necessarily $\theta_n$.
\end{proof}

We are now going to apply Lemma~\ref{lem:sigmarules} to find some members of $G^\sigma_k$. We first look into the configurations occurring within a single block of the fundamental partition of $\check{\sigma}_{k+1}$ (Lemma~\ref{lem:FPIi-max}). Then we examine configurations that spread over two consecutive blocks (Lemma~\ref{lem:FPIi}).

\begin{lemma}
\label{lem:FPIi-max}
Let $B$ be a block of the fundamental partition of $\check{\sigma}_{k+1}$. Then $\symm{B \setminus \{\max B\}} \subseteq G^\sigma_k$.
\end{lemma}
\begin{proof}
Since the blocks of fundamental partitions are intervals, there exist $a, b \in \nset{k+1}$ such that $a \leq b$ and $B = \{a, \dots, b\}$. If $b - a \leq 1$, then the claim is trivial. For, in this case $B$ equals either $\emptyset$ or $\{a\}$, and $\Sigma_\emptyset = \symm{\{a\}} = \{\id\} \subseteq G^\sigma_k$. We may thus assume that $b - a \geq 2$.

Let $c := (\check{\sigma}_{k+1})^{-1}(b)$. Since $B$ is a block of a fundamental partition, we must have that $c < b$; for, if $c = b$, i.e., $\check{\sigma}_{k+1}(b) = b$, then we could refine the partition and take $\{a, \dots, b-1\}$ and $\{b\}$ as new blocks. We claim that $\symm{\{c, \dots, b-1\}} \subseteq G^\sigma_k$. If $c = b - 1$, then this claim holds trivially, as $\symm{\{b-1\}} = \{\id\}$. If $c = b - 2$, then we have $(b - 2 \;\; b - 1) \in G^\sigma_k$ by Lemma~\ref{lem:sigmarules}~\eqref{alphabetagamma}, and the claim clearly holds. If $c \leq b - 3$, then the claim holds by Lemma~\ref{lem:sigmarules}~\eqref{alphabetagammadelta}.

If $c = a$, then we are done. Otherwise, we have $a < c < b$, and we will proceed applying the following claim.

\begin{claim}
\label{lem:FPIi-max:clm1}
Assume that $\check{\sigma}_{k+1}(a) \neq b$ and there is $c$ such that $a < c < b$, $\check{\sigma}_{k+1}(c) > c$, and $\symm{\{c, \dots, b-1\}} \subseteq G^\sigma_k$. Then there exists $d$ such that $a \leq d < c$, $\check{\sigma}_{k+1}(d) > d$, and $\symm{\{d, \dots, b-1\}} \subseteq G^\sigma_k$.
\end{claim}

\begin{pfclaim}[Proof of Claim~\ref{lem:FPIi-max:clm1}]
Let $m = \max \{\check{\sigma}_{k+1}(i) : a \leq i < c\}$, and let $d = (\check{\sigma}_{k+1})^{-1}(m)$. Since $B$ is a block of the fundamental partition of $\check{\sigma}_{k+1}$, we must have that $m \geq c$ (hence we have $\check{\sigma}_{k+1}(d) = \check{\sigma}_{k+1}((\check{\sigma}_{k+1})^{-1}(m)) = m \geq c > d$); for, otherwise we could refine the partition and take $\{a, \dots, c - 1\}$ and $\{c, \dots, b\}$ as new blocks. Consequently, there exists $q \in \{c, \dots, b\}$ such that $\check{\sigma}_{k+1}(q) < c$; in fact $q \neq c$. We are going to show next that $\symm{\{d, \dots, b-1\}} \subseteq G^\sigma_k$.

Consider first the case that $c = b - 1$. Then we have $\check{\sigma}_{k+1}(b - 1) = b$, $m = b - 1$, and $q = b$. If $d = b - 2$, then by Lemma~\ref{lem:sigmarules}~\eqref{alphaalpha+1beta} we have $\symm{\{b-2, b-1\}} = \symm{\{d, \dots, b-1\}} \subseteq G^\sigma_k$. If $d \leq b - 3$, then by Lemma~\ref{lem:sigmarules}~\eqref{alphabetaalpha+1gamma} we have $\symm{\{d, \dots, b-1\}} \subseteq G^\sigma_k$.

Consider then the case that $a + 2 \leq c \leq b - 2$. If $d > a$, then by Lemma~\ref{lem:sigmarules}~\eqref{betaalphagamma}, we have $(d \; \cdots \; q-1) \in G^\sigma_k$. If $d = a$, then we have again $(d \; \cdots \; q-1) \in G^\sigma_k$ by Lemma~\ref{lem:sigmarules}~\eqref{alphabetagamma}. The permutations in $\symm{\{c, \dots, b-1\}}$ together with $(d \; \cdots \; q-1)$ generate $\symm{\{d, \dots, b-1\}}$; thus we have $\symm{\{d, \dots, b-1\}} \in G^\sigma_k$.

Finally, consider the case that $c = a + 1$. Then $d = a$, so the assumption $\check{\sigma}_{k+1}(a) \neq b$ implies that there exists $\ell$ such that $\check{\sigma}_{k+1}(\ell) = m + 1$. If $\ell < q$, then by Lemma~\ref{lem:sigmarules}~\eqref{alphaalpha+1beta} we have $\symm{\{a, \ell, \dots, q-1\}} \subseteq G^\sigma_k$. The permutations in $\symm{\{a, \ell, \dots, q-1\}}$ and $\symm{\{a+1, \dots, b-1\}}$ generate $\symm{\{a, \dots, b-1\}}$; thus we have $\symm{\{d, \dots, b-1\}} \in G^\sigma_k$. If $\ell > q$, then by Lemma~\ref{lem:sigmarules}~\eqref{alphabetaalpha+1} we have $(a \; q \; \cdots \; \ell - 1) \in G^\sigma_k$. The permutations in $\symm{\{a+1, \dots, b-1\}}$ together with $(a \; q \; \cdots \; \ell - 1)$ generate $\symm{\{a, \dots, b-1\}}$; thus we have $\symm{\{d, \dots, b-1\}} \in G^\sigma_k$. This completes the proof.
\end{pfclaim}

(Proof of Lemma~\ref{lem:FPIi-max} continued)
A finite number of applications of Claim~\ref{lem:FPIi-max:clm1} yields that $\symm{\{a, \dots, b-1\}} \subseteq G^\sigma_k$.
\end{proof}

\begin{lemma}
\label{lem:FPIi}
Let $\Pi = \{B_1, \dots, B_r\}$ be the fundamental partition of $\check{\sigma}_{k+1}$, and assume that $B_i < B_j$ whenever $i < j$.
\begin{enumerate}[\rm (i)]
\item\label{lem:FPIinotboth2}
For every $i \in \nset{r-1}$, if $\card{B_i} \geq 2$, and $\card{B_i}$ and $\card{B_{i+1}}$ are not both equal to $2$, then $G^\sigma_k$ includes $\symm{S}$, where $S = B_i \cup B_{i+1} \setminus \{\max B_{i+1}\}$.
\item\label{lem:FPIiboth2}
If $\card{B_i} = \card{B_{i+1}} = 2$, then $(\min B_i \;\; \max B_i \;\; \min B_{i+1}) \in G^\sigma_k$.
\item\label{lem:FPIi2row}
If there exist $s, t \in \nset{r}$ such that $s < t$ and $\card{B_i} = 2$ whenever $s \leq i \leq t$, then $G^\sigma_k$ includes $\alt{S}$, where $S = (\bigcup_{i = s}^t B_i) \setminus \{\max B_t\}$.
\item\label{lem:FPIi2rowup}
If there exist $s, t \in \nset{r}$ such that $s < t < r$, $\card{B_i} = 2$ whenever $s \leq i \leq t$, and $\card{B_{t+1}} \neq 2$, then $G^\sigma_k$ includes $\symm{S}$, where $S = \bigcup_{i = s}^t B_i$.
\item\label{lem:FPIi2rowdown1}
If there exist $s, t \in \nset{r}$ such that $3 \leq s < t$, $\card{B_i} = 2$ whenever $s \leq i \leq t$, and $\card{B_{s-1}} = 1$, then $G^\sigma_k$ includes $\symm{S}$, where $S = (\bigcup_{i = s-1}^t B_i) \setminus \{\max B_t\}$.
\item\label{lem:FPIi2rowdownbig}
If there exist $s, t \in \nset{r}$ such that $2 \leq s < t$, $\card{B_i} = 2$ whenever $s \leq i \leq t$, and $\card{B_{s-1}} > 2$, then $G^\sigma_k$ includes $\symm{S}$, where $S = (\bigcup_{i = s-1}^t B_i) \setminus \{\max B_t\}$.
\end{enumerate}
\end{lemma}

\begin{proof}
By Lemma~\ref{lem:FPIi-max}, $G^\sigma_k$ includes $\symm{B_i \setminus \{\max B_i\}}$ for every $i \in \nset{r}$.

\eqref{lem:FPIinotboth2}
Let $i \in \nset{r-1}$.
Let $\alpha := \max B_i$, $p := (\check{\sigma}_{k+1})^{-1}(\alpha)$, and $\beta := \check{\sigma}_{k+1}(\alpha)$. Since $B_i$ is a block of the fundamental partition of $\check{\sigma}_{k+1}$, we have that $p < \alpha$ and $\beta < \alpha$. Block $B_{i+1}$ contains $\alpha + 1$; let $r := (\check{\sigma}_{k+1})^{-1}(\alpha + 1)$. By Lemma~\ref{lem:sigmarules}~\eqref{alphabetaalpha+1} we have $(p \;\; \alpha \;\; \cdots \;\; r-1) \in G^\sigma_k$. 

Assume first that $\card{B_{i+1}} = 1$. Then $r = \alpha + 1$, and $G^\sigma_k$ includes $\symm{B_i \setminus \{\alpha\}}$ and $(p \;\; \alpha)$, which generate $\symm{B_i}$. The claim thus holds because $B_i \cup B_{i+1} \setminus \{\max B_{i+1}\} = B_i$.

Assume then that $\card{B_{i+1}} \geq 2$; then $r > \alpha + 1$.
If $\card{B_i} > 2$ and $\card{B_{i+1}} = 2$, then $r = \alpha + 2$ and $G^\sigma_k$ includes $\symm{B_i \setminus \{\alpha\}}$ and $(p \;\; \alpha \;\; \alpha + 1)$, which generate $\symm{S}$, where $S = B_i \cup \{\alpha + 1\} = B_i \cup B_{i+1} \setminus \{\max B_{i+1}\}$.
If $\card{B_i} = 2$ and $\card{B_{i+1}} > 2$, then $B_i = \{p, \alpha\}$ and $r > \alpha + 1$, and $G^\sigma_k$ includes $\symm{B_{i+1} \setminus \{\max B_{i+1}\}}$ and $(p \;\; \alpha \;\; \cdots \;\; r - 1)$, which generate $\symm{T}$, where $T = B_i \cup B_{i+1} \setminus \{\max B_{i+1}\}$.

\eqref{lem:FPIiboth2}
We continue the argument from part~\eqref{lem:FPIinotboth2}. If $\card{B_i} = \card{B_{i+1}} = 2$, then we have that $(p \;\; \alpha \;\; \alpha+1) = (\min B_i \;\; \max B_i \;\; \min B_{i+1}) \in G^\sigma_k$.

\eqref{lem:FPIi2row}
By part~\eqref{lem:FPIiboth2}, $G^\sigma_k$ contains the cycle $(\min B_i \;\; \max B_i \;\; \min B_{i+1})$ for each $i \in \{s, \dots, t-1\}$. These generate $\alt{S}$, where $S = (\bigcup_{i = s}^t B_i) \setminus \{\max B_t\}$.

\eqref{lem:FPIi2rowup}
By part~\eqref{lem:FPIi2row}, $G^\sigma_k$ includes $\alt{S}$, where $S = (\bigcup_{i = s}^t I_i) \setminus \{\max I_t\}$. By part~\eqref{lem:FPIinotboth2}, $G^\sigma_k$ includes $\symm{B_t}$. Thus, $G^\sigma_k$ includes a generating set of $\symm{T}$, where $T = \bigcup_{i = s}^t I_i$.

\eqref{lem:FPIi2rowdown1}
By part~\eqref{lem:FPIi2row}, $G^\sigma_k$ includes $\alt{S}$, where $S = (\bigcup_{i = s}^t B_i) \setminus \{\max B_t\}$.
Let $\alpha$ be the unique element of $B_{s-1}$. We have the following configuration in $\check{\sigma}_{k+1}$:
\[
(\dots, \pos{\alpha - 1}{\beta}, \pos{\alpha}{\alpha}, \pos{\alpha+1}{\alpha+2}, \pos{\alpha+2}{\alpha+1}, \dots),
\]
where $\beta < \alpha$. By Lemma~\ref{lem:sigmarules}~\eqref{betaalphaalpha+2alpha+1}, we have $(\alpha \;\; \alpha+1) = (\min I_s - 1 \;\; \min I_s) \in G^\sigma_k$.
Thus $G^\sigma_k$ includes a generating set of $\symm{T}$, where $T = (\bigcup_{i = s-1}^t B_i) \setminus \{\max B_t\}$.

\eqref{lem:FPIi2rowdownbig}
By part~\eqref{lem:FPIi2row}, $G^\sigma_k$ includes $\alt{S}$, where $S = (\bigcup_{i = s}^t B_i) \setminus \{\max B_t\}$.
By part~\eqref{lem:FPIinotboth2}, $G^\sigma_k$ includes $\symm{B_{s-1}}$.
Let $\alpha := \min B_s$. We have the following configuration in $\check{\sigma}_{k+1}$:
\[
(\dots, \pos{p}{\alpha - 1}, \dots, \pos{\alpha-1}{\beta}, \pos{\alpha}{\alpha+1}, \pos{\alpha+1}{\alpha}, \dots),
\]
where $\beta < \alpha - 1$ and $\min B_{s-1} \leq p < \alpha - 1 = \max B_{s-1}$.
By Lemma~\ref{lem:sigmarules}~\eqref{alphabetaalpha+1}, $G^\sigma_k$ contains $(p \;\; \alpha-1 \;\; \alpha)$.
Thus $G^\sigma_k$ includes a generating set of $\symm{T}$, where $T = (\bigcup_{i = s-1}^t B_i) \setminus \{\max B_t\}$.
\end{proof}

\begin{lemma}
\label{lem:vsigmaG}
Let $\sigma \in \symm{n}$, and let $\Pi = \{B_1, \dots, B_r\}$ be the fundamental partition of $\check{\sigma}_{k+1}$, and assume that $B_i < B_j$ whenever $i < j$.
Then $\pi^\sigma_k \in G^\sigma_k$,
unless
$n = k + 1 \equiv 0 \pmod{4}$ and $\card{B_i} = 2$ for all $i \in \nset{r}$,
or
$n = k + 1 \equiv 1 \pmod{4}$ and $\card{B_1} = 1$ and $\card{B_i} = 2$ whenever $2 \leq i \leq r$.
\end{lemma}

\begin{proof}
It is clear from the definitions that $\pi^\sigma_k$ preserves the block $B_i$ for every $i \in \nset{r-1}$ and it also preserves $B_r \setminus \{k+1\}$. Thus, $\pi^\sigma_k = \tau_1 \cdots \tau_r$, where $\tau_i \in \symm{B_i}$ for $i \in \nset{r-1}$ and $\tau_r \in \symm{B_r \setminus \{k+1\}}$.

If it is not the case that all blocks of $\Pi$ have cardinality $2$ or all blocks except $B_1$ have cardinality $2$ and $\card{B_1} = 1$, then it follows from Lemma~\ref{lem:FPIi} that $\symm{B_i} \in G^\sigma_k$ for every $i \in \nset{r-1}$ and $\symm{B_r \setminus \{\max B_r\}} \subseteq G^\sigma_k$; hence $\pi^\sigma_k \in G^\sigma_k$.

Assume then that all blocks of $\Pi$ have cardinality $2$. Then we have $k + 1 = 2r$, i.e., $k + 1 \equiv 0, 2 \pmod{4}$. Lemma~\ref{lem:FPIi}~\eqref{lem:FPIi2row} implies that $G^\sigma_k$ includes $\alt{k}$. If $k + 1 \equiv 2 \pmod{4}$, then $r$ is odd and $\pi^\sigma_k = (1 \; 2) (3 \; 4) \cdots (k-2 \;\; k-1)$ is an even permutation and it is thus contained in $\alt{k}$ and hence in $G^\sigma_k$.

If $k + 1 \equiv 0 \pmod{4}$ and $n > k + 1$, then let
\begin{align*}
\vect{a} &= (1, 1, 2, 2, 3, 4, \dots, k, \underbrace{k, \dots, k}_{n - k - 2}) \in A^n, \\
\vect{b} &= (1, 2, 2, 2, 3, 4, \dots, k, \underbrace{k, \dots, k}_{n - k - 2}) \in A^n.
\end{align*}
We clearly have that $\ofo(\vect{a}) = \ofo(\vect{a}) = \vect{k}$. Furthermore,
\begin{align*}
\vect{a}\sigma &= (\dots, \pos{\sigma^{-1}(2)}{1}, \dots, \pos{\sigma^{-1}(1)}{1}, \dots, \pos{\sigma^{-1}(4)}{2}, \dots, \pos{\sigma^{-1}(3)}{2}, \dots), \\
\vect{b}\sigma &= (\dots, \pos{\sigma^{-1}(2)}{2}, \dots, \pos{\sigma^{-1}(1)}{1}, \dots, \pos{\sigma^{-1}(4)}{2}, \dots, \pos{\sigma^{-1}(3)}{2}, \dots),
\end{align*}
where any possible entry in place of the ellipses preceding position $\sigma^{-1}(3)$ is $k$. Since $\vect{a} \sigma$ and $\vect{b} \sigma$ differ only at position $\sigma^{-1}(2)$, we have
\begin{itemize}
\item $\ofo(\vect{a} \sigma) = (1, 2, \dots)$ and $\ofo(\vect{b} \sigma) = (2, 1, \dots)$, or
\item $\ofo(\vect{a} \sigma) = (k, 1, 2, \dots)$ and $\ofo(\vect{b} \sigma) = (k, 2, 1, \dots)$, or
\item $\ofo(\vect{a} \sigma) = (1, k, 2, \dots)$ and $\ofo(\vect{b} \sigma) = (2, k, 1, \dots)$.
\end{itemize}
Let $\pi$ and $\tau$ be the permutations satisfying $\ofo(\vect{a} \sigma) = \vect{k} \pi$ and $\ofo(\vect{b} \sigma) = \vect{k} \tau$. Then we have $\pi, \tau \in U^\sigma_k$, and $G^\sigma_k$ contains $\pi^{-1} \tau$, which is one of the transpositions $(1 \; 2)$, $(1 \; 3)$, and $(2 \; 3)$. Since the alternating group $\alt{k}$ and a transposition generate the full symmetric group $\symm{k}$, we conclude that $G^\sigma_k = \symm{k}$. Then it obviously holds that $\pi^\sigma_k \in G^\sigma_k$.

Assume then that all blocks of $\Pi$ except $B_1$ have cardinality $2$ and $\card{B_1} = 1$. Then we have $k + 1 = 2 (r - 1) + 1 = 2r - 1$, i.e., $k + 1 \equiv 1, 3 \pmod{4}$. Lemma~\ref{lem:FPIi}~\eqref{lem:FPIi2row} implies that $G^\sigma_k$ includes $\alt{\{2, \dots, k\}}$. If $k + 1 \equiv 3 \pmod{4}$, then $r - 1$ is odd and $\pi^\sigma_k = (2 \; 3)(4 \; 5) \cdots (k-2 \;\; k-1)$ is an even permutation fixing $1$ and it is thus an element of $\alt{\{2, \dots, k\}}$ and hence it is contained in $G^\sigma_k$.

If $k + 1 \equiv 1 \pmod{4}$ and $n > k + 1$, then let
\begin{align*}
\vect{a} &= (1, 2, 2, 3, 3, 4, \dots, k, \underbrace{k, \dots, k}_{n - k - 2}) \in A^n, \\
\vect{b} &= (1, 2, 3, 3, 3, 4, \dots, k, \underbrace{k, \dots, k}_{n - k - 2}) \in A^n.
\end{align*}
We clearly have that $\ofo(\vect{a}) = \ofo(\vect{a}) = \vect{k}$. Furthermore,
\begin{align*}
\vect{a}\sigma &= (\dots, \pos{\sigma^{-1}(1)}{1}, \dots, \pos{\sigma^{-1}(3)}{2}, \dots, \pos{\sigma^{-1}(2)}{2}, \dots, \pos{\sigma^{-1}(5)}{3}, \dots, \pos{\sigma^{-1}(4)}{3}, \dots), \\
\vect{b}\sigma &= (\dots, \pos{\sigma^{-1}(1)}{1}, \dots, \pos{\sigma^{-1}(3)}{3}, \dots, \pos{\sigma^{-1}(2)}{2}, \dots, \pos{\sigma^{-1}(5)}{3}, \dots, \pos{\sigma^{-1}(4)}{3}, \dots),
\end{align*}
where any possible entry in place of the ellipses preceding position $\sigma^{-1}(4)$ is $k$. Since $\vect{a} \sigma$ and $\vect{b} \sigma$ differ only at position $\sigma^{-1}(3)$, we have
\begin{itemize}
\item $\ofo(\vect{a} \sigma) = (1, 2, 3, \dots)$ and $\ofo(\vect{b} \sigma) = (1, 3, 2 \dots)$, or
\item $\ofo(\vect{a} \sigma) = (k, 1, 2, 3, \dots)$ and $\ofo(\vect{b} \sigma) = (k, 1, 3, 2, \dots)$, or
\item $\ofo(\vect{a} \sigma) = (1, k, 2, 3, \dots)$ and $\ofo(\vect{b} \sigma) = (1, k, 3, 2, \dots)$, or
\item $\ofo(\vect{a} \sigma) = (1, 2, k, 3, \dots)$ and $\ofo(\vect{b} \sigma) = (1, 3, k, 2, \dots)$.
\end{itemize}
Let $\pi$ and $\tau$ be the permutations satisfying $\ofo(\vect{a} \sigma) = \vect{k} \pi$ and $\ofo(\vect{b} \sigma) = \vect{k} \tau$. Then we have $\pi, \tau \in U^\sigma_k$, and $G^\sigma_k$ contains $\pi^{-1} \tau$, which is one of the transpositions $(2 \; 3)$, $(2 \; 4)$, and $(3 \; 4)$. Thus $G^\sigma$ includes a generating set of $\symm{\{2, \dots, k\}}$. It is clear that $\pi^\sigma_k$ fixes $1$, and we conclude that $\pi^\sigma_k \in G^\sigma_k$.
\end{proof}

\begin{lemma}
\label{lem:thetanemptyintersection}
If $n = k + 1 \equiv 0, 1 \pmod{4}$ and $\sigma = \theta_n$, then $U^\sigma_k \cap G^\sigma_k = \emptyset$.
\end{lemma}

\begin{proof}
If $\vect{a} \in A^n$ is such that $\ofo(\vect{a}) = \vect{k}$, then $\vect{a} = \vect{k} \delta_I$ for some $I \in \couples$. Let $\ell := \min(\max I, \theta_n(\max I))$. The $\ofo(\vect{a} \theta_n) = \vect{k} \lambda^\ell_k$ by Lemma~\ref{lem:lambda}, and $\lambda^\ell_k$ is an odd permutation by Remark~\ref{rem:thetanlambdaparities}. Thus $U^\sigma_k$ is a set of odd permutations. Therefore every permutation in $\Delta^\sigma_n$ is even, and it follows that $G^\sigma_k$ is a subgroup of the alternating group $\alt{k}$. Consequently, $U^\sigma_k \cap G^\sigma_k = \emptyset$.
\end{proof}

\begin{lemma}
\label{lem:k-equalizing}
Let $n$ and $k$ be integers greater than or equal to $2$, and assume that $n > k$. Let $\sigma \in \symm{n}$.
Then $\sigma$ is $k$-equalizing precisely unless $n = k + 1 \equiv 0, 1 \pmod{4}$ and $\sigma = \theta_n$.
\end{lemma}

\begin{proof}
Let $\sigma \in \symm{n}$, and let $\Pi = \{B_1, \dots, B_r\}$ be the fundamental partition of $\check{\sigma}_{k+1}$, and assume that $B_i < B_j$ whenever $i < j$.

If $n = k + 1 \equiv 0 \pmod{4}$ and $\card{B_i} = 2$ for all $i \in \nset{r}$ or $n = k + 1 \equiv 1 \pmod{4}$ and $\card{B_1} = 1$ and $\card{B_i} = 2$ whenever $2 \leq i \leq r$, then $\sigma = \theta_n$ by Lemma~\ref{lem:fundamentalthetan} and Lemma~\ref{lem:thetannotkequalizing} implies that $\sigma$ is not $k$-equalizing.

Otherwise Lemma~\ref{lem:vsigmaG} guarantees that $\pi^\sigma_k \in G^\sigma_k$. We also have that $\pi^\sigma_k \in U^\sigma_k$ by Lemma~\ref{lem:vsigmaU}. Lemma~\ref{lem:translation} then yields that $\sigma$ is $k$-equalizing.
\end{proof}

\begin{corollary}
Let $\sigma \in \symm{n}$. Then $\sigma$ is $k$-equalizing if and only if $U^\sigma_k \cap G^\sigma_k \neq \emptyset$.
\end{corollary}

\begin{proof}
The statement follows from Lemmas~\ref{lem:translation}, \ref{lem:fundamentalthetan}, \ref{lem:thetanemptyintersection}, and \ref{lem:k-equalizing}.
\end{proof}

We are now in the position to prove the main result of this section.

\begin{proof}[Proof of Theorem~\ref{thm:eqeq}]
Let us first show the necessity of the condition. If $n \leq k$, then $(n, k)$ is not equalizing by Lemma~\ref{lem:equalizing}~\eqref{lem:equalizing:itemnleqk}. If $n = k + 1$ and $k \equiv 0, 3 \pmod{4}$, i.e., $n \equiv 0, 1 \pmod{4}$, then $\theta_n$ is not $k$-equalizing by Lemma~\ref{lem:thetannotkequalizing}; thus $(n, k)$ is not equalizing by Lemma~\ref{lem:indstep}.

For sufficiency, we proceed by induction on $k$. By Lemma~\ref{lem:equalizing}~\eqref{lem:equalizing:item2k-1}, if $k = 2$, then $(n, k)$ is equalizing for every $n \geq 3$; and if $k = 3$, then $(n, k)$ is equalizing for every $n \geq 5$. Let $\ell \geq 3$, and assume that $(n, \ell)$ is equalizing for every $n$ satisfying the following condition: $\ell \equiv 1, 2 \pmod{4}$ and $n \geq k + 1$; or $\ell \equiv 0, 3 \pmod{4}$ and $n \geq k + 2$. Let $m$ be any integer satisfying the following condition: $\ell + 1 \equiv 1, 2 \pmod{4}$ and $m \geq \ell + 2$; or $\ell + 1 \equiv 0, 3 \pmod{4}$ and $m \geq \ell + 3$. By Lemma~\ref{lem:k-equalizing}, every $\sigma \in \symm{m}$ is $(\ell + 1)$-equalizing. Since $m \geq \ell + 2$, the inductive hypothesis yields that $(m, \ell)$ is equalizing. It follows from Lemma~\ref{lem:indstep} that $(m, \ell + 1)$ is equalizing. This completes the proof.
\end{proof}

%%%%%%%%%%%%%%%%%%%%%%%%%%%%%%%%%%%%%%%%%%%%%%%%%%

\section{The class of functions weakly determined by the order of first occurrence is weakly reconstructible}
\label{sec:results}

In this last section we show that the class of functions weakly determined by the order of first occurrence (of sufficiently large arity) are weakly reconstructible (Proposition~\ref{prop:ord}). Throughout this section, we assume that $A = \nset{k}$.

\begin{proposition}
\label{prop:ord}
Let $n$ and $k$ be positive integers such that $k \equiv 1, 2 \pmod{4}$ and $n \geq k + 2$, or $k \equiv 0, 3 \pmod{4}$ and $n \geq k + 3$. Let $f, g \colon A^n \to B$ be functions that are weakly determined by the order of first occurrence. If $\deck f = \deck g$, then $f \equiv g$.
\end{proposition}

\begin{proof}
We may assume without loss of generality that $f$ and $g$ are determined by the order of first occurrence (and not just equivalent to functions determined by the order of first occurrence).
Let $f^*$ and $g^*$ be mappings $A^\sharp \to B$ such that $f = f^* \circ {\ofo}|_{A^n}$ and $g = g^* \circ {\ofo}|_{A^n}$. By Proposition~\ref{prop:minord}, $f_I = f^* \circ {\ofo}|_{A^{n-1}}$ and $g_I = g^* \circ {\ofo}|_{A^{n-1}}$ for all $I \in \couples$. The assumption that $\deck f = \deck g$ implies that $f^* \circ {\ofo}|_{A^{n-1}} \equiv g^* \circ {\ofo}|_{A^{n-1}}$. By Theorem~\ref{thm:eqeq}, $(n - 1, k)$ is equalizing, which implies that $f^* \circ {\ofo}|_{A^{n-1}} = g^* \circ {\ofo}|_{A^{n-1}}$. Hence $f^* = g^*$ by Remark~\ref{rem:restrord}. Consequently, $f = g$, and it clearly holds that $f \equiv g$.
\end{proof}

The following result shows that in Proposition~\ref{prop:ord} the lower bound for $n$ in the case that $k \equiv 0, 3 \pmod{4}$ is sharp, and the answer to Question~\ref{q:rec1} is negative if $\card{A} \equiv 0 \pmod{4}$ or $\card{A} \equiv 3 \pmod{4}$ and $n = \card{A} + 2$.

\begin{proposition}
\label{prop:FnequivG}
Let $n$ and $k$ be positive integers such that $k \equiv 0, 3 \pmod{4}$ and $n = k + 2$. Then there exist functions $f, g \colon A^n \to B$ that are determined by the order of first occurrence such that $f \not\equiv g$ and $f_I \equiv g_J$ for all $I, J \in \couples$.
\end{proposition}

\begin{proof}
Let $\alpha$, $\beta$, $\gamma$ and $\phi^+_k, \psi^+_k \colon A^k_{\neq} \to B$ be as in Definition~\ref{def:phipsi}. Extend $\phi^+_k$ and $\psi^+_k$ into functions $f^*, g^* \colon A^\sharp \to B$ as follows:
\[
f^*(\vect{a}) =
\begin{cases}
\phi^+_k(\vect{a}), & \text{if $\vect{a} \in A^k_{\neq}$,} \\
\gamma, & \text{otherwise,}
\end{cases}
\qquad
g^*(\vect{a}) =
\begin{cases}
\psi^+_k(\vect{a}), & \text{if $\vect{a} \in A^k_{\neq}$,} \\
\gamma, & \text{otherwise.}
\end{cases}
\]
Let $f := f^* \circ {\ofo}|_{A^{k+2}}$ and $g := g^* \circ {\ofo}|_{A^{k+2}}$. By Proposition~\ref{prop:minord}, we have $f_I = f^* \circ {\ofo}|_{A^{k+1}}$ and $g_J = g^* \circ {\ofo}|_{A^{k+1}}$ for all $I, J \in \couples$.
Let $\vect{a} \in A^{k+1}$. If $\vect{a} \in A^k_{\neq}$, then by Lemma~\ref{lem:thetannotkequalizing} we have $f_I(\vect{a}) = \phi^+_k(\ofo(\vect{a})) = \psi^+_k(\ofo(\vect{a} \theta_{k+1})) = g_J(\vect{a} \theta_{k+1})$. If $\vect{a} \notin A^k_{\neq}$, then we have $f_I(\vect{a}) = \gamma = g_J(\vect{a} \theta_{k+1})$. Thus, $f_I \equiv g_J$ for all $I, J \in \couples$.

In order to verify that $f \not\equiv g$, we are going to find, for each permutation $\sigma \in \symm{k+2}$, a $(k+2)$-tuple $\vect{a}$ such that $f(\vect{a}) \neq g(\vect{a} \sigma)$. Denote $\vect{u} := (1, 2, \dots, k, k, k) \in A^{k+2}$, and note that $\ofo(\vect{u}) = \vect{k}$.

Let $\sigma \in \symm{n}$, and let $\rho \in \symm{k}$ be the unique permutation satisfying $\ofo(\vect{u} \sigma) = \vect{k} \rho$. If $k \equiv 0 \pmod{4}$ and $\sigma(1) \neq 1$, then $\ofo(\vect{u} \sigma)$ is a tuple whose first component is distinct from $1$; thus $f(\vect{u}) = \alpha \neq \gamma = g(\vect{u} \sigma)$, and we are done. Hence we assume from now on that if $k \equiv 0 \pmod{4}$, then $\sigma(1) = 1$. Furthermore, if $\rho$ is an even permutation, then $f(\vect{u}) = f^* \circ \ofo(\vect{u}) = \phi^+_k(\vect{k}) = \alpha \neq \beta = \psi^+_k(\vect{k} \rho) = g^* \circ \ofo(\vect{u} \sigma) = g(\vect{u} \sigma)$ and we are done. Hence we assume from now on that $\rho$ is odd. We split the analysis into three cases.

\textit{Case~1:} $\sigma(k + 1) = k + 1$ and $\sigma(k + 2) = k + 2$. Then $\rho$ equals the restriction of $\sigma$ to the set $\nset{k}$. Let $p$ be the largest $i \in \nset{k}$ such that $\sigma(i) \neq i$; such an element exists, because we are assuming that $\rho$ is odd and hence it is a nonidentical permutation. Moreover, $p > 1$ and $\sigma(p) < p$. Set $\vect{a} := (1, 2, \dots, p - 1, \sigma(p), p, p + 1, \dots, k, k)$. It is easy to verify that $\ofo(\vect{a}) = \vect{k}$ and $\ofo(\vect{a} \sigma) = \vect{k} \tau \rho$, where $\tau$ is the transposition of $p$ and $\sigma(p)$. Since both $\tau$ and $\rho$ are odd permutations, $\tau \rho$ is an even permutation. Moreover, if $k \equiv 0 \pmod{4}$, then both $p$ and $\sigma(p)$ are distinct from $1$; hence $\tau(1) = 1$ and $\tau \rho(1) = 1$. Consequently, we have $f(\vect{a}) = \phi^+_k(\vect{k}) = \alpha \neq \beta = \psi^+_k(\vect{k} \tau \rho) = g(\vect{a} \sigma)$ also in this case.

\textit{Case~2:} $\{\sigma(k), \sigma(k+1), \sigma(k+2)\} = \{k, k+1, k+2\}$. We shall assume that $\sigma$ does not fix all three elements of the set $\{k, k+1, k+2\}$, because this is already taken care of by Case~1.
The restriction of $\sigma$ to $\nset{k-1}$ is a permutation of $\nset{k-1}$, and we have $\vect{u} \sigma = (\sigma(1), \dots, \sigma(k-1), k, k, k)$ and $\ofo(\vect{u} \sigma) = (\sigma(1), \dots, \sigma(k-1), k) = \vect{k} \rho$. Thus $\rho$ is the permutation of $\nset{k}$ that fixes $k$ and coincides with $\sigma$ on $\nset{k-1}$.
If $\sigma^{-1}(k-1) \neq 1$, then let $d := \sigma(\sigma^{-1}(k-1) - 1)$; if $\sigma^{-1}(k-1) = 1$, then let $d := \sigma(2)$. Note that $d < k - 1$.
Let
\begin{align*}
\vect{a} &:= (1, 2, \dots, k - 2, d, k - 1, k, k), \\
\vect{b} &:= (1, 2, \dots, k - 2, d, k - 1, k, k - 1), \\
\vect{c} &:= (1, 2, \dots, k - 2, d, d, k - 1, k).
\end{align*}
It is clear that $\ofo(\vect{a}) = \ofo(\vect{b}) = \ofo(\vect{c}) = \vect{k}$. In each one of $\vect{a} \sigma$, $\vect{b} \sigma$, and $\vect{c} \sigma$, the first $k - 1$ entries are
\begin{equation}
\label{eq:firstk-1}
(\sigma(1), \dots, \sigma(\sigma^{-1}(k - 1) - 1), d, \sigma(\sigma^{-1}(k - 1) + 1), \dots, \sigma(k - 1)),
\end{equation}
and the last three entries, for every possible $\sigma$, are presented in Table~\ref{table:last3}.
\begin{table}
\begin{center}
\begin{tabular}{|c|ccc|}
\hline
& \multicolumn{3}{|c|}{The last three entries of} \\
$(\sigma(k), \sigma(k+1), \sigma(k+2))$ & $\vect{a} \sigma$ & $\vect{b} \sigma$ & $\vect{c} \sigma$ \\
\hline
$(k, k+2, k+1)$ & $(k-1, k, k)$ & $(k-1, k-1, k)$ & $(d, k, k-1)$ \\
$(k+1, k, k+2)$ & $(k, k-1, k)$ & $(k, k-1, k-1)$ & $(k-1, d, k)$ \\
$(k+1, k+2, k)$ & $(k, k, k-1)$ & $(k, k-1, k-1)$ & $(k-1, k, d)$ \\
$(k+2, k, k+1)$ & $(k, k-1, k)$ & $(k-1, k-1, k)$ & $(k, d, k-1)$ \\
$(k+2, k+1, k)$ & $(k, k, k-1)$ & $(k-1, k, k-1)$ & $(k, k-1, d)$ \\
\hline
\end{tabular}
\end{center}
\caption{The last three entries of $\vect{a} \sigma$, $\vect{b} \sigma$, and $\vect{c} \sigma$.}
\label{table:last3}
\end{table}
Since $d$ equals either $\sigma(\sigma^{-1}(k - 1) - 1)$ or $\sigma(\sigma^{-1}(k - 1) + 1)$, we can read off from \eqref{eq:firstk-1} and Table~\ref{table:last3} that
\begin{align*}
& \{\ofo(\vect{a} \sigma), \ofo(\vect{b} \sigma), \ofo(\vect{c} \sigma)\} \\
& \quad {} = \{ (\sigma(1), \dots, \sigma(\sigma^{-1}(k - 1) - 1), \sigma(\sigma^{-1}(k - 1) + 1), \dots, \sigma(k - 1), k - 1, k), \\
& \quad \phantom{{} = \{} (\sigma(1), \dots, \sigma(\sigma^{-1}(k - 1) - 1), \sigma(\sigma^{-1}(k - 1) + 1), \dots, \sigma(k - 1), k, k - 1) \} \\
& \quad {} = \{(\sigma(1), \dots, \sigma(k-1), k) \zeta, \: (\sigma(1), \dots, \sigma(k-1), k) \xi\} \\
& \quad {} = \{ \vect{k} \rho \zeta, \vect{k} \rho \xi \},
\end{align*}
where $\zeta$ and $\xi$ are the cycles
\begin{align*}
\zeta &= (\sigma^{-1}(k-1) \quad \sigma^{-1}(k-1) + 1 \quad \cdots \quad k - 1), \\
\xi &= (\sigma^{-1}(k-1) \quad \sigma^{-1}(k-1) + 1 \quad \cdots \quad k - 1 \quad k).
\end{align*}
Since the cycles $\zeta$ and $\xi$ have lengths that differ by one, $\zeta$ and $\xi$ have opposite parities; therefore $\rho \zeta$ and $\rho \xi$ have opposite parities. Moreover, if $k \equiv 0 \pmod{4}$, then $\sigma^{-1}(k-1) \neq 1$; hence each one of $\rho$, $\zeta$, and $\xi$ fixes $1$, and so do $\rho \zeta$ and $\rho \xi$. Thus, $\{g^*(\vect{k} \rho \zeta), g^*(\vect{k} \rho \xi)\} = \{\alpha, \beta\}$. We conclude that $f(\vect{a}) \neq g(\vect{a} \sigma)$ or $f(\vect{b}) \neq g(\vect{b} \sigma)$ or $f(\vect{c}) \neq g(\vect{c} \sigma)$.

\textit{Case~3:} $\{\sigma(k), \sigma(k+1), \sigma(k+2)\} \neq \{k, k+1, k+2\}$. Let $r, s, t \in \nset{k+2}$ be the unique elements satisfying $\{\sigma(r), \sigma(s), \sigma(t)\} = \{k, k+1, k+2\}$ and $r < s < t$. It also holds that $r < k$. Then the entries of $\vect{u} \sigma$ that are equal to $k$ occur exactly at the $r$-th, $s$-th, and $t$-th positions. For $i \in \nset{k+2} \setminus \{r, s, t\}$, the $i$-th entry of $\vect{u} \sigma$ equals $\sigma(i)$.

Next we define a tuple $\vect{a} \in A^{k+2}$. The definition depends on the values of $r$, $s$, and $t$. Table~\ref{table:help} illustrates the various possibilities.
\begin{itemize}
\item If $t < k + 2$, then let $d = \sigma(t + 1)$, and let $\vect{a} \in A^{k+2}$ be the tuple obtained from $\vect{u}$ by changing the entries at the $\sigma(r)$-th and $\sigma(s)$-th positions to $d$.

\item If $t = k + 2$ and $s < k + 1$, then let $d = \sigma(s + 1)$, and let $\vect{a} \in A^{k+2}$ be the tuple obtained from $\vect{u}$ by changing the entry at the $\sigma(r)$-th position to $d$.

\item If $t = k + 2$ and $s = k + 1$, then let $d = \sigma(k)$ (note that in this case $r < k$), and let $\vect{a} \in A^{k+2}$ be the tuple obtained from $\vect{u}$ by changing the entry at the $\sigma(r)$-th position to $d$.
\end{itemize}
\begin{table}
\begin{center}
\begin{tabular}{|c|cc|}
\hline
 & $\vect{u} \sigma$ & $\ofo(\vect{u} \sigma)$ \\
 & $\vect{a} \sigma$ & $\ofo(\vect{a} \sigma)$ \\
\hline
$t < k + 2$ & $(\dots, \pos{r}{k}, \dots, \pos{s}{k}, \dots, \pos{t}{k}, \pos{t+1}{d}, \dots)$ & $(\dots, \pos{r}{k}, \dots, \pos{t-1}{d}, \dots)$ \\
            & $(\dots, \pos{r}{d}, \dots, \pos{s}{d}, \dots, \pos{t}{k}, \pos{t+1}{d}, \dots)$ & $(\dots, \pos{r}{d}, \dots, \pos{t-1}{k}, \dots)$ \\
\hline
$t = k + 2$ & $(\dots, \pos{r}{k}, \dots, \pos{s}{k}, \pos{s+1}{d}, \dots, \pos{k+2}{k})$ & $(\dots, \pos{r}{k}, \dots, \pos{s}{d}, \dots)$ \\
$s < k + 1$ & $(\dots, \pos{r}{d}, \dots, \pos{s}{k}, \pos{s+1}{d}, \dots, \pos{k+2}{k})$ & $(\dots, \pos{r}{d}, \dots, \pos{s}{k}, \dots)$ \\
\hline
$t = k + 2$ & $(\dots, \pos{r}{k}, \dots, \pos{k}{d}, \pos{k+1}{k}, \pos{\;\;\;\; k+2}{k})$ & $(\dots, \pos{r}{k}, \dots, \pos{k}{d})$ \\
$s = k + 1$ & $(\dots, \pos{r}{d}, \dots, \pos{k}{d}, \pos{k+1}{k}, \pos{\;\;\;\; k+2}{k})$ & $(\dots, \pos{r}{d}, \dots, \pos{k}{k})$ \\
\hline
\end{tabular}
\end{center}
\caption{Illustration for the different possibilities in the definition of $\vect{a}$ in Case~3 of the proof of Proposition~\ref{prop:FnequivG}.}
\label{table:help}
\end{table}

With the help of Table~\ref{table:help}, it is easy to verify that $\ofo(\vect{a}) = \vect{k}$ and $\ofo(\vect{a} \sigma) = \vect{k} \tau \rho$, where $\tau$ is the transposition of $d$ and $k$. Since both $\tau$ and $\rho$ are odd permutations, $\tau \rho$ is an even permutation. Moreover, if $k \equiv 0 \pmod{4}$, then $r > 1$; hence both $\tau$ and $\rho$ fix $1$ and so does $\tau \rho$. Thus $f(\vect{a}) = \phi^+_k(\vect{k}) = \alpha \neq \beta = \psi^+_k(\vect{k} \tau \rho) = g(\vect{a} \sigma)$.

The three cases analysed above exhaust all possibilities. We found for every permutation $\sigma \in \symm{k+2}$ a tuple $\vect{a} \in A^{k+2}$ satisfying $f(\vect{a}) \neq g(\vect{a} \sigma)$. We conclude that $f \not\equiv g$.
\end{proof}

The results of this section inevitably evoke the following question.

\begin{question}
\label{que:ofo-reconstructible}
Let $n$ and $k$ be positive integers such that $k \equiv 1, 2 \pmod{4}$ and $n \geq k + 2$, or $k \equiv 0, 3 \pmod{4}$ and $n \geq k + 3$. Assume that $\card{A} = k$. Is every function $f \colon A^n \to B$ that is determined by the order of first occurrence reconstructible?
\end{question}

By Corollary~\ref{cor:Booleanwofo-reconstructible}, the answer to Question~\ref{que:ofo-reconstructible} is positive when $k = 2$. For $k > 2$, this question remains an open problem.

%%%%%%%%%%%%%%%%%%%%%%%%%%%%%%%%%%%%%%%%%%%%%%%%%%

\section*{Acknowledgments}

The author would like to thank Miguel Couceiro, Karsten Sch\"olzel, and Tam\'as Waldhauser for inspiring discussions on minors of functions, reconstruction problems, and permutations.

The use of the Magma computer algebra system \cite{Magma} led the author to discover Theorem~\ref{thm:eqeq}.

%%%%%%%%%%%%%%%%%%%%%%%%%%%%%%%%%%%%%%%%%%%%%%%%%%

\end{document}